\documentclass[a4paper,11pt]{amsart}
\usepackage[left=2.7cm,right=2.7cm,top=3.5cm,bottom=3cm]{geometry}
\usepackage{amsthm,amssymb,amsmath,amsfonts,mathrsfs,amscd,dsfont,verbatim}
\usepackage[utf8]{inputenc}
\usepackage[T1]{fontenc}
\usepackage[all,cmtip]{xy}
\usepackage{latexsym}
\usepackage{longtable}
\usepackage{color}
\usepackage{mathtools}
\usepackage{graphicx} 

\numberwithin{equation}{section}
\newfont{\cyr}{wncyr10 scaled 1100}
\newfont{\cyrr}{wncyr9 scaled 1000}

\usepackage[bookmarks,colorlinks=true,pagebackref,final,breaklinks=true,pdfsubject={LaTeX}]{hyperref}
\hypersetup{
        colorlinks=true
        bookmarksnumbered=true,
        linkcolor=bwblue,
        citecolor=bwmagenta,
}

\definecolor{bwmagenta}{rgb}{0.0,0.45,0.6}

\definecolor{bwblue}{rgb}{0.4,0.1,0.2}

\theoremstyle{theorem}

\newtheorem*{theoremA}{Theorem A}
\newtheorem*{theoremB}{Theorem B}
\newtheorem*{theoremC}{Theorem C}

\newtheorem{theo}{Theorem}[section]
\newtheorem{cor}[theo]{Corollary}

\newtheorem{lemma}[theo]{Lemma}
\newtheorem{proposition}[theo]{Proposition}
\theoremstyle{remark}
\newtheorem{remark}[theo]{Remark}

\theoremstyle{definition}
\newtheorem{definition}[theo]{Definition}
\newtheorem{hy}[theo]{Hypothesis}

\newcommand{\quo}[1]{ \mathbf{Z}/p^{n}\mathbf{Z}  }

\newcommand{\iw}{\Lambda}

\newcommand{\gaun}{  \mathfrak{G}    }

\newcommand{\fre}[1]{\stackrel{#1}{\rightarrow}}

\usepackage[OT2,T1]{fontenc}
\DeclareSymbolFont{cyrletters}{OT2}{wncyr}{m}{n}
\DeclareMathSymbol{\sha}{\mathalpha}{cyrletters}{"58}
\newcommand{\inlim}{\mathop{\varprojlim}\limits}

\newcommand{\dia}[1]{\left<{}#1\right>}

\newcommand{\lri}[1]{\left(#1\right)}

\newcommand{\divp}{\mathbf{Q}_p/\mathbf{Z}_p}
\newcommand{\Hom}[1]{\mathrm{Hom}_{#1}}
\newcommand{\ord}{\mathrm{ord}}

\newcommand{\lfre}[1]{\stackrel{#1}{\longrightarrow}}

\newcommand{\M}{\texttt{M}}

\newcommand{\Z}{\mathbf{Z}}
\newcommand{\Q}{\mathbf{Q}}

\newcommand{\N}{\mathbf{N}}

\newcommand{\F}{\mathbf{F}}

\newcommand{\les}{\leqslant}
\newcommand{\ges}{\geqslant}

\newcommand{\C}{\mathbf{C}}

\newcommand{\hlfre}[1]{\stackrel{#1}{\lhook\joinrel\relbar\joinrel\rightarrow}}
\DeclareRobustCommand\longtwoheadrightarrow
     {\relbar\joinrel\twoheadrightarrow}
\newcommand{\tlfre}[1]{\stackrel{#1}{\longtwoheadrightarrow}}

\newcommand{\lra}{{\longrightarrow }}

\author{Massimo Bertolini, Matteo Longo, and Rodolfo Venerucci}

\begin{document}

\title{The anticyclotomic main conjectures for elliptic curves}

\maketitle

\setcounter{tocdepth}{1}

\tableofcontents

\section{Introduction}

Let $E/\Q$ be a modular elliptic curve of conductor $N$ and let $f$ be the cuspidal eigenform on $\Gamma_0(N)$ associated with  $E$ by the modularity theorem. Denote by $K_\infty$ the anticyclotomic $\Z_p$-extension of an imaginary quadratic field $K$, where $p$ is a prime number. The goal of this article is to obtain a proof of the main conjectures of Iwasawa theory for $E$ over $K_\infty$, both when $p$ is \emph{good ordinary} and when $p$ is \emph{supersingular} for $E$. 

The anticyclotomic setting displays a well-known dichotomy, depending on whether the generic sign of the functional equation of the complex $L$-function of $E/K$ twisted by finite order characters of the Galois group of $K_\infty/K$ is $+1$ or $-1$.  For reasons which will be explained later  we call the former case {\em definite} and the latter case {\em indefinite}.

Assume first that $p$ is a {\em good ordinary} prime for $E$.
In the {\em indefinite} case, a norm-compatible sequence of Heegner points arising from a Shimura curve parametrisation is defined over the finite layers of $K_\infty/K$. Its
position in the  compact $p$-adic Selmer group of $E/K_\infty$ is encoded by an element $L_p(f)$ of the anticyclotomic Iwasawa algebra $\Lambda$, called the {\em indefinite anticyclotomic $p$-adic $L$-function}. 
(The notation $L_p(f)$ instead of $L_p(E)$ is adopted throughout, in order to achieve notational uniformity in the modular arguments of this article.) The {\em indefinite anticyclotomic Iwasawa main conjecture (IAMC)}, formulated by Perrin-Riou \cite{PR-Heeg}, states that $L_p(f)$ generates the square-root of the characteristic ideal of the $\Lambda$-torsion part of the Pontrjagin dual of the $p$-primary  Selmer group of $E/K_\infty$. The proof of this conjecture is one of the main results of this paper. We remark that one divisibility of characteristic ideals -- notably the fact that $L_p(f)$ is divisible by the characteristic ideal of the relevant Selmer group -- is obtained in Howard's paper \cite{How-HK}, as a direct application of the theory of Euler systems. 

We now turn to the {\em definite} (good ordinary) case, in which the {\em definite anticyclotomic $p$-adic $L$-function} $L_p(f)$ interpolates central critical values of twists of the complex $L$-function of $E/K$, described   in Section \ref{Shimura curves} in terms of special points on the Gross curve.  The level raising of the modular form $f$ at certain admissible primes yields congruent eigenforms modulo arbitrary powers of $p$. These eigenforms belong to the indefinite setting and therefore the Heegner construction becomes available on the Shimura curves supporting them. (See Section \ref{sec:admissible} for the precise definitions.) This basic observation is the opening gambit of the article \cite{Be-Da-main} by Bertolini--Darmon, which builds on it by establishing a  {\em first explicit reciprocity law} relating the resulting Heegner cohomology classes to $L_p(f)$. Moreover, with the help of a {\em second explicit reciprocity law}, this article sets up an inductive procedure (which may be viewed as an analogue of Kolyvagin's induction) proving that $L_p(f)$ is divisible by the characteristic ideal of the Pontrjagin dual of the $p$-primary  Selmer group of $E/K_\infty$. This shows one divisibility in the {\em definite Iwasawa anticyclotomic main conjecture (DAMC)}. 
(The explicit reciprocity laws are reviewed in Section \ref{sec:reciprocity}.) This procedure has been formalised in Howard's paper \cite{How-heeg}, leading to the concept of {\em bipartite Euler system}. 
The full DAMC proved in this paper is rather based on a refinement of the induction argument in \cite{Be-Da-main}. It requires to show the non-vanishing modulo $p$ of values of the definite $p$-adic $L$-function attached to an eigenform congruent to $f$, obtained by raising the level at sufficiently many admissible primes. This maximality property ultimately rests on a fundamental {\em $p$-converse} theorem of Skinner--Urban \cite{S-U}, as explained in Step 4 of Section \ref{BSD0}. It should be stressed that both the DAMC and the IAMC  are obtained in this article from the same unified approach based on the above-mentioned  inductive process. The article \cite{BCK} by Burungale--Castella--Kim uses directly the techniques of bipartite Euler systems to obtain a proof of the IAMC (that is, Perrin-Riou's Heegner point main conjecture). 

Assume now that $p$ is a {\em supersingular} prime for $E$. As customary in the supersingular theory, two cases indexed by a sign $\varepsilon=\pm $ need to be distinguished. Depending on the choice of $\varepsilon$, one is led  to introduce the concepts of $\varepsilon$-points, $\varepsilon$-Selmer groups and $\varepsilon$-$p$-adic $L$-functions $L_p^\varepsilon(f)$. In terms of these objects, this article formulates and proves the analogues of the AMCs outlined above. In the definite setting, the analogous inclusion of \cite{Be-Da-main} was obtained by Darmon--Iovita \cite{Da-Io} when $p$ is split in $K$ and $a_p(E)=0$, and extended by Burungale--B\"uy\"ukboduk--Lei \cite{BBL} without assuming $a_p(E)=0$ and covering also the case $p$ inert in $K$. 

The following two specific aspects of the supersingular setting are worth noting. 

On the one hand, our study of the structure of the $\varepsilon$-Selmer groups rests in a fundamental way on the {\em control} result stated in Proposition \ref{hypext}. The proof of this result is based on Theorem \ref{prop:formal-group}, which was known for $p$ split in $K$ thanks to the work of Iovita--Pollack \cite{Io-Po}. For $p$ inert in $K$, Theorem \ref{prop:formal-group} is a consequence of the recent proof of Rubin's conjecture on local points in $p$-adic towers due to Burungale--Kobayashi--Ota  \cite{BKO1}.
In a previous version of this article, the control statement of  Proposition \ref{hypext} was a running assumption in the inert case.

When $p$ is inert in $K$, the supersingular setting displays a subcase for $\varepsilon=+$, called {\em exceptional} in Definition \ref{def:exceptional} of this Introduction. In the exceptional case, the $+$-$p$-adic $L$-function acquires an extra-zero of local nature and our approach only allows us to show one divisibility in the AMCs. It would be interesting to further investigate this exceptional zero and the possibility of establishing the full AMC in the exceptional case.

We now formulate our main results more precisely. In order to obtain unified statements, we adopt the convention that $\varepsilon=\emptyset$ in the ordinary case, so that the concept of $\varepsilon$-point, $\varepsilon$-Selmer group and $\varepsilon$-$p$-adic $L$-function simply stands for point, Selmer group and $p$-adic $L$-function (then in particular $L_p(f)=L_p^\varepsilon(f)$ is this case). 

%\vskip 5cm
Fix throughout the paper algebraic closures $\bar\Q$ of $\Q$ and $\bar\Q_p$ of $\Q_p$, 
as well as embeddings $\bar\Q\hookrightarrow\bar\Q_p$ and $\bar\Q\hookrightarrow\C$.  

Our main results are proved under the following assumptions. Let $N$ be as above the conductor of $E$,  assumed to be coprime with the discriminant of $K$. Factor $N$ as $N=N^+N^-$, where $N^+$ resp.\ $N^-$ is divisible only by primes which are split, resp.\ inert in $K$. 
\begin{hy}\label{thehy} 
\hfill
\begin{enumerate}
\item The rational prime $p$ is $\geqslant 5$ and does not divide $N$.
\item The rational prime $p$ does not divide the class number $h_K$ and the discriminant of $K$.
\item The representation $\bar{\varrho}_{E,p} : G_{\Q}\fre{}\mathrm{GL}_{2}(\F_{p})$
arising from the $p$-torsion $E(\bar{\Q})_{p}$ of $E$ is surjective. 
\item $N^-$ is squarefree.
\item If $E$ has good ordinary reduction at $p$, then $a_{p}(E)\not\equiv{}\pm1\pmod{p}$ if $p$ is inert in $K$, and $a_{p}(E)\not\equiv{}1\pmod{p}$ if $p$ splits in $K$.
\item If $q$ is a prime dividing $N^{+}$, then $H^{0}(G_{\Q_q},E_{p})=0$ and $\bar{\varrho}_{E,p}$ is ramified at $q$.
\item If $q\Vert{}N^{-}$ and $q\equiv{}\pm{}1\pmod{p}$, then $\bar{\varrho}_{E,p}$
is ramified at $q$.
\end{enumerate}
\end{hy}

\begin{remark}
The set of conditions in Hypothesis \ref{thehy} are enough to state and prove our main results listed below. Since they are not simultaneously required at all stages, 
we indicate at the beginning of each section  the assumptions under which  the results obtained therein hold.  
In order to simplify our analysis, we do not consider in this paper primes $p$ of multiplicative reduction for $E$. The assumption $p\nmid h_K$, i.e.\ $K_\infty/K$ is totally ramified at $p$,  is made only to ease notations in the formulas for the compatibility of special points under the trace operators. As customary in Iwasawa theory, $p$ is assumed to be odd. Moreover, in order to quote directly the literature at various points, notably in Section \ref{sec:admissible}, we also require that $p>3$. The surjectivity of the mod $p$ representation $\bar{\varrho}_{E,p} $, as well as assumption (7), is used for the level raising results of Section \ref{sec:admissible}. The non-anomalous assumption (5) is needed for the control results of Section \ref{secSel} and for the special value formula of Lemma \ref{nonexcord}. Assumption (6) simplifies the treatment of the Selmer conditions at the primes dividing $N^+$. The assumption that $p$ is either split or inert in $K$ allows us to quote the existing literature on the structure of local points over the anticyclotomic tower. Finally, assumption (4), i.e.\ $N^-$ is squarefree, simplifies considerably the moduli description of the Shimura curves and their special points. 
\end{remark} 

\begin{definition}\label{def:exceptional}
We say that $(E,K,p,\varepsilon)$ is \emph{exceptional} if $E$ has supersingular reduction at $p$, $p$ is inert in $K$ and $\varepsilon=+$. 
\end{definition}

For $\varepsilon=\pm,\emptyset$, let $L_p^\varepsilon(f)$ be the anticyclotomic $p$-adic $L$-function introduced in Chapter \ref{p-adic L} in the definite case and in Section \ref{def:indefinite_p-adic_L} in the indefinite case.
Moreover, let $\mathrm{Char}_p^\varepsilon(f)$ be the ``algebraic anticyclotomic $p$-adic $L$-function'' defined to be the characteristic ideal of a certain Selmer module in Section \ref{def:definite_char}, resp. \S\ref{def:indefinite_p-adic_L} in the definite, resp. indefinite case. Note that in the indefinite case, $L_p^\varepsilon(f)$ describes the position of a Heegner class in a compact Selmer group and $\mathrm{Char}_p^\varepsilon(f)$ refers to the torsion part of an Iwasawa module of rank one. 

The next theorem contains our results on the DAMC and IAMC. Although we have strived for maximal notational uniformity, the reader should keep in mind that the nature of the result in the two cases is rather different!

\begin{theoremA}[DAMC \& IAMC]
$(L_p^\varepsilon(f))\subseteq(\mathrm{Char}_p^\varepsilon(f))$ with equality in the non-exceptional case. 
\end{theoremA} 

The proof of Theorem A is obtained by compiling information from the finite layers of the anticyclotomic tower, via a standard method which will not be recalled in detail in this paper.
Specifically, it follows immediately from the $\varepsilon$-Birch and Swinnerton-Dyer (BSD) formulas of Theorem \ref{BSD0}, resp.\ of Theorem \ref{BSD-Indef}  in the definite, resp.\ indefinite case, by making use of
an argument due to Mazur--Rubin \cite[Section 5.2]{MR-KolSys} and Howard \cite[Section 2.2]{How-HK}. 

The Birch and Swinnerton-Dyer conjecture leads one to expect BSD formulas for the usual Selmer groups over the finite anticyclotomic layers.
These BSD formulas are obtained  in Chapter \ref{sec:proof_B_C} as a consequence of the   above-mentioned $\varepsilon$-BSD formulas, via a comparison between the $\varepsilon$-Selmer groups and the standard Selmer groups.
We refer the reader to Chapter \ref{sec:epsilonBSD} for the definition of the Selmer group $\mathrm{Sel}(K,A_f(\chi))$ as well as of the Shafarevich--Tate group $\sha(K,A_f(\chi))$, and to Sections \ref{sec:proof_B} and \ref{sec:proof_C} for an explanation of the constants $C$ (related to certain archimedean periods) and of the regulator $\mathrm{Reg}_\chi(E/K)$, which appear in the statements below.

\begin{theoremB}[Definite BSD formulas]Let $\chi$ be a finite 
order character of conductor $p^n$ of the Galois group of $K_\infty/K$. Then $\mathrm{Sel}(K,A_f(\chi))$ is finite if and only if  
$L(E/K,\chi,1)\neq0$. In this case one has
\[\mathrm{length}_{\mathscr{O}_\chi}\!\big(\mathrm{Sel}(K,A_f(\chi))\big)\leqslant \ord_\chi\!\left(\frac{L(E/K,\chi,1)}{C}\right)\] 
with equality in the non-exceptional case. 
\end{theoremB}

\begin{theoremC}[Indefinite BSD formulas]Let $\chi$ be a finite 
order character of conductor $p^n$  of the Galois group of $K_\infty/K$. Then $\mathrm{Sel}(K,A_f(\chi))$ has corank equal to $1$ if and only 
if $L^\prime(E/K,\chi,1)\neq 0$. In this case one has
\[\mathrm{length}_{\mathscr{O}_\chi}\!\big(\sha(K,A_f(\chi))\big)\leqslant
\ord_\chi\!\lri{\frac{L^\prime(E/K,\chi,1)}{C\cdot \mathrm{Reg}_\chi(E/K)}}\]      
with equality in the non-exceptional case.  \end{theoremC}

\begin{remark} \label{remark:n=0}
The case $n=0$ can be obtained more directly by applying \cite{S-U, FW} in the setting of Theorem B and the techniques of \cite{weiz, BBV} in the setting of Theorem C. In the non-exceptional case it follows as well from the AMCs proved in this paper. The presence of a local zero in the exceptional case
 prevents us to treat the trivial character on the same ground as the other characters. 
\end{remark}

\subsection*{Conventions} The following conventions are adopted to lighten the notation (as recalled also in the appropriate parts of the paper).
\begin{itemize} 
\item Besides denoting with $\varepsilon$ one of the signs $+$ or $-$, we also sometimes write $\varepsilon=+1$ when $\varepsilon=+$ and $\varepsilon=-1$ when $\varepsilon=-$. With this convention,  the equation 
$(-1)^n=\varepsilon$ for an integer $n$ implies that $n$ is even if $\varepsilon=+$ and $n$ is odd if $\varepsilon=-$. 
\item  Given a principal ideal $(x)$  of a commutative ring with unity $R$ and an $R$-module  $M$, we sometimes write $M/x$ to denote $M/(x)=M/xM$. 
\end{itemize}

\vskip .2cm\noindent
{\em Acknowledgements.} It is a pleasure to thank the referee for several comments, which allowed us to improve the exposition of this paper. 

\section{Special points on Shimura curves and Gross curves}\label{Shimura curves}

\subsection{Shimura curves and Gross curves} 
Fix a positive integer $N$ and a factorisation 
$N=N^{+}N^{-}$ into coprime integers, with $N^-$ squarefree. 
Let $\mathscr{B}$ be the quaternion algebra over $\Q$ whose discriminant has finite part equal to $N^{-}$.
The algebra $\mathscr{B}$ (which is unique up to isomorphism) is said to be \emph{indefinite}, (resp., \emph{definite}) if it is split (resp., non-split) at infinity. So $\mathscr{B}$ is indefinite if and only if $N^-$ is divisible by an \emph{even} number of primes.

For every abelian group $Z$, let $\hat{Z}$ denote $Z\otimes_{\Z}\hat{\Z}$,
where $\hat{\Z}=\prod_{\ell\ \text{prime}}\Z_{\ell}$ 
is the profinite  completion of $\Z$.  Let $\mathrm{Hom}(\C,\mathscr{B}_{\infty})$ 
be the set of $\mathbf{R}$-algebra morphisms  
of  $\C$ in $\mathscr{B}_{\infty}=\mathscr{B}\otimes_\Q\mathbf{R}$. 
The group $\mathscr{B}^{\ast}$ acts via the diagonal embedding on $\hat{\mathscr{B}}^{\ast}$, and via conjugation on 
$\mathrm{Hom}(\C,\mathscr{B}_{\infty})$.  
Fix a  maximal order
$\tilde{\mathscr{R}}$ in $\mathscr{B}$, and an Eichler order $\mathscr{R}$ of level $N^{+}$ contained in 
$\tilde{\mathscr{R}}$. 
Define the set 
\begin{equation}
\label{eqn:shimura_curve}
                 Y_{N^{+},N^{-}}(\C):=
                 \hat{\mathscr{R}}^{\ast}\big\backslash{}\hat{\mathscr{B}}^{\ast}\times{}\mathrm{Hom}(\C,\mathscr{B}_{\infty})
                 \big/\mathscr{B}^{\ast}. 
\end{equation}
As notation suggests, $Y_{N^{+},N^{-}}(\C)$ is a Riemann surface,
arising as the set of complex points of a smooth curve. This curve can be defined over $\Q$, and its description, which we will recall in the next paragraphs, 
markedly depends on whether $\mathscr{B}$
is definite or indefinite.

In the indefinite case, let 
$\Gamma_{N^{+},N^{-}}\subset{}\mathrm{SL}_{2}(\mathbf{R})$ be the discrete subgroup of $\iota_{\infty}\lri{\mathscr{R}^{\ast}}$
consisting of elements of determinant $1$; here $\iota_\infty:\mathscr{B}\cong\M_2(\mathbb{R})$ is a fixed isomorphism. Then the strong approximation theorem 
shows that
\[Y_{N^{+},N^{-}}(\C)\cong{}\Gamma_{N^{+},N^{-}}\big\backslash\mathcal{H},\] 
where $\mathcal{H}:=\{z\in{}\C : \Im(z)>0\}$, and the left 
action of $\Gamma_{N^+,N^-}$ on $\mathcal{H}$ 
is by fractional linear transformations. If $N^-\neq 1$, we set 
$X_{N^+,N^-}(\C)=Y_{N^+,N^-}(\C)$, while if $N^-=1$ then 
$Y_{N^+,N^-}(\C)$ is the usual modular curve of level $\Gamma_0(N)$, 
and we let $X_{N^+,N^-}(\C)$ denote its standard 
compactification obtained by adding a finite set of cusps. The Riemann 
surface $X_{N^+,N^-}(\C)$ has a model $X_{N^+,N^-}$ defined 
over $\Q$, which is called the \emph{Shimura curve} of discriminant $N^-$ 
and level $N^+$ (up to isomorphism, it is independent of the choices made).
 
In the definite case, the double coset  space 
$\hat{\mathscr{R}}^{\ast}\big\backslash{}\hat{\mathscr{B}}^{\ast}\big/\mathscr{B}^{\ast}$ 
is a finite set, in bijection with the set $\{\mathscr{R}_{1},\dots,\mathscr{R}_{h}\}$ of conjugacy classes of (oriented) Eichler orders of level $N^{+}$ in $\mathscr{B}$.
For every $j=1,\dots,h$, set 
$\Gamma_{j}:=\mathscr{R}_{j}^{\ast}/\Z^{\ast}$; each $\Gamma_j$ 
is a finite group. Then, again by the strong approximation theorem,  
\[Y_{N^{+},N^{-}}(\C)\cong{}\coprod_{j=1}^{h} \Gamma_{j}\backslash\mathrm{Hom}(\C,\mathscr{B}_\infty).\] Attach a conic $\mathscr{C}/\Q$ to $\mathscr{B}$, by the rule
\[
                     \mathscr{C}(A):=\big\{x\in{}\mathscr{B}\otimes_{\Q}A : x\not=0, \mathrm{Nr}(x)=\mathrm{Tr}(x)=0 \big\}\big/A^\ast,
\]
where $\mathrm{Nr}$ and $\mathrm{Tr}$ denote reduced norm and trace, respectively.
There is a natural bijection between 
$\mathrm{Hom}(\C,\mathscr{B}_{\infty})$ and $\mathscr{C}(\C)$, 
from which it follows that $Y_{N^{+},N^{-}}(\C)$ is identified with the set of complex points of the disjoint union 
$
X_{N^{+},N^{-}}:=\coprod_{j=1}^{h}\mathscr{C}_{j}
$
of the genus zero curves $\mathscr{C}_{j}:=\Gamma_{j}\backslash\mathscr{C}$ defined over $\Q$. The curve $X_{N^+,N^-}$ is called the \emph{Gross curve} of discriminant 
$N^-$ and level $N^+$.

\subsection{Hecke operators}\label{heckeoper} Since $\mathrm{Pic}(\Z)\cong{}\hat{\Q}^{\ast}/\Q^{\ast}\hat{\Z}^{\ast}$ 
is trivial, one has a bijection
\[
             Y_{N^{+},N^{-}}(\C)\cong{}\Big(\hat{\mathscr{R}}^{\ast}\big\backslash{}\hat{\mathscr{B}}^{\ast}\big/\hat{\Q}^{\ast}
             \times{}\mathrm{Hom}(\C,\mathscr{B}_{\infty})\Big)\big/\mathscr{B}^{\ast}.
\]
The double coset space $\hat{\mathscr{R}}^{\ast}\big\backslash{}\hat{\mathscr{B}}^{\ast}\big/\hat{\Q}^{\ast}$ is equal to the product over all prime numbers $\ell$ of the local double coset spaces 
$\mathcal{T}_\ell=\mathscr{R}_\ell^*\backslash\mathscr{B}_\ell^*/\Q_\ell^*$, 
where $\mathscr{R}_\ell=\mathscr{R}\otimes_\Z\Z_\ell$ and $\mathscr{B}_\ell=\mathscr{B}\otimes_\Q\Q_\ell$.  
If $\ell\nmid{}N$, then $\mathcal{T}_{\ell}$ is isomorphic to the set of vertices of the Bruhat-Tits tree 
$\mathcal{T}_{\ell}
                \cong{}\mathrm{PGL}_{2}(\Z_{\ell})\backslash\mathrm{PGL}_{2}(\Q_{\ell})$ of $\mathrm{PGL}_{2}(\Q_{\ell})$. 
This decomposition gives rise to an action of
Hecke operators $T_{\ell}$, for primes $\ell\nmid N$,
and $U_{\ell}$ for $\ell\mid N$ by 
$\Q$-rational correspondences on $X_{N^{+},N^{-}}$. By covariant functoriality, they 
induce endomorphisms of the Picard group 
\[J_{N^{+},N^{-}}=\mathrm{Pic}(X_{N^{+},N^{-}}/\Q)\]
of the curve $X_{N^{+},N^{-}}/\Q$, denoted in the same way. Define  
$\mathbf{T}_{N^{+},N^{-}}$ to be the $\Z$-subalgebra of the ring 
$\mathrm{End}_{\Q}(J_{N^{+},N^{-}})$ 
generated over $\Z$ by the operators $T_{\ell}$ and $U_{\ell}$.
Note that in the definite case $J_{N^+,N^-}$ is 
a free $\Z$-module of rank equal to the number of connected components of 
the Gross curve $X_{N^+,N^-}$.  

\subsection{The Jacquet--Langlands correspondence}\label{JLsec} Let
 $\mathbb{T}_{N^{+},N^{-}}$ be the Hecke algebra acting faithfully on the $\C$-vector space $S_{2}(\Gamma_{0}(N))^{N^{-}\text{-new}}$ of weight-two cusp forms of level $N$ 
which are new at $N^{-}$, generated over $\Z$ by Hecke operators $T_{\ell}$ for primes $\ell\nmid{}N$ and $U_{\ell}$
for primes $\ell|N$. The Jacquet--Langlands correspondence states the existence of a canonical isomorphism
$\mathbb{T}_{N^{+},N^{-}}\cong{}\mathbf{T}_{N^{+},N^{-}}$ 
identifying Hecke operators indexed by the same prime numbers. 
It follows that $\mathbb{T}_{N^{+},N^{-}}$ acts as a group of $\Q$-rational endomorphisms of $J_{N^{+},N^{-}}$. See Section 1.6 of \cite{Be-Da1} for details. 
 
\subsection{Special points}\label{special points} Let $p>3$ be a prime number 
such that $p\nmid N$, 
and $K/\Q$ be an imaginary quadratic field of discriminant $D_K$ coprime with $Np$. Assume in this subsection that the factorization $N=N^+N^-$ satisfies the following 
\emph{generalized Heegner hypothesis}: 
a prime divisor $q$ of $N$ divides $N^+$ if and only 
if it is split in $K$. 

The  inclusion
$\mathrm{Hom}(K,\mathscr{B})\subset \mathrm{Hom}(\C,\mathscr{B}_\infty)$
arising from extension of scalars induces a map from the set 
\[
                         \mathscr{S}_{N^+,N^-}(K):=
                         \hat{\mathscr{R}}^{\ast}\big\backslash{}\hat{\mathscr{B}}^{\ast}\times{}\mathrm{Hom}(K,\mathscr{B})\big/\mathscr{B}^{\ast}
\]
to $Y_{N^{+},N^{-}}(\C)$. 
A \emph{special point} of $X_{N^{+},N^{-}}$ associated with $K$ is any point 
in the image of this map. 
When $\mathscr{B}$ is indefinite (resp., definite), so that $X_{N^+,N^-}$ is a Shimura curve (resp., a Gross curve),
we say that the points in $\mathscr{S}_{N^+,N^-}(K)$ are  \emph{Heegner points} 
(resp., \emph{Gross points}) 
associated with $K$.

Let $P\in{}\mathscr{S}_{N^+,N^-}(K)$ be represented by $g\times{}f\in{}\hat{\mathscr{B}}^{\ast}\times{}\mathrm{Hom}(K,\mathscr{B})$.
Then $P$ is a said to be of \emph{conductor $p^{n}$} if 
\[
                       f(K)\cap{}g^{-1}\hat{\mathscr{R}}^{\ast}g=f(\mathcal{O}_{p^{n}}),
\]
where $\mathcal{O}_{p^{n}}:=\Z+p^n\mathcal{O}_K$ ($n\ges 0$) is the order of $K$ of conductor $p^n$.
Write $\mathscr{S}_{N^+,N^-}(\mathcal{O}_{p^n})$ 
for the set of special points of conductor $p^{n}$ in $X_{N^{+},N^{-}}(\C)$.
The theory of local embeddings guarantees that,   
under the condition recalled at the beginning of this subsection, 
the set $\mathscr{S}_{N^+,N^-}(\mathcal{O}_{p^n})$ is not empty for all $n\ges 0$
(see \cite{Be-Da1}, Section 2.2). 

The set of special points $\mathscr{S}_{N^+,N^-}(K)$ is equipped with an algebraic 
Galois action of the group $\mathrm{Gal}(K^{\mathrm{ab}}/K)$, where $K^\mathrm{ab}$ is the maximal abelian extension of $K$. 
Let $P\in\mathscr{S}_{N^+,N^-}(K)$ be represented by a pair 
$g\times{}f\in{}\hat{\mathscr{B}}^{\ast}\times{}\mathrm{Hom}(K,\mathscr{B})$ 
and let $\sigma$ be represented under the inverse of the Artin map by the class of 
an element $\mathfrak{a}\in \hat{K}^*$. Thadelisationen 
$\sigma(P)$ is the
special point in $\mathscr{S}_{N^+,N^-}(K)$ 
represented by the pair $g\hat{f}(\mathfrak{a})\times f$, 
where $\hat{f}$ is the adelisation of $f$.   

Let $\mathrm{Pic}(\mathcal{O}_{p^{n}})=K^{\ast}\backslash{}\hat{K}^{\ast}/\hat{\mathcal{O}}_{p^{n}}$
be the Picard group of $\mathcal{O}_{p^{n}}$. By class field theory there exists an
abelian extension $\tilde{K}_{n}/K$, the ring class field of conductor $p^{n}$,
such that the Galois group 
$\tilde{G}_{n}=\mathrm{Gal}(\tilde{K}_{n}/K)$ is isomorphic to 
$\mathrm{Pic}(\mathcal{O}_{p^{n}})$ via the inverse of the Artin map.
Recall that the Galois group $\mathrm{Gal}(K/\Q)$ acts on $\tilde{G}_n$ as inversion.  

If $X_{N^+,N^-}$ is a Shimura curve, then the theory of complex multiplication shows 
that $\mathscr{S}_{N^+,N^-}(\mathcal{O}_{p^n})$ is contained in 
$X_{N^+,N^-}(\tilde{K}_n)$, for all $n\ges0$, and Shimura's reciprocity law states 
that the algebraic Galois action on the set of special points $\mathscr{S}_{N^+,N^-}(K)$ described 
above coincides with
the usual geometric action of $\mathrm{Gal}(\bar\Q/\Q)$ on $X_{N^+,N^-}(\bar\Q)$.
In this case, for any extension $H/\Q$ in $\bar\Q$,
denote as usual $J_{N^+,N^-}(H)$ 
the subgroup of $H$-rational divisors of $J_{N^+,N^-}(\bar\Q)$, i.e. 
those fixed by $\mathrm{Gal}(\bar\Q/H)$. 

If  $X_{N^+,N^-}$ is a Gross curve, then 
the algebraic action of $\mathrm{Gal}(K^\mathrm{ab}/K)$ 
on $\mathscr{S}_{N^+,N^-}(K)$ described above 
does not correspond to any geometric Galois action, since all 
special points are already  defined over $K$. 
However, one can check that each element in 
$\mathscr{S}_{N^+,N^-}(\mathcal{O}_{p^n})$ is fixed by the algebraic 
action of $\mathrm{Gal}(K^\mathrm{ab}/\tilde{K}_n)$, for all $n\ges0$. 
Extend canonically 
the action of $\mathrm{Gal}(K^\mathrm{ab}/K)$ on $\mathscr{S}_{N^+,N^-}(K)$ 
defined above to the subgroup of $J_{N^+,N^-}$ generated by the image of 
$\mathscr{S}_{N^+,N^-}(K)$. 
Given an abelian extension $H$ of $K$ and an element $D\subseteq J_{N^+,N^-}$ supported on $\mathscr{S}_{N^+,N^-}(K)$, write with an abuse of notation
$D\in  J_{N^+,N^-}(H)$ to mean that $D$ is fixed by the action of $\mathrm{Gal}(K^\mathrm{ab}/H)$. 

\subsection{Compatible sequences of special points}\label{compheeg}
Let $K$, $N=N^+N^-$ and $p\nmid N$ be fixed as in the previous subsection, 
and assume that $p\nmid h_K$, the class number of $K$. 
Recall from the Introduction the anticyclotomic $\Z_p$-extension $K_\infty/K$,
and for any integer $n\ges0$ 
let $K_n$ be the subfield of $K_\infty$ such that $G_n=\mathrm{Gal}(K_n/K)\cong\Z/p^n\Z$. 
Since $p\nmid{}h_{K}$, we have 
$\tilde{K}_{n+1}=K_{n}\cdot{}\tilde{K}_{1}$,
and $K_{n}\cap{}\tilde{K}_{1}=K$. 
In particular $\tilde{G}_{n+1}=\Delta\times{}G_{n}$, with $\Delta=\tilde{G}_{1}$. By convention, set $K_{-1}=K$. 

Let $L\ges 1$ be a squarefree integer, prime to $Np$; when $L>1$, we suppose that $L$ 
is the product of primes which are inert in $K$. 
The set $\mathscr{S}_{N^+,LN^-}(\mathcal{O}_{p^{n}})$
of special points of conductor $p^{n}$ in $X_{N^+,LN^-}(\C)$ 
is then non-empty for every $n\ges{}0$ (note that $X_{N^+,LN^-}$ might 
be a Gross curve or a Shimura curve, accordingly with the parity of the 
number of prime divisors of $N^-L$). 
As in Section $2.4$
of \cite{Be-Da1} fix a \emph{compatible sequence}
$\tilde{P}_{\infty}(L)=(\tilde{P}_{n}(L))_{n\ges{}0}$ of special points of $p$-power conductor, where  
$\tilde{P}_{n}(L)\in{}\mathscr{S}_{N^+,LN^-}(\mathcal{O}_{p^{n}})$.
For every integer $n\ges{}-1$ define
\[
                 P_{n}(L)=\sum_{\sigma\in{}\Delta}
                 \sigma\big(\tilde{P}_{n+1}(L)\big)\in{}J_{N^{+},LN^{-}}(K_{n}),\]
\[                 P_{K}(L)=\sum_{\sigma\in{}\mathrm{Pic}(\mathcal{O}_{K})}\sigma\big(\tilde{P}_{0}(L)\big)
                 \in{}J_{N^{+},N^{-}}(K).
\]  
Let $\epsilon_K$ be the quadratic character 
associated with $K$, and
$u_{K}$ be one half of the order of the unit group $\mathcal{O}_{K}^{\ast}$. 
Define 
\begin{equation}\label{u_p}
u_p=(p-\epsilon_K(p))/u_K.\end{equation} 
Then these points  satisfy the following relations:  
\begin{equation}\label{P-1}
P_{-1}(L)=u_{p}\cdot{}P_{K}(L),\end{equation}
\begin{equation} \label{P-0} 
u_{K}\cdot{}P_{0}(L)=
\begin{cases}
 T_pP_{K}(L) \qquad\qquad\quad \text{ if } \epsilon_{K}(p)=-1, \\
 T_{p}P_{K}(L)-2P_{K}(L) \ \!\, \text{ if }\epsilon_{K}(p)=+1,
\end{cases}
\end{equation}
\begin{equation}\label{eq:compdef} 
\mathrm{Trace}_{K_{n+1}/K_{n}}(P_{n+1}(L))=T_{p}P_{n}(L)-P_{n-1}(L), \text{ for every $n\ges{}0$},
\end{equation}
where $\mathrm{Trace}_{{K}_{n+1}/{K}_{n}}=\sum_{\sigma\in{}
\mathrm{Gal}({K}_{n+1}/{K}_{n})}\sigma$. If $L=1$ we simply write $P_n=P_n(L)$ and $P_K=P_K(L)$. 

\section{Admissible primes and raising the level}
\label{sec:admissible}

Fix a positive integer  $N$, a factorisation $N=N^{+}N^{-}$ into coprime positive integers, and a rational prime $p>5$ coprime to $N$. Assume that $N^-$ is square-free.

\subsection{Eigenforms of level $(N^{+},N^{-})$}\label{eigenforms}
Recall the Hecke algebra $\mathbb{T}_{N^+,N^-}$ defined in 
Section \ref{JLsec} 
and let $R$ be a complete local Noetherian ring with finite residue field $k_{R}$
of characteristic $p$. A \emph{$R$-valued (weight two) eigenform of level $(N^{+},N^{-})$}
is a \emph{surjective} morphism 
$f : \mathbb{T}_{N^{+},N^{-}}\rightarrow{}R$. 
Denote by  $S_{2}(N^{+},N^{-};R)$ the set of such eigenforms. 
To every eigenform $f\in S_{2}(N^{+},N^{-};R)$ is associated a Galois representation 
\[\bar{\rho}_{f} : G_{\Q}\longrightarrow{}\mathrm{GL}_{2}(k_{R}),\] 
unramified at every prime $q\nmid{}Np$, and  such that an arithmetic  Frobenius $\mathrm{Frob}_{q}\in{}G_{\Q}$
at $q$ has characteristic polynomial $\mathrm{char}\lri{\bar{\rho}_{f}(\mathrm{Frob}_{q})}=X^{2}-\bar{f}(T_{q})X+q\in{}k_{R}[X]$, where 
$\bar{f}(T_q)$ denotes the reduction of $f(T_q)$ modulo the maximal ideal of $R$. 
The semi-simplification of $\bar{\rho}_{f}$ is characterised by these properties. 
Moreover, as proved by Carayol \cite{Car},
if $\bar{\rho}_{f}$ is \emph{irreducible} (hence absolutely irreducible since $p$ is odd), it can be lifted \emph{uniquely} to a Galois representation
\[\rho_{f} : G_{\Q}\longrightarrow{}\mathrm{GL}_{2}(R),\] 
unramified at $q\nmid{}Np$, and such that $\mathrm{trace}(\mathrm{Frob}_{q})=f(T_{q})$
and $\det(\mathrm{Frob}_{q})=q$
for such a $q$. 
Assuming that $\bar{\rho}_{f}$ is  irreducible,  write 
$T_{f}$ for an $R$-module giving rise to the representation $\rho_{f}$, 
and, for $R$ a quotient of $\Z_p$, define 
$A_f=\mathrm{Hom}_{\Z_p}(T_f,\mu_{p^\infty})$, 
where $\mu_{p^\infty}$ is the group of $p$-power roots of unity.   

Let $n\in \N\cup\{\infty\}$ and define $R=\Z_p$ if $n=\infty$ and $R=\Z/p^n\Z$ if $n<\infty$. Let 
$f\in S_2(N^+,N^-;R)$.  
If $k\in\N\cup\{\infty\}$ and $1\les k \les n$, 
let $f_k=f\pmod{p^k}$ denote the reduction of $f$ modulo $p^k$, 
with the convention that $f_\infty=f$ if $n=k=\infty$.   
Let $T_{f,k}$ and $A_{f,k}$ be the 
modules introduced above for $f_k$ 
(i.e. $T_{f,k}=T_{f_k}$ and $A_{f,k}=A_{f_k}$). 
If particular, if $n=\infty$ then $T_{f,\infty}=T_f$ and $A_{f,\infty}=A_f$. 
Finally, for any $\Z_p$-algebra $\mathscr{O}$, we will write 
$T_{f,k,\mathscr{O}}=T_{f,k}\otimes_{\Z_p}\mathscr{O}$ and 
$A_{f,k,\mathscr{O}}=A_{f,k}\otimes_{\Z_p}\mathscr{O}$.

\subsection{Admissible primes}\label{admsec} 
Let $R$ denote a complete, local Noetherian ring
with finite residue field $k_R$ of characteristic $p$ and let
 $f\in{}S_{2}(N^{+},N^{-};R)$
be an $R$-valued eigenform of level $(N^{+},N^{-})$.
Fix a quadratic imaginary field $K/\Q$ of discriminant $D_{K}$ coprime with $Np$. Following \cite{Be-Da-main}, we say that
a rational prime $\ell$ is an \emph{admissible prime relative to $(f,K)$} if the following conditions are satisfied:
\begin{itemize}
\item[$A1.$] $\ell$ does not divide $Np$;
\item[$A2.$] $p$ does not divide $\ell^{2}-1$;
\item[$A3.$] $f(T_{\ell})^{2}=(\ell+1)^{2}\in{}R$;
\item[$A4.$] $\ell$ is inert in $K/\Q$.
\end{itemize}
Write $\mathscr S(f,K)$ for the set of squarefree  products of admissible primes for $(f,K)$. 

Let $n\in\N\cup\{\infty\}$, and put $R=\Z_p$ if $n=\infty$ and $R=\Z/p^n\Z$ if $n<\infty$. 
For $f\in S_2(N^+,N^-;R)$, and $k\in\N\cup\{\infty\}$ with $1\les k \les n$,
call \emph{$k$-admissible prime} 
any admissible prime relative to $(f_k,K)$, where recall that $f_k$ is the 
reduction of $f$ modulo $p^k$, with the convention that $f_\infty=f$. 
With an abuse of notation, if no confusion may 
arise, we write $\mathscr{S}_k$ for $\mathscr{S}(f_k,K)$. 
We say that 
$L\in{}\mathscr{S}_{k}$ is \emph{definite} if  $\epsilon_{K}(LN^{-})=-1$ and  \emph{indefinite} if  $\epsilon_{K}(LN^{-})=+1$, and write $\mathscr{S}_{k}^\mathrm{def}$ and 
$\mathscr{S}_{k}^\mathrm{ind}$ for the subsets of $\mathscr S_k$ consisting of definite and indefinite integers, respectively; 
clearly $\mathscr S_k= \mathscr{S}_{k}^\mathrm{def}\cup \mathscr{S}_{k}^\mathrm{ind}$. 

The following lemma is proved by the same argument appearing in the proof of Lemma 2.6 of \cite{Be-Da-main}.
\begin{lemma}\label{decadm} Let $\ell$ be an admissible prime relative to $(f,K)$, and let $K_{\ell}/\Q_{\ell}$ be the completion of $K$ at the unique prime dividing $\ell$
$($so that $K_{\ell}=\Q_{\ell^{2}}$ is the quadratic unramified extension of $\Q_{\ell}$ $)$. There is a decomposition of $R[G_{K_{\ell}}]$-modules 
$ T_{f}=R(\varepsilon)\oplus{}R$, 
where $R(\varepsilon)$ (resp., $R$) denotes a copy of $R$ on which $G_{K_{\ell}}$ acts via the $p$-adic cyclotomic character $\varepsilon$
(resp., acts trivially). 
\end{lemma}

\subsection{Level raising}\label{level raising}

Let $n$ be a positive integer, and let $f\in{}S_{2}(N^{+},N^{-};\Z/p^{n}\Z)$.

\begin{hy}\label{asslevelraising}   The data $(\bar{\rho}_{f},N^{+},N^{-},p)$ with $N=N^{+}N^{-}$ satisfy the following conditions: 
\begin{enumerate}
\item $N^{-}$ is squarefree, $p\ges 5$ and $p\nmid{}N$;
\item $\bar{\rho}_{f} : G_{\Q}\fre{}\mathrm{GL}_{2}(\F_{p})$ is surjective.
\item If $q\Vert{}N^{-}$ and $q\equiv{}\pm{}1\ \mathrm{mod}\ p$, then $\bar{\rho}_{f}$
is ramified at $q$.
\end{enumerate}
\end{hy}

\begin{theo}\label{raising} Assume that $f\in{}S_{2}(N^{+},N^{-};\Z/p^{n}\Z)$
satisfies Assumption $\ref{asslevelraising}$. Let $S\in\mathscr{S}_k$  for some integer $1\leqslant k\leqslant n$. 
Then there exists an eigenform
\[f_{S} : \mathbb{T}_{N^{+},N^{-}S}\longrightarrow{}\Z/p^{k}\Z\]  
of level $(N^{+},N^{-}S)$ such that
$f_{S}(T_{q})=f_k(T_{q})$ for $q\nmid{}NS$ and $f_{S}(U_{q})=f(U_{q})$ for $q\mid N$, where recall that $f_k=f\pmod{p^k}$. 
Moreover, $f_{S}$ is unique up to multiplication by a unit in $\lri{\Z/p^{k}\Z}^{\ast}$.
\end{theo}
\begin{proof} 
Assume that $N^{-}>1$ and that  $N^{-}$ has an odd (resp., even) number of prime divisors.
In this case Theorem \ref{raising} is proved in Section $5$ (resp., Section $9$) of \cite{Be-Da-main} under slightly more restrictive
assumptions on $(\overline{\rho}_{f},N^{+},N^{-},p)$, subsequently removed in \cite{P-W}.
The method of \cite{Be-Da-main}  builds on work of Ribet (who considered the  case $n=k=1$),
and makes essential use of the generalisation of Ihara's Lemma to Shimura curves obtained by Diamond--Taylor \cite{DT}. 
We refer to \cite{Be-Da-main} for more details and references. 

Assume now that $N^{-}=1$. If $n=k=1$, the theorem was proved by Ribet. 
If $n>1$, it can be proved  by following the arguments of  Section $9$ of \cite{Be-Da-main}
(see in particular Proposition 9.2 and Theorem $9.3$), using Ihara's Lemma 
(rather than its generalisation by Diamond--Taylor) in the proof of Proposition 9.2.
\end{proof}

In the situation of Theorem \ref{raising}, we say that $f_S$ is the \emph{level raising 
of $f_k=f\pmod{p^n}$ at $S$}; it is defined up to units in $\Z/p^k\Z$.

\section{$p$-adic $L$-functions and special values formulae}\label{p-adic L}
Let $E/\Q$ be the elliptic curve fixed in the introduction, and 
let $f$ be the cuspform associated to $E$ by modularity.  
In this section we assume
the following: 
\begin{hy}\label{asspadicLfunctions}\hfill
\begin{enumerate}
\item $N^{-}$ is squarefree, $p\ges 5$ and $p\nmid{}N$.
\item $p\nmid h_K$. 
\item If $E/\Q_{p}$ has good ordinary reduction, then $a_{p}(E)\not\equiv{}\pm1\pmod{p}$ if $p$ is inert in $K$ and $a_{p}(E)\not\equiv{}1\pmod{p}$ if $p$ is split in $K$.
\item $\bar{\rho}_{f} : G_{\Q}\fre{}\mathrm{GL}_{2}(\F_{p})$ is surjective.
\item If $q\Vert{}N^{-}$ and $q\equiv{}\pm{}1\ \mathrm{mod}\ p$, then $\bar{\rho}_{f}$
is ramified at $q$.
\end{enumerate}
\end{hy}

Let $k\in{}\Z_{\ges{}1}\cup\{\infty\}$ be an integer or the symbol $\infty$. 
If $k$ is an integer let 
$L\in{}\mathscr{S}_{k}^{\mathrm{def}}$ be a \emph{definite} 
$k$-admissible integer (i.e. $\epsilon_{K}(LN^{-})=-1$) and denote by  
$g=f_{L}\in{}S_{2}(N^{+},LN^{-};\Z/p^{k}\Z)$ 
the $L$-level raising of $f$ modulo $p^{k}$.
If $k=\infty$  assume that $f$ is \emph{definite} (i.e. $\epsilon_{K}(N^{-})=-1$), and set $L=1$ 
and $g=f$; under the Jacquet--Langlands isomorphism $\mathbb{T}_{N^{+},LN^{-}}\cong{}\mathbf{T}_{N^{+},LN^{-}}$
recalled in Section $\ref{Shimura curves}$, if $k=\infty$ the form $g$ induces a $\Z_p$-valued 
ring homomorphism $\mathbf{T}_{N^{+},LN^{-}}\rightarrow \Z_p$, denoted by the same symbol $g$.
In both cases $X_{N^{+},LN^{-}}$ is a \emph{Gross curve}. Define $R=\Z_p$ if $k=\infty$ and $R=\Z/p^k\Z$ if $k$ is an integer. We may in both cases view $g$ as a morphism 
\[g : \mathbf{T}_{N^{+},LN^{-}}\longrightarrow R.\]

Fix a topological generator $\gamma$ of $G_{\infty}$. Let 
$\omega_{n}=\gamma^{p^{n}}-1$, denote  $\Phi_{n+1}(T)=\sum_{j=0}^{p-1}T^{j\cdot{}p^{n}}\in{}\Z[T]$ the $p^{n+1}$-cyclotomic polynomial. Set $\nu_{p}=0$ (resp., $\nu_{p}=1$) if $p$ is inert (resp., splits) in $K$. Define 
\begin{itemize}
\item $\omega_{0}^{+}=\omega_1^+=(\gamma-1)^{\nu_{p}}.$ 
\item $\omega_0^-=\gamma-1$. 
\item For each integer $n\ges2$, 
\[
          \omega_{n}^{+}=(\gamma-1)^{\nu_{p}}\prod_{1\les{}j\les{}\lfloor{}\frac{n}{2}\rfloor}\Phi_{2j}(\gamma).\]
\item For each integer $n\ges1$. 
\[          \omega_{n}^{-}=(\gamma-1)\prod_{1\les{}j\les{}\lfloor{}\frac{n+1}{2}\rfloor}\Phi_{2j-1}(\gamma).
\]
\end{itemize}

\subsection{Modular parametrisations}\label{snpL} 
Let $\mathfrak{m}_{L}\subset{}\mathbf{T}_{N^{+},LN^{-}}$ be the kernel of the reduction 
\[\bar{g} : \mathbf{T}_{N^{+},LN^{-}}\longrightarrow{}\F_{p}\] of $g$ modulo $p$,
and let $J_{\mathfrak{m}_{L}}$ and $\mathbf{T}_{\mathfrak{m}_{L}}$ denote the completions of 
$J_{N^{+},LN^{-}}$ and $\mathbf{T}_{N^{+},LN^{-}}$ at $\mathfrak{m}_{L}$ respectively.
Thanks to Ribet's level lowering theorem, Hypothesis $\ref{thehy}$ imply that 
\emph{Hypothesis CR} in \cite{P-W} holds true, so 
according to Theorem 6.2 and Proposition 6.5 of \cite{P-W} 
(relaxing one of the assumptions of \cite[Lemma 2.2]{Be-Da-main}) it follows that 
$J_{\mathfrak{m}_{L}}$ is a free $\mathbf{T}_{\mathfrak{m}_{L}}$-module of rank one. 
As a consequence $g$ induces a surjective morphism 
\[
               \psi_{g} : J_{N^{+},LN^{-}}\otimes_{\Z}\Z_{p}\lfre{}R
\]
satisfying $\psi_{g}(t\cdot{}x)=g(t)\cdot{}\psi_{g}(x)$ for every $t\in{}\mathbf{T}_{N^{+},LN^{-}}$
and every  $x\in{}J_{N^{+},LN^{-}}$, which is uniquely determined by $g$ up to multiplication 
by a $p$-adic unit. 
Let finally $\Lambda_R=\Lambda\otimes_{\Z_p}R=R[\![G_\infty]\!]$ (so that $\Lambda_R=\Lambda$ if 
$k=\infty$ and $\Lambda_R=\Lambda/p^k\Lambda$ if $k<\infty$) and put $\Lambda_{n,k}=R[G_n]$. 

\subsection{The ordinary case}\label{goodord} In the ordinary case, the construction of the 
$p$-adic $L$-function, which we will recall below, has been obtained in \cite{Be-Da1}.
Assume in this section that $E/\Q_{p}$ has \emph{good ordinary} reduction at $p$, i.e. $p\nmid{}N$ and 
$g(T_{p})=a_{p}(E) \pmod{p^{k}}$ is a $p$-adic unit. The  Hecke polynomial 
$X^{2}-g(T_{p})X+p$
has a unique root $\alpha_{p}(g)$ in $R$ which is congruent to $g(T_{p})$ modulo $p$ and hence is a 
$p$-adic unit.
Recall the compatible sequence $P_{\infty}(L)=(P_{n}(L))_{n\ges{}-1}$ of Gross points  fixed in Section $\ref{compheeg}$.
For every $n\ges{}1$  define 
\[
             \mathcal L_{g,n}=\frac{1}{\alpha_{p}(g)^{n}}
             \sum_{\sigma\in{}G_{n}}\Big(\psi_{g}\big(\sigma(P_{n-1}(L))\big)-\alpha_{p}(g)\cdot{}
             \psi_{g}\big(\sigma(P_{n}(L))\big)\Big)\cdot{}\sigma\in{}\iw_{n,k}.
\]
Since $\psi_{g}(T_{p}x)=a_{p}(E)\cdot{}\psi_{g}(x)$ for every $x\in{}J_{N^{+},LN^{-}}$,
a direct computation based on Equation $(\ref{eq:compdef})$ shows that the elements 
$(\mathcal L_{g,n})_{n\ges{}1}$ are compatible under the natural projection maps 
$
\iw_{n+1,k}\twoheadrightarrow{}\iw_{n,k}$. 
Define the  \emph{anticyclotomic square root $p$-adic $L$-function}
\[
                 \mathcal L_{g}=\lim_{n\to\infty}\mathcal L_{g,n}\in{}\Lambda_R
\]
as the inverse limit of the compatible sequence $(\mathcal L_{g,n})_{n\ges{}1}$
in $\inlim\iw_{n,k}=\Lambda_R$. 

For any $x\in \Lambda$ and any ring homomorphism $\chi:\Lambda\rightarrow \mathscr{O}$, 
define as usual $x(\chi)=\chi(x)$. Denote $\mathbf{1}$ the trivial character.  
One has (cf. Equation $(\ref{eq:compdef})$)
\begin{equation}\label{eq:gord1}
              \mathcal L_{g}(\mathbf{1})=\left\{ \begin{array}{lr} \!\!   \frac{1}{u_{K}}\big(1-\alpha_{p}(g)^{2}\big)
                 \cdot{}\psi_{g}(P_{K}(L))  &
                           \text{\ \ if\ }\ \  \epsilon_{K}(p)=-1, \\\vspace{-3mm}
                          \\
                                                       \!\!  \frac{-1}{u_{K}}\big(1-\alpha_{p}(g)\big)^{2}\cdot{}\psi_{g}(P_{K}(L))
                                                          & \text{\ \ if\ }\ \    \epsilon_{K}(p)=+1.  \end{array}      \right.
\end{equation}

\begin{lemma}\label{nonexcord} The equality $\mathcal{L}_{g}(\mathbf{1})=\psi_{g}\big(P_{K}(L)\big)$
holds in $R$ up to multiplication by an element in $R^*$.  
\end{lemma}

\begin{proof} This follows from the formulas above and 
Hypothesis \ref{asspadicLfunctions}(3). 
\end{proof}

The definition of $\mathcal L_{g}$ depends on the choice of the compatible system of Heegner points 
$P_{\infty}(L)$. If $Q_{\infty}(L)$ is another compatible system, then there exists $\gamma\in{}G_{\infty}$
such that $\gamma\big(P_{n}(L)\big)=Q_{n}(L)$ for every $n\ges{}0$
(cf. Section 2 of \cite{Be-Da1}). As a consequence the square root $p$-adic $L$-function $\mathcal L_{g}$
is well defined up to multiplication by $G_{\infty}$. Define 
the \emph{anticyclotomic $p$-adic $L$-function}
\[
                      L_p(g)=\mathcal L_{g}\cdot{}\mathcal L_{g}^{\iota}\in{} \Lambda_R,
\]
where $\iota$ is Iwasawa's main involution. Note that $L_p(g)$ is 
independent of the choice of $P_{\infty}(L)$. 

\subsection{The supersingular case} \label{p adic L supersingular} 
In the supersingular case, the construction of the 
$p$-adic $L$-function has been obtained in \cite{Da-Io} when $p$ is split, 
building on the fundamental work of Pollack \cite{Pollack}.
We extend the construction to the inert case. 
Assume that $E/\Q_{p}$ has good \emph{supersingular} reduction. 
As $p>3$ the Hasse bound gives $a_{p}(E)=0$. 
Let $\iw_{n,k}=\iw/(\omega_{n},p^{k})=R[G_n]$, and define $\iw_{n,k}^{\pm}=\iw/(\omega_{n}^{\pm},p^{k})$. 
Set 
\begin{itemize} 
\item $\tilde\omega_0^+=\tilde\omega_1^+=1$.
\item $\tilde\omega_0^-=\gamma-1$.
\item For each integer $n\ges2$, 
\[
             \tilde \omega_{n}^{+}=\prod_{1\les{}j\les{}\lfloor{}\frac{n}{2}\rfloor}\Phi_{{2j}}(\gamma).\]
\item For each integer $n\ges1$, 
            \[
                  \tilde\omega_{n}^{-}=\prod_{1\les{}j\les{}\lfloor{}\frac{n+1}{2}\rfloor}\Phi_{{2j-1}}(\gamma).
\] 
\end{itemize} 
Recall that $\gamma$ is  a topological generator of $G_\infty$, 
so that $\omega_{n}=(\gamma-1)\cdot{}\tilde\omega_{n}^{+}\cdot{}\tilde\omega_{n}^{-}$
(and $\tilde\omega_n^+=\omega_n^+$ in the inert case) for each $n\ges0$. 
For every $n\ges{}0$ define 
\[
                \mathcal{L}_{g,n}=\sum_{j=0}^{p^{n}-1}\psi_{g}\big(\gamma^{j}(P_{n}(L))\big)\cdot{}\gamma^{j}
                \in \Lambda_R.
\]

\begin{lemma}\label{inertpm} Let $\varepsilon=(-1)^n$. Then 
$\omega_{n}^{\varepsilon}\cdot{}\mathcal{L}_{g,n}\in{}\omega_{n}\cdot{}\Lambda_R$ for all integers $n\ges0$. \end{lemma}

\begin{proof} 
The case when $p$ is split is \cite[Proposition 2.8(1)]{Da-Io}, and we only need to check the inert case. 
The proof is by induction. 
If $n=0$, we have, using \eqref{P-0} for the second equality and $a_p(E)=0$ for the last,   
\[u_K\cdot\mathcal{L}_{g,0}=\psi_{g}\big(P_{0}(L)\big)=a_p(E) \psi_g(P_K(L))=0,\]
so $\mathcal{L}_{g,0}=0$ and the statement is true. 
If $n=1$, the statement is trivially true because $\omega_1^-=\omega_1$.  
%Understanding that congruences are modulo $\omega_1$, we have: 
%\[\mathcal{L}_{g,1}=\sum_{j=0}^{p-1}\psi_{g}\big(\gamma^{j}(P_{1}(L))\big)\cdot{}\gamma^{j}
%\equiv\sum_{j=0}^{p-1}\psi_{g}\big(\gamma^{j}(P_{1}(L))\big)
%=\psi_{g}\big(\mathrm{Trace}_{K_1/K_0}(P_{1}(L))\big)
%=-u_p\psi_g(P_K(L))\] where the congruence follows because 
%$\gamma\equiv1\pmod{\omega_1}$, 
%and the last equality follows from \eqref{P-1} and \eqref{eq:compdef}. 
%Now $\omega_1^-=\omega_1$ and therefore  $\omega_1^-\mathcal{L}_{g,1}=-\omega_1u_p\psi_g(P_K)\in\omega_1\Lambda_{1,k}$. 
So we can fix $n\ges 1$, assume the statement true for all integers $s$ with $0\les s\les n$ 
and prove it for $n+1$. We first prove a relation between $\mathcal{L}_{g,m+1}$ and $\mathcal{L}_{g,m-1}$, where $m\ges1$ is an integer. Understanding that congruences are modulo $\omega_{m}$, we have
\begin{equation}\label{eq:pinertpmrel}
\begin{split}
              \mathcal{L}_{g,m+1}
              &=\sum_{j=0}^{p^{m}-1}\lri{\sum_{i=0}^{p-1}\psi_{g}(\gamma^{j+ip^{m}}(P_{m+1}(L)))\cdot{}\gamma^{ip^{m}}}
              \cdot{}\gamma^{j+ip^m} \\
              &\equiv{}
              \sum_{j=0}^{p^{m}-1}\psi_{g}\big(\gamma^{j}\big(\mathrm{Trace}_{K_{m+1}/K_{m}}(P_{m+1}(L))\big)\big)
  \cdot{}\gamma^{j}  \text{ (because $\gamma^{ip^{m}}-1\equiv{}0\!\!\! \pmod{\omega_{m}}$)}\\
   &\equiv\sum_{j=0}^{p^{m}-1}
      \psi_{g}(T_{p}\gamma^{j}(P_{m})-\gamma^{j}(P_{m-1}(L)))\cdot{}\gamma^{j} 
      \text{ (by Equation \eqref{eq:compdef})} \\
          & 
            \equiv
              -\sum_{j=0}^{p^{m}-1}\psi_{g}\big(\gamma^{j}(P_{m-1}(L))\big)\cdot{}\gamma^{j} \text{ (because $a_{p}(E)=0$)}
              \\
&\equiv{}-\sum_{j=0}^{p^{m-1}-1}\sum_{i=0}^{p-1}
              \psi_{g}(\gamma^{j+ip^{m-1}}(P_{m-1}(L)))\cdot{}\gamma^{j+ip^{m-1}}  \\
 &      \equiv{}-\sum_{i=0}^{p-1}\lri{\sum_{j=0}^{p^{m-1}-1}
         \psi_{g}(\gamma^{j}(P_{m-1}(L)))\cdot{}\gamma^{j}}
              \cdot{}\gamma^{ip^{m-1}} \text{ (because $\gamma^{p^{m-1}}(P_{m-1}(L))=P_{m-1}(L)$)}\\
  &            \equiv{}-(\gamma^{(p-1)p^{m-1}}+\cdots{}+\gamma^{p^{m-1}}+1)\cdot{}\mathcal{L}_{g,m-1}\\
  &
              \equiv{}
              -\Phi_{{m}}(\gamma)\cdot{}\mathcal{L}_{g,m-1}  
              .
\end{split}\end{equation}
The statement then follows easily.
Indeed, if $n=2m+2\ges 2$ is even, 
\[\begin{split}
        \omega_{2m+2}^{+}\cdot{}\mathcal{L}_{g,2m+2}
        &=\omega_{2m+2}^{+}\cdot{}(-\Phi_{{2m+1}}(\gamma)\cdot{}\mathcal{L}_{g,2m}+\omega_{2m+1}\cdot{}\xi) \text{ (for some $\xi\in \Lambda_{2m+2,k}$)} \\
      &  =-\Phi_{{2m+2}}(\gamma)\cdot{}\Phi_{{2m+1}}(\gamma)\cdot{}\omega_{2m}^{+}\cdot{}\mathcal{L}_{g,2m}
        +\omega_{2m+2}^{+}\cdot{}\omega_{2m+1}\cdot{}\xi\\
& \in{}\underbrace{-\Phi_{{2m+2}}(\gamma)\cdot{}\Phi_{{2m+1}}(\gamma)\cdot{}\omega_{2m}}_{=\omega_{2m+2}}
    \cdot{}\iw_{2m+2,k}+\underbrace{\omega_{2m+2}^{+}\cdot{}\omega_{2m+1}}_{
      \text{divisible by }\omega_{2m+2}}\cdot{}\iw_{2m+2,k} 
\end{split}\]
where the last equality follows by the inductive hypothesis 
and because $\omega_{2m+2}^+\omega_{2m+1}$ is divisible by 
$\omega_{2m+2}$, since $\omega_{2m+2}^-=\omega_{2m-1}^-$. 
If $n=2n+3\ges 3$ is odd, we have similarly 
\[\begin{split}
        \omega_{2m+3}^{-}\cdot{}\mathcal{L}_{g,2m+3}
       & =\omega_{2m+3}^{-}\cdot{}(-\Phi_{{2m+2}}(\gamma)\cdot{}\mathcal{L}_{g,2m+1}+\omega_{2m+2}\cdot{}\xi) \text{ (for some $\xi\in\Lambda_R$)}\\
      &  =-\Phi_{{2m+3}}(\gamma)\cdot{}\Phi_{p^{2m+2}}(\gamma)\cdot{}\omega_{2m+1}^{-}\cdot{}\mathcal{L}_{g,2m+1}
        +\omega_{2m+3}^{-}\cdot{}\omega_{2m+2}\cdot{}\xi\\
 &  \in{}\underbrace{-\Phi_{{2m+3}}(\gamma)\cdot{}\Phi_{{2m+2}}(\gamma)\cdot{}\omega_{2m+1}}_{=\omega_{2m+3}}
    \cdot{}\iw_{2m+3,k}+
    \underbrace{\omega_{2m+3}^{-}\cdot{}\omega_{2m+2}}_{
      \text{divisible by}\ \omega_{2m+3}}\cdot{}\iw_{2m+3,k}.
\end{split}\]
where the last equality follows by the inductive hypothesis 
and because $\omega_{2m+3}^-\omega_{2m+2}$ is divisible by 
$\omega_{2m+3}$, since  $\omega_{2m+2}^{+}=\omega_{2m+3}^{+}$
\end{proof}

We will use repeatedly  the following elementary result. 
\begin{lemma}\label{divlemma} Suppose that $R$ is UFD. Let $x,y$ be non-zero elements of $R$, and let $z=xy$. Let finally $\pi$ be an element of $R$ such that $\pi$ and $z$ do not have common factors. The multiplication by $x:R\rightarrow R$ defines an isomorphism
$xR/(y,\pi)\cong (R/(z,\pi))[y].$ \end{lemma}

\begin{proof}
Multiplication by $x$ induces a map $x:R/(y,\pi)\rightarrow R/(z,\pi)$: indeed, suppose that $a$ and $b$ are elements of $R$ which satisfy $a=b+\alpha c$ for some $c\in R$ and $\alpha\in (y,\pi)$; then $xa=xb+x\alpha c $, and $x\alpha\in x(y,\pi)\subseteq (z,x\pi)\subseteq (z,\pi)$, so $x[a] =[xa]=[xb]=x[b]$ in $R/(z,\pi)$. Since $y[xa]=[axy]=[za]=[0]$, the image of the map is contained in $(R/(z,\pi))[y]$. We next show that the map is injective. Suppose $[xa]=[xb]$. Then $xa=xb+c\alpha$ for some $c\in R$ and $\alpha\in (z,\pi)$, so there exist $d$ and $e$ such that $x(a-b-dy)=e\pi$; now $\pi$ and $x$ do not have common irreducible factors, so since $R$ is a UFD, we have that $x\mid e$. Therefore, we can write $x(a-b-dy-f\pi)=0$ for some $f$.  
Since $x\neq 0$ and $R$ is a domain, we have $a-b-cy-f\pi=0$, so $[a]=[b]$ in $R/(y,\pi)$. We finally show that the map is surjective. Fix $[c]$ such that $y[c]=[cy]=0$. Then $cy=d\alpha$ for some $d\in R$ and $\alpha\in (z,\pi)$, so $y(c-ex)=f\pi$ for some $e$ and $f$, and again, since $y$ and $\pi$ do not have common factors, we see that $y\mid f$, so we can write
$y(c-ex-g\pi)=0$ for some $g$, and 
since $y\neq 0$ and $R$ is a domain, we have $c=ex+g\pi$. So $x[d]=[xd]=[c]$. \end{proof}

By Lemma \ref{divlemma}, in the split case multiplication by $\tilde\omega_{n}^{\mp}$ gives an isomorphism between 
$\iw_{n,k}^{\pm}$ and the $\omega_{n}^{\pm}$-torsion submodule of  $\iw_{n,k}$ (\emph{cf.} \cite[Section 4 ]{Io-Po}). In the inert case, again by Lemma \ref{divlemma}, multiplication by $\omega_n^+=\tilde\omega_{n}^{+}$ gives an isomorphism between 
$\iw_{n,k}^{-}$ and the $\omega_{n}^{-}$-torsion submodule of  $\iw_{n,k}$, and 
multiplication by $\omega_n^-=(\gamma-1)\tilde\omega_{n}^{-}$ gives an isomorphism between 
$\iw_{n,k}^{+}$ and the $\omega_{n}^{+}=\tilde\omega_{n}^{+}$-torsion submodule of  $\iw_{n,k}$.
Lemma $\ref{inertpm}$ then implies that if $\varepsilon=(-1)^n$, 
there exists elements $\mathcal{L}^\varepsilon_{n,k}\in \Lambda_{n,k}^\varepsilon$ such that  
\begin{itemize}\item  If $p$ is split in $K$ or $p$ is inert in $K$ and $\varepsilon=-1$ (the non-exceptional case): 
\[
\mathcal L_{g,n}
=\begin{cases}
(-1)^{n/2}\tilde\omega_n^{-}\mathcal L_{g,n}^+, \qquad \text{ if $n$ is even},\\
(-1)^{(n-1)/2}\tilde\omega_n^{+}\mathcal L_{g,n}^-, \ \text{ if $n$ is odd}.
\end{cases}\] 
\item If $p$ is inert in $K$ and $\varepsilon=+1$: 
\[
\mathcal L_{g,n}
= 
(-1)^{n/2} \omega_n^{-}\mathcal L_{g,n}^+ \] 
\end{itemize} 
Denote by $\pi_{2m+2}^{+} : \iw_{2m+2,k}^{+}\fre{}\iw_{2m,k}^{+}$
and $\pi_{2m+3}^{-} : \iw_{2m+3,k}^{-}\fre{}\iw_{2m+1,k}^{-}$ the natural projections. 

\begin{lemma}\label{normcomppm}
$\pi_{2m+2}^{+}(\mathcal{L}_{g,2m+2}^{+})=\mathcal{L}_{g,2m}^{+}$
and $\pi_{2m+3}^{-}(\mathcal{L}_{g,2m+3}^{-})=\mathcal{L}_{g,2m+1}^{-}$
for every $m\ges{}0$. 
\end{lemma}

\begin{proof} In the split case, this is Lemma 2.9 of \cite{Da-Io}, so we only need to check the inert case. 
Equation $(\ref{eq:pinertpmrel})$ shows that \begin{equation}\label{eq1}
\mathcal{L}_{g,2m+2}=-\Phi_{{2m+1}}(\gamma)
\cdot{}\mathcal{L}_{g,2m}+\omega_{2m+1}\cdot{}z\end{equation}
for some $z\in{}\Lambda_R$, for each $m\ges0$. 

We first prove the statement for $\mathcal{L}^+_{g,m}$. 
From \eqref{eq1}
\[
        (-1)^{m+1}\cdot{} \omega_{2m+2}^{-}\cdot{} {\mathcal{L}}_{g,2m+2}^{+}
        =(-1)^{m+1}\Phi_{{2m+1}}(\gamma)\cdot{}
       \omega_{2m}^{-}\cdot{} {\mathcal{L}}_{g,2m}^{+}+\omega_{2m+1}\cdot{}z.\] 
Both sides of the previous equation 
are divisible by $\omega_{2m+2}^{-}$. Since
$\Lambda_{2m+2}^+$ has no  nontrivial $\omega_{2m+2}^{-}$-torsion, 
dividing by $\omega_{2m+2}^{-}$, we get the result. 

We now prove the statement for $\mathcal{L}^-_{g,m}$. 
From \eqref{eq1}
\[
        (-1)^{m+1}\cdot{} \tilde\omega_{2m+3}^{+}\cdot{} {\mathcal{L}}_{g,2m+3}^{-}
        =(-1)^{m+1}\cdot{}\Phi_{{2m+3}}(\gamma)\cdot{}
        \tilde\omega_{2m+1}^{+}\cdot{} {\mathcal{L}}_{g,2m+1}^{-}+\omega_{2m+2}\cdot{}z. 
\] 
Both sides of the previous equation 
are divisible by $\tilde\omega_{2m+3}^{+}$. Since  $\Lambda_R^-$ has no  nontrivial 
$\tilde\omega_{2m+3}^{+}$-torsion, 
dividing by $\tilde\omega_{2m+3}^{+}$ we get the result.   
\end{proof}

Since $\inlim\iw_{2m,k}^{+}\cong{}\Lambda_R\cong{}\inlim\iw_{2m+1,k}^{-}$ (cf.\ Section 4 of \cite{Io-Po})
the previous lemma allows us to define
\[\mathcal{L}_g^\varepsilon=  \inlim_{m\in\N^\varepsilon}\mathcal{L}_{g,n}^{\varepsilon}\in{}\Lambda_R,\] where $\N^\varepsilon$ is the set of natural numbers $n$ satisfying $(-1)^n=\varepsilon$. 
Every continuous character $\chi : G_{\infty}\fre{}\bar{\Q}_{p}^{\ast}$ extends uniquely to 
a morphism  $\chi : \Lambda_R\fre{}\mathscr{O}_{\chi}/p^{k}\mathscr{O}_\chi$ of $\Z_{p}$-algebras,
where $\mathscr{O}_{\chi}=\Z_{p}[\chi(G_{\infty})]$. 
As before, denote by $\mathcal{L}_{g}^{\pm}(\chi)=\chi\big(\mathcal{L}_{g}^{\pm}\big)$ the value
of $\chi$ at 
$\mathcal{L}_{g}^{\pm}$ and by 
$\mathbf{1}$ the trivial character of $G_{\infty}$.

\begin{lemma}\label{nonexcss}
If $(f,K,p,\varepsilon)$ is non-exceptional, then 
the equality 
$\mathcal{L}_{g}^{\varepsilon}(\mathbf{1})=\psi_{g}(P_{K}(L))$
holds in $R$ up to multiplication by an element in $R^*$.  
\end{lemma}

\begin{proof}
This follows from \eqref{eq:compdef}, after noticing that  $P_{-1}(L)=u_{p}\cdot{}P_{K}(L)$ by \eqref{P-1} in
the non-exceptional case.
\end{proof}

\begin{remark}
In the exceptional case, a result analogous  to the equality in Lemma \ref{nonexcss} is not currently available, to the best knowledge of the authors. Indeed, 
$\mathcal{L}_{g,0}=u_K^{-1}a_p(E) \psi_g(P_K(L))=0$ by Lemma \ref{inertpm} (recall $a_p(E)=0$ under our assumptions). As a consequence, on the one hand $\mathcal{L}_{g,0}$ does not have a direct relation with $\psi_g(P_K(L))$, which is instead directly related to the special value of the $L$-series of $E$ over $K$. On the other hand, the equality $\mathcal{L}_{g,0}=0$, which can be interpreted as an exceptional-zero phenomenon, makes it possible to divide by $\gamma-1$ to define the anticyclotomic $p$-adic $L$-function $\mathcal{L}_g^+$; however,  the $p$-adic $L$-function thus obtained does not seem to have a clear relation with $\psi_g(P_L(K))$ as well. Therefore,  it might be interesting to further investigate an analogue of Lemma \ref{nonexcss} in the exceptional case, since it seems to require new ideas and a different approach than in the non-exceptional case.\end{remark}

As in the ordinary case, define 
\[L_p^\varepsilon(g)= \mathcal L_{g}^\varepsilon\cdot{}(\mathcal L_{g}^\varepsilon)^{\iota}\in{}\Lambda_R,\]
which is independent of the choice of $P_\infty(L)$. 

\section{Selmer groups}\label{secSel} 
Recall the notation introduced in Section \ref{compheeg}: 
for every integer $n\ges{}0$, 
$K_{n}/K$ is the cyclic subextension of $K_{\infty}/K$
of degree $p^{n}$, and 
$G_{n}=\mathrm{Gal}(K_{n}/K)$
(as in Section \ref{compheeg}, we assume that $p$ does not divide the class number of $K$, cf.\ Hypothesis \ref{controlass} below). Let $G_{\infty}=\mathrm{Gal}(K_{\infty}/K)$,   
$\iw_{n}=\Z_{p}[G_{n}]$ and $\iw=\Z_p[\![G_\infty]\!]$. In this section we also fix a finite flat extension $\mathscr{O}/\Z_p$, and 
define $\iw_{\mathscr{O},n}=\mathscr{O}[G_{n}]$ and $\iw_\mathscr{O}=\mathscr{O}[\![G_\infty]\!]$.
For each prime ideal 
$w$ of $K$, denote by $K_w$ the completion 
of $K$ at $w$. Fix an algebraic closure $\bar{K}_w$ of $K_w$, 
define $G_{K_w}=\mathrm{Gal}(\bar{K}_w/K_w)$, and let $I_{K_w}$ be 
the inertia subgroup of $G_{K_w}$.

Let $E/\Q$ be an elliptic curve of conductor $N$, and let $p\ges5$ be a prime number not dividing $N$.  We also let $N=N^+N^-$ denote the factorization of $N$ as before (a prime divides $N^+$ if and only if it is split in $K$, and divides $N^-$ if and only if it is inert in $K$). 
Let \[f\in S_2(\Gamma_0(N))\] be the newform attached to 
$E$ by modularity, which we identify, with a slight abuse of notation, 
with a modular form $f\in S_2(N^+,N^-;\Z_p)$ 
by the Jacquet--Langlands correspondence (cf.\ Section \ref{JLsec}).  
The representations $A_f$ and $T_f$ associated with $f$ as in Section \ref{eigenforms} 
are then the $p$-divisible group and the $p$-adic Tate module of $E$, respectively.  
In this section we work under the following:

\begin{hy}\label{controlass} \hfill
\begin{enumerate}
\item $N^{-}$ is squarefree, $p\ges 5$ and $p\nmid{}N$.
\item The rational prime $p$ does not divide the class number and the discriminant of $K$.
\item $\bar{\rho}_{f} : G_{\Q}\fre{}\mathrm{GL}_{2}(\F_{\!p})$ is irreducible.
\item If $E$ has good ordinary reduction at $p$, then $a_{p}(E)\not\equiv{}\pm1\pmod{p}$ if $p$ is inert in $K$, and $a_{p}(E)\not\equiv{}1\pmod{p}$ if $p$ is split in $K$.
\item If $q$ is a prime dividing $N^{+}$, then 
$H^{0}(G_{\Q_{q}},E_{p})=0$ and $\bar{\rho}_{f}$ is ramified at $q$.
\end{enumerate} 
\end{hy}

\subsection{$\varepsilon$-rational points}\label{efinsub} 
In this subsection we assume that $E$ has supersingular reduction at $p$. 
Let $\mathfrak{p}$ be a prime of $K$ dividing $p$.
For every $n\in{}\N\cup\{\infty\}$ denote by $\Psi_{n}=K_{n,\mathfrak{p}}$ 
the completion of $K_{n}$ at the unique prime dividing $\mathfrak{p}$, by $O_{n}$ the ring of integers of 
$\Psi_{n}$ and by $\mathfrak{m}_{n}$ its maximal ideal. Set $\Psi=\Psi_{0}$, $O=O_{0}$ and 
$\mathfrak{m}=\mathfrak{m}_{0}$. Then 
$\Psi_{\infty}$ is a totally ramified $\Z_{p}$-extension of $\Psi$,
whose Galois group can be identified with $G_{\infty}$
(via  $i_{p} : \bar{\Q}\hookrightarrow{}\bar{\Q}_{p}$).
Let $\mathbf{E}/O$ be the formal group of $E/\Psi$, which  gives the kernel of the reduction modulo 
$p$ on $E$, and let $\log_{\mathbf{E}} : \mathbf{E}\fre{}\mathbf{G}_{a}$
be the formal group logarithm. The formal group $\mathbf{E}/O$ is a Lubin--Tate  group 
for  the uniformiser $-p\in{}O$ (\cite{Honda}). 
Since $E_{p^{m}}\cong{}\mathbf{E}_{p^{m}}$
for every $m\ges{}1$ as $G_{\mathfrak{p}}$-modules, the $p$-adic Tate module $T_{f}$ of $E$ 
is isomorphic to the $p$-adic Tate module of $\mathbf{E}$, hence has a natural  structure of 
$O[\mathrm{Gal}(\bar{\Psi}/\Psi)]$-module. 

Denote by $\Xi$ the set of $\bar{\Q}_{p}^{\ast}$-valued finite order characters on $G_{\infty}$.
For every  $\chi\in{}\Xi$ let $n_{\chi}$ be  the smallest nonnegative integer such that 
$\chi$ factors through $G_{n_{\chi}}$, and let 
\[\Xi^{\pm}=\{\chi\in{}\Xi\ |\ 
n_{\chi}\ges{}1,\ (-1)^{n_{\chi}}=\pm1\}.\] If $p$ splits in $K/\Q$,
set $\Xi_{p}^{\pm}=\Xi^{\pm}$; if $p$ is inert in $K/\Q$ set 
$\Xi_{p}^{+}=\Xi^{+}$ and $\Xi_{p}^{-}=\Xi^{-}\cup\{\mathbf{1}\}$,
where $\mathbf{1}\in{}\Xi$ is the trivial character on $G_{\infty}$.
Let
$\log_{\chi} : \mathbf{E}(\mathfrak{m}_{\infty})\fre{}\bar{\Q}_{p}$ be the morphism 
sending $y$ in $\mathbf{E}(\mathfrak{m}_{\infty})$ to
$$\log_{\chi}(y)=p^{-m}\sum_{\sigma\in{}G_{m}}\chi(\sigma)^{-1}\log_{\mathbf{E}}(y^{\sigma}),$$
where $m=m(\chi,y)$ is any positive integer large enough so that  $m\ges{}n_{\chi}$ 
and  $y$ belongs to $\mathbf{E}(\mathfrak{m}_{m})$. 
Following  \cite{Rub-EllHeeg} set 
\begin{align*}
                         \mathbf{E}(\mathfrak{m}_{\infty})_{\pm} & =\big\{y\in{}\mathbf{E}(\mathfrak{m}_{\infty})\ \big|\ 
                         \log_{\chi}(y)=0\ \text{for every }\ \! \chi\in{}\Xi_{p}^{\mp}
                         \big\},
\end{align*}
and for every $1\les{}k\les{}\infty$ define  
\[
         H^{1}_{\mathrm{fin},\pm}(\Psi_{\infty},A_{f,k})=\mathbf{E}(\mathfrak{m}_{\infty})_{\pm}\otimes_{\Z}(\divp)_{p^{k}},
\]
viewed as submodules of  $H^{1}(\Psi_{\infty},A_{f,k})$
under the local Kummer map. 

Let $\varepsilon$ denote one of the signs $+$ or $-$. With an abuse of notation, we sometimes identify the sign $\varepsilon=\pm$ with $\pm1$,
so that the equation 
$(-1)^n=\varepsilon$ makes sense for each integer $n$, and states that $n$ is even if $\varepsilon=+$ and $n$ is odd if $\varepsilon=-$. 
For every integer $n\ges{}0$
and every prime $\mathfrak{p}$ of $K$ dividing $p$ set  
\[
\mathbf{E}(\mathfrak{m}_{n})_{\varepsilon}=\mathbf{E}(\mathfrak{m}_{n})\cap 
\mathbf{E}(\mathfrak{m}_{\infty})_{\varepsilon}.\] 
If follows from the definitions that 
$\omega_{n}^{\varepsilon}\cdot{}\mathbf{E}(\mathfrak{m}_n)_{\varepsilon}=0$
for every $n\ges{}0$ and $\mathfrak{p}|p$; in particular, if $p$ is inert in $K$ 
then
$\mathbf{E}(\mathfrak{m})_{+}=0$ and 
$\mathbf{E}(\mathfrak{m})_{-}=\mathbf{E}(\mathfrak{m})$, while if 
$p$ is split in $K$, we have 
$\mathbf{E}(\mathfrak{m})_{+}=\mathbf{E}(\mathfrak{m})_{-}=\mathbf{E}(\mathfrak{m})$. 

Denote by  $\iw_{O}$ the tensor product of $\iw$ with $O$.
For every Galois extension 
$\Psi'/\Psi$ the group $\mathbf{E}(\Psi')$
is a module over  $O[\mathrm{Gal}(\Psi'/\Psi)]$. 
The next theorem, which elucidates the structure of the $\Lambda_O$-modules $\mathbf{E}(\Psi_n)$ and $\mathbf{E}(\Psi_n)_\varepsilon$, has been obtained 
by Iovita--Pollack \cite[Theorem 4.5]{Io-Po} in the split case (building on the work of Kobayashi \cite{Kobss}) and by Burungale--Kobayashi--Ota \cite[Theorem 5.5]{BKO1} in the inert case.
It shows that the $\varepsilon$-local points enjoy trace relations analogous to those satisfied by the families of Heegner points and Gross points intervening in the definition of the $\varepsilon$-$p$-adic $L$-functions. 

\begin{theo}\label{prop:formal-group} \hfill
\begin{enumerate}
\item $\mathbf{E}(\Psi)$ is a free $O$-module of rank $1$. 
Choose a generator $\boldsymbol{d}_{\mathfrak{p},0}\in \mathbf{E}(\Psi)$; define $\boldsymbol{d}_{\mathfrak{p},0}^{+}=\boldsymbol{d}_{\mathfrak{p},0}^{-}=\boldsymbol{d}_{\mathfrak{p},0}$ if $p$ is split in $K$ and 
$\boldsymbol{d}_{\mathfrak{p},0}^{+}=0$, $\boldsymbol{d}_{\mathfrak{p},0}^{-}=\boldsymbol{d}_{\mathfrak{p},0}$ if $p$ is inert in $K$. 
\item If $n\ges{}1$ and $\varepsilon=(-1)^n$, then   
$\mathbf{E}(\Psi_n)_{\varepsilon}$ is a free $\iw_{O}/(\omega_{n}^{\varepsilon})$-module of rank $1$. We can choose generators $\boldsymbol{d}_{\mathfrak{p},n}^\epsilon\in \mathbf{E}(\Psi_n)_{\varepsilon}$ satisfying  
the following trace relations:
\begin{itemize} 
\item $\mathrm{Trace}_{n+2/n+1}(\boldsymbol{d}_{\mathfrak{p},n+2}^{\varepsilon})=-
\boldsymbol{d}_{\mathfrak{p},n}^{\varepsilon}$; 
\item $\mathrm{Trace}_{1/0}(\boldsymbol{d}_{\mathfrak{p},1}^{-})=u\cdot \boldsymbol{d}_{\mathfrak{p},0}$, for some unit $u\in\Z_p^\times$. 
\end{itemize}

\item If $n\ges{}1$ and $\varepsilon=-(-1)^{n}$, define  
$\boldsymbol{d}_{\mathfrak{p},n}^{\varepsilon}=\boldsymbol{d}_{\mathfrak{p},n-1}^{\varepsilon}\in \mathbf{E}(\Psi_{n-1})_{\varepsilon}$. Then 
the $\iw_{O}$-module $\mathbf{E}(\Psi_n)$ is generated by 
$\boldsymbol{d}_{\mathfrak{p},n}^{\varepsilon}$
and $\boldsymbol{d}_{\mathfrak{p},n}^{-\varepsilon}$. 
\end{enumerate}
\end{theo}

Theorem \ref{prop:formal-group} furnishes elements $\boldsymbol{d}_{\mathfrak{p},n}^\varepsilon$ defined for all $n\ges 0$ and $\varepsilon\in\{\pm\}$ which we consider fixed from now on.

\subsection{Selmer groups}\label{selsec} Let $k\in{}\N\cup\{\infty\}$.
If $k\in{}\N$ and $L\in{}\mathscr{S}_{k}$, 
let $g: \mathbb{T}_{N^{+},N^{-}L}\fre{}\Z/p^{k}\Z$ 
be the level raising of $f_k=f\pmod{p^k}$ at $L$ (cf. Section \ref{level raising}).
If $k=\infty$, set $L=1$ and $g=f$.
Fix an isomorphism of $G_{K}$-modules between $T_{f,k}$ and the $p$-adic representation 
$T_{g}$ associated with $g$, which also fixes an isomorphism between $A_{f,k}$ and 
$A_{g}=\Hom{\Z_{p}}(T_{g},\mu_{p^{\infty}})$. 
We often identify $A_{f,k}$ with $T_{f}\otimes_{\Z_{p}}(\divp)_{p^{k}}$, hence 
$A_{g}$ with $T_{g}\otimes_{\Z_{p}}(\divp)_{p^{k}}$, using the Weil pairing 
on $E$ (with the convention $p^{\infty}=0$).
Let $\iota : \iw\lfre{}\iw$ be Iwasawa's main involution (acting as inversion on group-like elements), and let  
\[
                      \mathbf{T}_{g}=T_{g}\otimes_{\Z_{p}}\iw(\epsilon_{\infty}^{-1})\ \ \ \text{and}\ \ \ 
                     \mathbf{A}_{g}=\Hom{\mathrm{cont}}(\mathbf{T}_{g}^{\iota},\mu_{p^{\infty}}),
\]
where $\epsilon_{\infty} : G_{K}\fre{}\iw^{\ast}$ is the tautological representation (obtained 
by composing the canonical projection $G_K\twoheadrightarrow G_\infty$ with the inclusion $G_\infty\hookrightarrow\Lambda^*$ of group-like elements) 
and one writes 
$M^{\iota}=M\otimes_{\iw,\iota}\iw$ for every $\iw$-module $M$. 
We also define the scalar extensions 
\[
\mathbf{T}_{g,\mathscr{O}}=\mathbf{T}_g\otimes_{\Z_p}\mathscr O\ \ \ \text{and}\ \ \ 
\mathbf{A}_{g,\mathscr{O}}=\mathbf{A}_g\otimes_{\Z_p}\mathscr O.
\] 
For every  ideal $\mathfrak{P}$ of $\iw_\mathscr{O}$, set    $\mathscr{O}_{\mathfrak{P}}=\iw_\mathscr{O}/\mathfrak{P}$, and set 
\[
T_{g,\mathscr{O}}(\mathfrak{P})=\mathbf{T}_{g,\mathscr{O}}/\mathfrak{P}\cdot{}\mathbf{T}_{g,\mathscr{O}}\ \ \  \text{and}\ \ \
A_{g,\mathscr{O}}(\mathfrak{P})=\mathbf{A}_{g,\mathscr{O}}[\mathfrak{P}].
\]
Write $T_g(\mathfrak P)$ for $T_{g,\Z_p}(\mathfrak{P})$ and 
$A_{g}(\mathfrak{P})$ for $A_{g,\Z_p}(\mathfrak{P})$.
Then $A_{g,\mathscr{O}}(\mathfrak{P})$ is isomorphic as a $\iw_\mathscr{O}[G_{K}]$-module 
to the Kummer dual $\Hom{\mathscr{O}}(T_{g,\mathscr{O}}(\mathfrak{P}^{\iota})^{\iota},\mu_{p^{\infty}}\otimes_{\Z_p}\mathscr{O})$ of 
$T_{g,\mathscr{O}}(\mathfrak{P}^{\iota})^{\iota}$, where $\mathfrak{P}^{\iota}=\iota(\mathfrak{P})$. 
For every finite prime $w$ of $K$, local Tate duality gives then a perfect
$\mathscr{O}$-bilinear  pairing 
\begin{equation}\label{eq:paievP}
         \dia{-,-}_{\mathfrak{P},w} : 
         H^{1}(K_{w},T_{g,\mathscr{O}}(\mathfrak{P}))\times{}H^{1}(K_{w},A_{g,\mathscr{O}}(\mathfrak{P}^{\iota}))\lfre{}
         \mathscr{K}/\mathscr{O},
\end{equation}
where $\mathscr{K}=\mathrm{Frac}(\mathscr{O})$ is the fraction field of $\mathscr{O}$, 
such that $$\dia{\lambda\cdot{}x,y}_{\mathfrak{P},w}=\dia{x,\iota(\lambda)\cdot{}y}_{\mathfrak{P},w}$$
for every $\lambda\in{}\iw_\mathscr{O}$, every $x\in H^{1}(K_{w},T_{g,\mathscr{O}}(\mathfrak{P}))$,
and every $y\in{}H^{1}(K_{w},A_{g,\mathscr{O}}(\mathfrak{P}^{\iota}))$.

\subsubsection{Primes dividing $p$}\label{subsec-p}
Let $\mathfrak{p}$ be a prime  of $K$  dividing $p$
and let $\varepsilon\in{}\{\emptyset,\pm\}$. 
For $\varepsilon=\pm$, recall the groups  
$H^{1}_{\mathrm{fin},\varepsilon}(K_{\infty,\mathfrak{p}},A_{f,k})$
defined in Section \ref{efinsub}, 
and for $\varepsilon=\emptyset$ define 
\[ 
         H^{1}_{\mathrm{fin}}(K_{\infty,\mathfrak{p}},A_{f,k})=E(K_{\infty,\mathfrak{p}})\otimes_{\Z}(\divp)_{p^{k}},
\]
viewed as submodules of  $H^{1}(K_{\infty,\mathfrak{p}},A_{f,k})$
under the Kummer map. 
Shapiro's Lemma  yields a  natural isomorphism 
of $\iw$-modules $H^{1}(K_{\mathfrak{p}},\mathbf{A}_{g})\cong{}H^{1}(K_{\infty,\mathfrak{p}},A_{g})$,
which we often consider an equality, and  we denote by 
$H^1_{\mathrm{fin},\varepsilon}(K_{\mathfrak p},\mathbf{A}_g)$ the submodule of $H^{1}(K_{\mathfrak{p}},\mathbf{A}_{g})$
corresponding to $H^{1}_{\mathrm{fin},\varepsilon}(K_{\infty,\mathfrak{p}},A_{f,k})$ via this isomorphism. 
We then define 
\[
   H^1_{\mathrm{fin},\varepsilon}(K_{\mathfrak p},\mathbf{A}_{g,\mathscr{O}})\subset{}H^{1}(K_{\mathfrak{p}},\mathbf{A}_{g,\mathscr{O}})
\]
as the image of $H^1_{\mathrm{fin},\varepsilon}(K_{\mathfrak p},\mathbf{A}_g)\otimes_{\Z_p}\mathscr O$ 
via the canonical isomorphism  
$H^1(K_{\mathfrak p},\mathbf{A}_g)\otimes_{\Z_p}\mathscr{O}\cong{}H^1(K_{\mathfrak p},\mathbf{A}_{g,\mathscr{O}})$.
For every ideal  $\mathfrak{P}$ of $\iw_{\mathscr{O}}$ define  
\[
         H^{1}_{\mathrm{fin},\varepsilon}(K_{\mathfrak{p}},A_{g,\mathscr{O}}(\mathfrak{P}))
         \subset{}H^{1}(K_{\mathfrak{p}},A_{g,\mathscr{O}}(\mathfrak{P}))
\]
as the inverse image of $H^{1}_{\mathrm{fin},\varepsilon}(K_{\mathfrak{p}},\mathbf{A}_{g,\mathscr{O}})$ under the map 
$H^{1}(K_{\mathfrak{p}},A_{g,\mathscr{O}}(\mathfrak{P}))\lfre{} H^{1}(K_{\mathfrak{p}},\mathbf{A}_{g,\mathscr{O}})$ (induced in cohomology by the inclusion  
$A_{g,\mathscr{O}}(\mathfrak{P})=\mathbf{A}_{g,\mathscr{O}}[\mathfrak{P}]\hlfre{}\mathbf{A}_{g,\mathscr{O}}$), and define 
\[
        H^{1}_{\mathrm{fin},\varepsilon}(K_{\mathfrak{p}},T_{g,\mathscr{O}}(\mathfrak{P}))
         \subset{}H^{1}(K_{\mathfrak{p}},T_{g,\mathscr{O}}(\mathfrak{P}))
\]
as  the orthogonal complement of $H^{1}_{\mathrm{fin},\varepsilon}(K_{\infty,\mathfrak{p}},A_{g,\mathscr{O}}(\mathfrak{P}^{\iota}))$
under the  local Tate pairing $\dia{-,-}_{\mathfrak{P},\mathfrak{p}}$. 
If $\texttt{M}_{g,\mathscr{O}}$ denotes either $T_{g,\mathscr{O}}$ or $A_{g,\mathscr{O}}$, let  $H^{1}_{\mathrm{sing},\varepsilon}(K_{\mathfrak{p}}
,\texttt{M}_{g,\mathscr{O}}(\mathfrak{P}))$ be  the quotient of $H^{1}(K_{\mathfrak{p}},\texttt{M}_{g,\mathscr{O}}(\mathfrak{P}))$
by the finite subgroup $H^{1}_{\mathrm{fin},\varepsilon}(K_{\mathfrak{p}},\texttt{M}_{g,\mathscr{O}}(\mathfrak{P}))$, so that we have a canonical exact sequence 
\[0\longrightarrow H^{1}_{\mathrm{fin},\varepsilon}(K_{\mathfrak{p}},\texttt{M}_{g,\mathscr{O}}(\mathfrak{P}))\longrightarrow H^{1}(K_{\mathfrak{p}},\texttt{M}_{g,\mathscr{O}}(\mathfrak{P}))\longrightarrow H^{1}_{\mathrm{sing},\varepsilon}(K_{\mathfrak{p}}
,\texttt{M}_{g,\mathscr{O}}(\mathfrak{P}))\longrightarrow 0.\]
A global class $x\in{}H^{1}(K,\texttt{M}_{g,\mathscr{O}}(\mathfrak{P}))$ is 
said to be 
\emph{$\varepsilon$-finite at $\mathfrak{p}$} if $\mathrm{res}_{\mathfrak{p}}(x)\in{}
H^{1}_{\mathrm{fin},\varepsilon}(K_{\mathfrak{p}},\texttt{M}_{g,\mathscr{O}}(\mathfrak{P}))$.
For any element 
$s\in H^1(K,\texttt{M}_{g,\mathscr{O}}(\mathfrak P))$, denote $\partial_\mathfrak{p}(s)$ the projection of the restriction of  $s$ at $\mathfrak{p}$ to 
the singular quotient of $H^1(K_\mathfrak{p},\texttt{M}_{g,\mathscr{O}}(\mathfrak P))$. We call $\partial_\mathfrak{p}$ 
the \emph{residue map} at $\mathfrak p$, and $\partial_p=\oplus_{\mathfrak{p}\mid p}\partial_\mathfrak{p}$ 
the \emph{residue map at $p$}. 

\subsubsection{Primes dividing $N^-$}\label{subsec-N^-}
Let $\mathfrak{P}$ be an ideal of  $\iw_\mathscr{O}$ and let 
$\ell$ be a rational prime dividing $N^{-}$.
Then $\ell$ is inert in $K/\Q$ and $\ell\cdot{}\mathcal{O}_{K}$ splits completely in $K_{\infty}/K$. 
As a consequence the $G_{K_\ell}$-representation $T_{g,\mathscr{O}}(\mathfrak{P})$ is isomorphic to the base change 
$T_{g}\otimes_{\Z_{p}}\mathscr{O}_{\mathfrak{P}}$ (with $G_{K_\ell}$ acting trivially on $\mathscr{O}_{\mathfrak{P}}$).
The elliptic curve $E/K_{\ell}$ is a Tate curve, i.e.\ is isomorphic as a rigid analytic variety 
to the quotient of the multiplicative group $\mathbf{G}_{m}/K_{\ell}$ by the lattice $q_{\ell}^{\Z}$ generated by the Tate period 
$q_{\ell}\in{}\ell\cdot{}\Z_{\ell}$ (\cite[Chapter 5]{Sil-2}). 
This gives a short exact sequence of $G_{K_\ell}$-modules 
\[
            0\lfre{}T_{f,k}^{(\ell)}\lfre{}T_{f,k}\lfre{}T_{f,k}^{[\ell]}\lfre{}0,
\]
where $T_{f,k}^{(\ell)}\cong{}\Z_{p}/p^{k}(1)$ and $T_{f,k}^{[\ell]}\cong{}\Z_{p}/p^{k}$, which in turn 
induces an exact sequence of $\mathscr{O}_{\mathfrak{P}}[G_{K_\ell}]$-modules
\[
0\lfre{}T_{g,\mathscr{O}}^{(\ell)}(\mathfrak{P})\lfre{}T_{g,\mathscr{O}}(\mathfrak{P})\lfre{}T_{g,\mathscr{O}}^{[\ell]}(\mathfrak{P})\lfre{}0,
\] 
with $T_{g,\mathscr{O}}^{(\ell)}(\mathfrak{P})\cong{}\mathscr{O}_{\mathfrak{P}}(1)\otimes_{\Z_{p}}\Z_{p}/p^{k}$
and $T_{g,\mathscr{O}}^{[\ell]}(\mathfrak{P})\cong{}\mathscr{O}_{\mathfrak{P}}\otimes_{\Z_{p}}\Z_{p}/p^{k}$.
Define $A_{g,\mathscr{O}}^{(\ell)}(\mathfrak{P})$ and $A_{g,\mathscr{O}}^{[\ell]}(\mathfrak{P})$
to be the Kummer duals of $T_{g,\mathscr{O}}^{[\ell]}(\mathfrak{P}^{\iota})^{\iota}$
and $T_{g,\mathscr{O}}^{(\ell)}(\mathfrak{P}^{\iota})^{\iota}$ respectively, 
so that one has an exact sequence 
\[
           0\lfre{}A_{g,\mathscr{O}}^{(\ell)}(\mathfrak{P})\lfre{}A_{g,\mathscr{O}}(\mathfrak{P})\lfre{}A_{g,\mathscr{O}}^{[\ell]}(\mathfrak{P})\lfre{}0
\]
of $\mathscr{O}_{\mathfrak{P}}[G_{K_\ell}]$-modules.
If $\texttt{M}_{g,\mathscr{O}}$ is either $T_{g,\mathscr{O}}$ or $A_{g,\mathscr{O}}$, define the \emph{ordinary subspace} of 
$H^{1}(K_{\ell},\texttt{M}_{g,\mathscr{O}}(\mathfrak{P}))$ by 
\[
              H^{1}_{\mathrm{ord}}(K_{\ell},\texttt{M}_{g,\mathscr{O}}(\mathfrak{P}))
              =\mathrm{Im}\lri{H^{1}(K_{\ell},\texttt{M}_{g,\mathscr{O}}^{(\ell)}(\mathfrak{P}))\lfre{}H^{1}(K_{\ell},\texttt{M}_{g,\mathscr{O}}(\mathfrak{P}))}.
\]
As easily proved, $H^{1}_{\mathrm{ord}}(K_{\ell},T_{g,\mathscr{O}}(\mathfrak{P}))$
is the orthogonal complement of $H^{1}_{\mathrm{ord}}(K_{\ell},A_{g,\mathscr{O}}(\mathfrak{P}^{\iota}))$
under $\dia{-,-}_{\mathfrak{P},\ell}$.
A global class in $H^{1}(K,\texttt{M}_{g,\mathscr{O}}(\mathfrak{P}))$ is said to be 
\emph{ordinary at $\ell$}
if its restriction at $\ell$ belongs to the ordinary subspace 
$H^{1}_{\mathrm{ord}}(K_{\ell},\texttt{M}_{g,\mathscr{O}}(\mathfrak{P}))$.

\subsubsection{Primes dividing $L$}\label{primesL}
Let $\mathfrak{P}$ be an ideal  of $\iw_\mathscr{O}$,
and let  $\ell$ be a prime divisor of $L$. As above $\ell\cdot{}\mathcal{O}_{K}$
splits completely in $K_{\infty}/K$ and 
$T_{g,\mathscr{O}}(\mathfrak{P})=T_{g}\otimes_{\Z_{p}}\mathscr{O}_{\mathfrak{P}}$ 
as $G_{K_\ell}$-modules, with $G_{K_\ell}$ acting trivially on the 
second factor. 
Lemma $\ref{decadm}$ then implies that  the $\mathscr{O}_{\mathfrak{P}}[G_{K_\ell}]$-module 
$T_{g,\mathscr{O}}(\mathfrak{P})$ is isomorphic to the direct sum of 
$T_{g,\mathscr{O}}^{(\ell)}(\mathfrak{P})=\mathscr{O}_{\mathfrak{P}}(1)\otimes_{\Z_{p}}\Z_{p}/p^{k}$
and $T_{g,\mathscr{O}}^{[\ell]}(\mathfrak{P})=\mathscr{O}_{\mathfrak{P}}\otimes_{\Z_{p}}\Z_{p}/p^{k}$
(where by definition $\mathscr{O}_\mathfrak P(1)=\Z_p(1)\otimes_{\Z_p}\mathscr{O}_\mathfrak P$ as Galois modules).
Let as above  $A_{g,\mathscr{O}}^{(\ell)}(\mathfrak{P})$ and $A_{g,\mathscr{O}}^{[\ell]}(\mathfrak{P})$
be the Kummer duals of $T_{g,\mathscr{O}}^{[\ell]}(\mathfrak{P}^{\iota})^{\iota}$
and $T_{g,\mathscr{O}}^{(\ell)}(\mathfrak{P}^{\iota})^{\iota}$ respectively.
For $\texttt{M}_{g,\mathscr{O}}\in\{T_{g,\mathscr{O}},A_{g,\mathscr{O}}\}$,   
define the \emph{ordinary subspace} of $H^{1}(K_{\ell},\texttt{M}_{g,\mathscr{O}}(\mathfrak{P}))$ by 
the equality
\[
             H^{1}_{\mathrm{ord}}(K_{\ell},\texttt{M}_{g,\mathscr{O}}(\mathfrak{P}))=
             H^{1}(K_{\ell},\texttt{M}_{g,\mathscr{O}}^{(\ell)}(\mathfrak{P})).
\]
By Lemma $\ref{decadm}$, $ H^{1}_{\mathrm{ord}}(K_{\ell},\texttt{M}_{g,\mathscr{O}}(\mathfrak{P}))$
is also isomorphic  to the \emph{singular quotient}
\[H^{1}_{\mathrm{sing}}(K_{\ell},\texttt{M}_{g,\mathscr{O}}(\mathfrak{P}))=H^{1}(K_{\ell},\texttt{M}_{g,\mathscr{O}}(\mathfrak{P}))/H^{1}_{\mathrm{fin}}(K_{\ell},\texttt{M}_{g,\mathscr{O}}(\mathfrak{P})).\]
Moreover, note that $H^{1}_{\mathrm{ord}}(K_{\ell},T_{g,\mathscr{O}}(\mathfrak{P}^{\iota}))$
(resp., $H^{1}_{\mathrm{fin}}(K_{\ell},T_{g,\mathscr{O}}(\mathfrak{P}^{\iota}))$)
is the orthogonal complement 
of $H^{1}_{\mathrm{ord}}(K_{\ell},A_{g,\mathscr{O}}(\mathfrak{P}))$
(resp., $H^{1}_{\mathrm{fin}}(K_{\ell},A_{g,\mathscr{O}}(\mathfrak{P}))$) 
under $\dia{-,-}_{\mathfrak{P},\ell}$ .
A global class in $H^{1}(K,\texttt{M}_{g,\mathscr{O}}(\mathfrak{P}))$ is said to be  \emph{ordinary}
(resp., \emph{finite}) \emph{at $\ell$} if its restriction at $\ell$ 
belongs to the ordinary (resp., finite) subspace of $H^{1}(K_{\ell},\texttt{M}_{g,\mathscr{O}}(\mathfrak{P}))$.

\subsubsection{Primes outside $LN^-p$}\label{subsec-outside}
Let  $w$ be a prime of $K$ which does not divide $LN^{-}p$,  let $\mathfrak{P}$
be an ideal  of $\iw_\mathscr{O}$ and let $\texttt{M}_{g,\mathscr{O}}$ denote either $T_{g,\mathscr{O}}$ or $A_{g,\mathscr{O}}$.
A  global class in $H^{1}(K,\texttt{M}_{g,\mathscr{O}}(\mathfrak{P}))$
is \emph{finite} (resp., \emph{trivial}) \emph{at $w$} if its restriction at $w$ belongs to the finite subspace 
\[H^{1}_{\mathrm{fin}}(K_{w},\texttt{M}_{g,\mathscr{O}}(\mathfrak{P}))=H^{1}(G_{K_w}/I_{K_w},\texttt{M}_{g,\mathscr{O}}(\mathfrak{P})^{I_{K_w}})\]
of $H^{1}(K_{w},\texttt{M}_{g,\mathscr{O}}(\mathfrak{P}))$
(resp., is zero). 

\subsubsection{Discrete and compact Selmer groups}\label{Sec.DCSelmer}
Let $S$ be a positive squarefree   integer and let 
$\mathfrak{P}$ be an ideal of $\iw_\mathscr{O}$. The \emph{discrete Selmer group}
\[
              \mathrm{Sel}_{\varepsilon}^{S}(K,A_{g,\mathscr{O}}(\mathfrak{P}))\subset{}H^{1}(K,A_{g,\mathscr{O}}(\mathfrak{P}))
\]
is defined to be the  $\mathscr{O}_{\mathfrak{P}}$-module of global cohomology classes in 
$H^{1}(K,A_{g,\mathscr{O}}(\mathfrak{P}))$ which are 
\begin{itemize}
\item[$\bullet$] $\varepsilon$-finite at primes dividing $p$;
 
\item[$\bullet$]   ordinary at primes dividing $LN^{-}$;

\item[$\bullet$]  trivial at primes dividing $SN^{+}$;

\item[$\bullet$] finite outside $SLNp$.
\end{itemize}
The \emph{compact Selmer group}
\[
            \mathfrak{Sel}^{\varepsilon}_{S}(K,T_{g,\mathscr{O}}(\mathfrak{P}))\subset{}H^{1}(K,T_{g,\mathscr{O}}(\mathfrak{P}))
\]
is   the 
$\mathscr{O}_{\mathfrak{P}}$-module of global cohomology classes in $H^{1}(K,T_{g,\mathscr{O}}(\mathfrak{P}))$ which are

\begin{itemize}
\item[$\bullet$] $\varepsilon$-finite at  primes dividing $p/\text{g.c.d.}(S,p)$;
 
\item[$\bullet$] ordinary at primes dividing $LN^{-}/\text{g.c.d.}(S,LN^{-})$;

\item[$\bullet$] finite outside  $SLNp$.
\end{itemize}
Write $\mathrm{Sel}_{\varepsilon}(K,A_{g,\mathscr{O}}(\mathfrak{P}))$
and   $\mathfrak{Sel}^{\varepsilon}(K,T_{g,\mathscr{O}}(\mathfrak{P}))$
as shorthands for $\mathrm{Sel}^{1}_{\varepsilon}(K,A_{g}(\mathfrak{P}))$ and 
$\mathfrak{Sel}^{\varepsilon}_{1}(K,T_{g,\mathscr{O}}(\mathfrak{P}))$ respectively. 
If $\mathfrak{P}$ is the zero ideal, so that $T_{g,\mathscr{O}}(\mathfrak{P})=\mathbf{T}_{g,\mathscr{O}}$
and $A_{g,\mathscr{O}}(\mathfrak{P})=\mathbf{A}_{g,\mathscr{O}}$, set 
\[
            \mathrm{Sel}^{S}_{\varepsilon}(K_{\infty},A_{g,\mathscr{O}})=\mathrm{Sel}^{S}_{\varepsilon}(K,\mathbf{A}_{g,\mathscr{O}})\ \ \ \text{and}\ \ \ 
           \mathfrak{Sel}^{\varepsilon}_{S}(K_{\infty},T_{g,\mathscr{O}})=\mathfrak{Sel}^{\varepsilon}_{S}(K,\mathbf{T}_{g,\mathscr{O}}). 
\] Note that if $p\mid S$ then $\mathrm{Sel}^{S}_{\varepsilon}(K,\mathbf{A}_{g,\mathscr{O}})=\mathrm{Sel}^{S}(K,\mathbf{A}_{g,\mathscr{O}})$ and 
$\mathfrak{Sel}^{\varepsilon}_{S}(K,\mathbf{T}_{g,\mathscr{O}})=\mathfrak{Sel}_{S}(K,\mathbf{T}_{g,\mathscr{O}})$. 

\subsection{Local properties}\label{locsec}
If $E$ has (good) ordinary reduction at $p$, set $\varepsilon=\emptyset$ and 
$H^{1}_{\mathrm{fin},\varepsilon}=H^{1}_{\mathrm{fin}}$.
If $E$ has (good) supersingular reduction at $p$,
let $\varepsilon$ denote either $+$ or $-$.  

\subsubsection{Primes dividing $p$}\label{locsecp} Fix  a prime  $\mathfrak{p}$  of $K$ dividing  $p$. We first investigate local properties of points, and then we consider finite and singular subgroups. 

\begin{proposition}\label{hypext} For every nonnegative integer $n$, the restriction map induces  an isomorphism 
\[
                E(K_{n,\mathfrak{p}})_{\varepsilon}\otimes_{\Z}\divp\cong{}\big(E(K_{\infty,\mathfrak{p}})_{\varepsilon}
                \otimes_{\Z}\divp\big)[\omega_{n}^{\varepsilon}].
\]
\end{proposition}

\begin{proof}
If $E/\Q_{p}$ has good supersingular reduction, 
this follows from Theorem \ref{prop:formal-group}.  
More precisely, the Pontrjagin dual of the restriction map
$$E(K_{n,\mathfrak{p}})_{\varepsilon}\otimes_{\Z}\divp\lra{}\big(E(K_{m,\mathfrak{p}})_{\varepsilon}
                \otimes_{\Z}\divp\big)[\omega_{n}^{\varepsilon}]$$ 
is a surjective morphism of $\Lambda_{\mathfrak{p}}$-modules, for all integers $m\ges n$. Since Theorem $\ref{prop:formal-group}$ implies that its source 
and target are finite free $\Z_{p}$-modules of the same rank
(indeed both are isomorphic to $\Lambda_{\mathfrak{p}}/\omega_{n}^{\varepsilon}$), it is an isomorphism. 

Assume that $E/\Q_{p}$ has good ordinary reduction and consider the restriction maps
\[
              r_{n} : H^{1}(K_{n,\mathfrak{p}},A_{f})\lfre{}H^{1}(K_{\infty,\mathfrak{p}},A_{f})[\omega_{n}]
\]
and
\[              r_{n}^{\mathrm{sing}} :     \frac{H^{1}(K_{n,\mathfrak{p}},A_{f})}{E(K_{n,\mathfrak{p}})\otimes_{\Z}\divp}
                   \lfre{}\frac{H^{1}(K_{\infty,\mathfrak{p}},A_{f})}{E(K_{\infty,\mathfrak{p}})\otimes_{\Z}\divp}.
\]
Lemma 3.4 of \cite{Gr-1} proves that the kernel of $r_{n}^{\mathrm{sing}}$
has  cardinality  $|E(\F_{\mathfrak{p}})_{p}|^{2}$. 
(Loc.\ cit.\ considers  the cyclotomic $\Z_{p}$-extension $F_{\infty}/F$ of a 
finite extension  $F/\Q_{p}$,  but the argument  works for every $\Z_{p}$-extension $F_{\infty}/F$
such that the inertia subgroup of $\mathrm{Gal}(F_{\infty}/F)$ has finite index, cf.\ \cite[Proposition 2.4]{Gr-1}.)
The  inflation-restriction sequence shows that  $r_{n}$ is surjective and that its kernel is isomorphic to a quotient 
of $H^{0}(K_{\infty,\mathfrak{p}},A_{f})=E(K_{\infty,\mathfrak{p}})_{p^{\infty}}$. 
Assumption $\ref{controlass}(4)$ then implies that $r_{n}$ is an isomorphism and that $r_{n}^{\mathrm{sing}}$ is injective.
The statement follows.
\end{proof}

Fix an ideal $\mathfrak{P}$ of $\iw_{\mathscr{O}}$ generated by a regular sequence. 
\begin{proposition}\label{freeloc}  \hfill
\begin{enumerate}
\item 
$H^{1}_{\mathrm{fin},\varepsilon}(K_{\mathfrak{p}},T_{f,\mathscr{O}}(\mathfrak{P}))$  and 
$H^{1}_{\mathrm{sing},\varepsilon}(K_{\mathfrak{p}},T_{f,\mathscr{O}}(\mathfrak{P}))$ 
are free $\iw_\mathscr{O}/\mathfrak{P}$-modules of rank $[K_{\mathfrak{p}}:\Q_{p}]$.
\item 
$H^{1}_{\mathrm{fin},\varepsilon}(K_{\mathfrak{p}},A_{f,\mathscr{O}}(\mathfrak{P}))$
and $H^{1}_{\mathrm{sing},\varepsilon}(K_{\mathfrak{p}},A_{f,\mathscr{O}}(\mathfrak{P}))$
are co-free $\iw_\mathscr{O}/\mathfrak{P}$-modules of rank $[K_{\mathfrak{p}}:\Q_{p}]$.
\end{enumerate}
\end{proposition}

\begin{proof}
Since $H^{1}_{\mathrm{fin},\varepsilon}(K_{\mathfrak{p}},T_{f,\mathscr{O}}(\mathfrak{P}))$
and $H^{1}_{\mathrm{sing},\varepsilon}(K_{\mathfrak{p}},T_{f,\mathscr{O}}(\mathfrak{P}))$ are isomorphic to the Pontryagin duals 
of $H^{1}_{\mathrm{sing},\varepsilon}(K_{\mathfrak{p}},A_{f,\mathscr{O}}(\mathfrak{P}^{\iota}))^{\iota}$
and $H^{1}_{\mathrm{fin},\varepsilon}(K_{\mathfrak{p}},A_{f,\mathscr{O}})(\mathfrak{P}^{\iota}))^{\iota}$
respectively, it is sufficient to prove (2).
The proof is divided into three  steps. 

\emph{Step 1.} If $\mathscr{D}_{\iw}$ denotes the Pontrjagin dual of $\iw$, one has isomorphisms of $\iw$-modules 
\begin{equation}\label{eq:claim112}\begin{split}
               H^{1}_{\mathrm{fin},\varepsilon}(K_{\mathfrak{p}},\mathbf{A}_{f})\cong\mathscr{D}_{\iw}^{[K_{\mathfrak{p}}:\Q_{p}]}\ \ \ \text{and}\ \ \ 
                              H^{1}_{\mathrm{sing},\varepsilon}(K_{\mathfrak{p}},\mathbf{A}_{f})\cong\mathscr{D}_{\iw}^{[K_{\mathfrak{p}}:\Q_{p}]}.
                              \end{split}
\end{equation}

If $E/\Q_{p}$ has good supersingular reduction and $p$ is splits
in $K/\Q_{p}$,
(the duals of) Equations \eqref{eq:claim112} are proved in  
Propositions 4.16 of  \cite{Io-Po}, which in turn is a slight generalisation of Theorem 
6.2 of \cite{Kobss} (see also \cite[Proposition 9.2]{Kobss}). 
If $E/\Q_{p}$ has good supersingular reduction and $p$ is inert 
in $K/\Q_{p}$, this is a consequence of Rubin's conjecture proved in 
\cite{BKO1}: if, as in Section \ref{efinsub}, $O$ denotes the valuation ring of $\Psi=K_\mathfrak{p}$, 
it is proved in \cite{Rub-EllHeeg} that $H^{1}_{\mathrm{fin},\varepsilon}(K_{\mathfrak{p}},\mathbf{A}_{f})$ is a co-free $\iw_O$-module of rank one.
Moreover \cite{BKO1} proves that (as conjectured in \cite{Rub-EllHeeg}) $H^{1}(K_{\mathfrak{p}},\mathbf{A}_{f})$ is the direct sum of  
$H^{1}_{\mathrm{fin},+}(K_{\mathfrak{p}},\mathbf{A}_{f})$ and $H^{1}_{\mathrm{fin},-}(K_{\mathfrak{p}},\mathbf{A}_{f})$.
The statement follows.

If $E/\Q_{p}$ has good ordinary reduction,  
the representation $T_{f}$ is ordinary at $p$, \emph{i.e.} there exists a short exact sequence of 
$\Z_{p}[G_{\Q_{p}}]$-modules 
\[
             0\lfre{}T_{f}^{\bullet}\lfre{}T_{f}\lfre{}T_{f}^{\circ}\lfre{}0,
\]
arising from the reduction modulo $p$ on $E(\bar{\Q}_{p})$.
More precisely let $\alpha,\beta\in{}\Z_{p}$
be the roots of the Hecke polynomial 
$X^{2}-a_{p}(E)X+p$. 
Since $a_{p}(E)$ is a $p$-adic unit, one can assume $\alpha\in{}\Z_{p}^{\ast}$
and $\beta\in{}p\Z_{p}$.
Then $T_{f}^{\bullet}\cong{}\Z_{p}(\chi_{\mathrm{cyc}}\cdot{}\psi^{-1})$ and $T_{f}^{\circ}\cong{}\Z_{p}(\psi)$,
where $\psi: G_{\Q_{p}}\fre{}\Z_{p}^{\ast}$ is the unramified  character which sends an arithmetic 
Frobenius to  $\alpha$.
Set $\mathbf{T}_{f}^{\star}=T_{f}^{\star}\otimes_{\Z_{p}}\iw(\epsilon_{\infty}^{-1})$ for $\star\in \{\bullet,\circ\}$, 
so that 
there is an  exact sequence of $\iw[G_{K}]$-modules 
$\mathbf{T}_{f}^{\bullet}\hookrightarrow{}\mathbf{T}_{f}\twoheadrightarrow{}\mathbf{T}_{f}^{\circ}$.
According to a result of  Greenberg (cf. Proposition 2.4 of \cite{Gr-1}) 
\begin{equation}\label{eq:ord1}
            H^{1}_{\mathrm{fin}}(K_{\mathfrak{p}},\mathbf{T}_{f})=\mathrm{Image}
            \big(H^{1}(K_{\mathfrak{p}},\mathbf{T}_{f}^{\bullet})\lfre{}H^{1}(K_{\mathfrak{p}},\mathbf{T}_{f})\big)
            \cong{}H^{1}(K_{\mathfrak{p}},\mathbf{T}_{f}^{\bullet}). 
\end{equation}
Let $I$ be the augmentation ideal of $\iw$. Since $\mathbf{T}_{f}^{\bullet}/I=T_{f}^{\bullet}$
and $H^{0}(K_{\mathfrak{p}},T_{f}^{\bullet})=0$, one has $H^{1}(K_{\mathfrak{p}},\mathbf{T}_{f}^{\bullet})[I]=0$.
Moreover $H^{1}(K_{\mathfrak{p}},\mathbf{T}_{f}^{\bullet})/I$ is a free $\Z_{p}$-module, because it is isomorphic to a submodule of 
the free $\Z_{p}$-module $H^{1}(K_{\mathfrak{p}},T_{f}^{\bullet})$.
This implies that $H^{1}(K_{\mathfrak{p}},\mathbf{T}_{f}^{\bullet})$ is a free $\iw$-module of rank $[K_{\mathfrak{p}}:\Q_{p}]$,
hence so is $H^{1}_{\mathrm{fin}}(K_{\mathfrak{p}},\mathbf{T}_{f})$ by Equation \eqref{eq:ord1}.
The short exact sequence $\mathbf{T}_{f}^{\bullet}\hookrightarrow{}\mathbf{T}_{f}\twoheadrightarrow{}\mathbf{T}_{f}^{\circ}$
and Equation $(\ref{eq:ord1})$ induce an exact sequence of $\iw$-modules  
\[
             0\lfre{}H^{1}_{\mathrm{sing}}(K_{\mathfrak{p}},\mathbf{T}_{f})
             \lfre{}H^{1}(K_{\mathfrak{p}},\mathbf{T}_{f}^{\circ})\lfre{}H^{2}(K_{\mathfrak{p}},\mathbf{T}_{f}^{\bullet}).
\]
By Hypothesis $\ref{controlass}(4)$, 
$\psi(\mathrm{Frob}_\mathfrak{p})\not\equiv1\pmod{p}$, hence 
$H^2(K_\mathfrak{p},T_f^\bullet\otimes_{\Z_p}\mathbf{F}_p)=0$. This implies that 
$H^{2}(K_{\mathfrak{p}},T_{f}^{\bullet})$ vanishes, hence so does  
$H^{2}(K_{\mathfrak{p}},\mathbf{T}_{f}^{\bullet})$ by Nakayama's Lemma. It follows that $H^{1}_{\mathrm{sing}}(K_{\mathfrak{p}},\mathbf{T}_{f})$
is isomorphic to $H^{1}(K_{\mathfrak{p}},\mathbf{T}_{f}^{\circ})$. 
Another application of Hypothesis $\ref{controlass}(4)$ gives
$H^0(K_\mathfrak{p},T_f^\circ\otimes\Q_p/\Z_p)=0$, which implies that 
$H^{1}(K_{\mathfrak{p}},T_{f}^{\circ})$ is a free $\Z_{p}$-module. As above one deduces that 
$H^{1}(K_{\mathfrak{p}},\mathbf{T}_{f}^{\circ})$ is a free $\iw$-module of rank $[K_{\mathfrak{p}}:\Q_{p}]$,
hence so is $H^{1}_{\mathrm{sing}}(K_{\mathfrak{p}},\mathbf{T}_{f})$.

\emph{Step 2.} The inclusion $A_{f,\mathscr{O}}(\mathfrak{P})\fre{}\mathbf{A}_{f,\mathscr{O}}$
induces isomorphisms of $\iw_\mathscr{O}/\mathfrak{P}$-modules 
\[H^{1}(K_{\mathfrak{p}},A_{f,\mathscr{O}}(\mathfrak{P}))\cong{}
H^{1}(K_{\mathfrak{p}},\mathbf{A}_{f,\mathscr{O}})[\mathfrak{P}]\ \ \ \text{and}\ \ \ 
H^{1}_{\mathrm{fin},\varepsilon}(K_{\mathfrak{p}},A_{f,\mathscr{O}}(\mathfrak{P}))\cong{}
H^{1}_{\mathrm{fin},\varepsilon}(K_{\mathfrak{p}},\mathbf{A}_{f,\mathscr{O}})[\mathfrak{P}].\]

If $E$ has good supersingular reduction the $G_{K_\mathfrak{p}}=\mathrm{Gal}(\bar{K}_{\mathfrak{p}}/K_{\mathfrak{p}})$-representation $E_{p}$
is irreducible (see for example \cite[Proposition 12]{Serre-finord}). If $E/\Q_{p}$ has good ordinary reduction,
the kernel of the reduction modulo $p$ on $E(K_{\mathfrak{p}})$
is isomorphic to $\F_{\!p}(1)$ as a representation of the inertia subgroup of $G_{K_\mathfrak{p}}$.
Then Hypothesis $\ref{controlass}(4)$ implies that $H^{0}(K_{\mathfrak{p}},E_{p})$
vanishes in all cases. Because $\mathbf{A}_{f}[\mathfrak{m}_{\iw}]$ is isomorphic to $E_{p}$ this gives 
$H^{0}(K_{\mathfrak{p}},\mathbf{A}_{f})=0$, and therefore $H^{0}(K_{\mathfrak{p}},\mathbf{A}_{f,\mathscr{O}})=0$, 
which in turn easily implies that
$H^{1}(K_{\mathfrak{p}},A_{f,\mathscr{O}}(\mathfrak{P}))$ is isomorphic to the $\mathfrak{P}$-torsion submodule of 
$H^{1}(K_{\mathfrak{p}},\mathbf{A}_{f,\mathscr{O}})$. 
The claim follows directly from this and the definitions.

\emph{Step 3.} 
Since $H^{1}(K_{\mathfrak{p}},\mathbf{A}_{f})$ is a co-free $\iw$-module of rank $2[K_{\mathfrak{p}}:\Q_{p}]$
(cf.\ Step 1), it follows from Step 2 and the flatness of $\mathscr{O}/\Z_p$ 
that $H^{1}(K_{\mathfrak{p}},A_{f,\mathscr{O}}(\mathfrak{P}))$
is a co-free $\iw_\mathscr{O}/\mathfrak{P}$-module of rank $2[K_{\mathfrak{p}}:\Q_{p}]$. 
To conclude the proof it is then sufficient to show  that 
$H^{1}_{\mathrm{fin},\varepsilon}(K_{\mathfrak{p}},A_{f,\mathscr{O}}(\mathfrak{P}))$
is a co-free $\iw_\mathscr{O}/\mathfrak{P}$-module of rank $[K_{\mathfrak{p}}:\Q_{p}]$. 
This is a consequence of the previous two steps. 
\end{proof}

\begin{cor}\label{local control} Let $\mathscr{F}$ denote one of the symbols $\emptyset$, \emph{<<fin,$\varepsilon$>>} or \emph{<<sing,$\varepsilon$>>}.
\begin{enumerate}
\item  The projection $\mathbf{T}_{f,\mathscr{O}}\fre{}T_{f,\mathscr{O}}(\mathfrak{P})$
induces an isomorphism 
$H^{1}_{\mathscr{F}}(K_{\mathfrak{p}},\mathbf{T}_{f,\mathscr{O}})/\mathfrak{P}
\cong{}H^{1}_{\mathscr{F}}(K_{\mathfrak{p}},T_{f,\mathscr{O}}(\mathfrak{P}))$.

\item The inclusion $A_{f,\mathscr{O}}(\mathfrak{P})\fre{}\mathbf{A}_{f,\mathscr{O}}$
induces an isomorphism $H^{1}_{\mathscr{F}}(K_{\mathfrak{p}},A_{f,\mathscr{O}}(\mathfrak{P}))\cong{}
H^{1}_{\mathscr{F}}(K_{\mathfrak{p}},\mathbf{A}_{f,\mathscr{O}})[\mathfrak{P}]$.
\end{enumerate}
\end{cor}
\begin{proof} The second sentence has been proved in \emph{Step 2} of the proof of Proposition \ref{freeloc}, and the first sentence follows by duality.
\end{proof}

We end this section by proving (following arguments of \cite{BD-Kim}) that the $\varepsilon$-finite local conditions are suitably self-dual.
Set $\Psi=K_{\mathfrak{p}}$ and $I_{n}=\omega_{n}\cdot{}\iw$, so that $T_{g}(I_{n})=T_{g}\otimes{}\Z_{p}[G_{n}]$ (with $g$ in $G_{K}$
acting as $g\otimes{}g^{-1}$), and Shapiro's isomorphism identifies 
$H^{1}(\Psi,T_{g}(I_{n}))$ with $H^{1}(\Psi_{n},T_{g})$.
Similarly Shapiro's lemma identifies $H^{1}(\Psi,A_{g}(I_{n}))$ with $H^{1}(\Psi_{n},A_{g})$.
The Weil pairing then defines a perfect pairing 
\begin{equation}\label{eq:painv}
     [\cdot,\cdot]_{n}=[\cdot,\cdot]_{\mathfrak{p},n} : H^{1}(\Psi,T_{g}(I_{n}))\times{}H^{1}(\Psi,T_{g}(I_{n}))\lfre{}\Z/p^{k},
\end{equation}
such that $[\lambda\cdot{}x,y]_{n}=[x,\iota(\lambda)\cdot{}y]_{n}$ for each $\lambda$ in $\iw$ and each $x$ and $y$ in $H^{1}(\Psi,T_{g}(I_{n}))$.

\begin{lemma}\label{lele self} $H^{1}_{\mathrm{fin},\varepsilon}(\Psi,T_{g}(I_{n}))$ is its own orthogonal complement under $[\cdot,\cdot]_{n}$.
\end{lemma}
\begin{proof} 
Set $L_{n}^{\ast}=H^{1}_{\mathrm{fin},\varepsilon}(\Psi,T_{g}(I_{n}))$. By definition (after identifying $A_{g}$ with $T_{g}$ via the Weil pairing, and $H^{1}(\Psi_{n},A_{g})$
with the $\omega_{n}$-torsion submodule of $H^{1}(\Psi_{\infty},A_{g})$ via restriction), it is the orthogonal complement of 
$L_{n}=\big(\mathbf{E}(\Psi_{\infty})_{\varepsilon}\otimes\divp\big)[\omega_{n},p^{k}]=\big(\mathbf{E}(\Psi_{\infty})_{\varepsilon}/p^{k}\big)[\omega_{n}]$
under $[\cdot,\cdot]_{n}$. In addition, $L_{n}$ and $L_{n}^{\ast}$ have the same cardinality (cf.\ Proposition $\ref{freeloc}$), hence it is sufficient to prove the claim 
\begin{equation}\label{eq:cl23}
    \left[L_{n},\big(\mathbf{E}(\Psi_{m})_{\varepsilon}/p^{k}\big)[\omega_{n}]\right]_{n}=0
\end{equation}
for each integer $m\ges{}n$. 

We start by proving Equation $(\ref{eq:cl23})$ in the special case $n=m$.
Since $L_{n}$ is co-free over $\iw_{n,k}$ (cf.\ Proposition $\ref{freeloc}$), $L_{n}=L_{n}[\omega_{n}^{\varepsilon}]+L_{n}[\omega_{n}/\omega_{n}^{\varepsilon}]
=L_{n}[\omega_{n}^{\varepsilon}]+\omega_{n}^{\varepsilon}\cdot{}L_{n}$. As $\iota(\omega_{n}^{\varepsilon})$ kills $\mathbf{E}(\Psi_{n})_{\varepsilon}/p^{k}$, 
one has $[\omega_{n}^{\varepsilon}\cdot{}L_{n},\mathbf{E}(\Psi_{n})_{\varepsilon}/p^{k}]_{n}%=[L_{n},\iota(\omega_{n}^{\varepsilon})\cdot{}\mathbf{E}(\Psi_{n})_{\varepsilon}/p^{k}]_{n}
=0$.
Moreover,  
Proposition $\ref{hypext}$ gives 
$L_{n}[\omega_{n}^{\varepsilon}]=\mathbf{E}(\Psi_{n})_{\varepsilon}/p^{k}\subset{}\mathbf{E}(\Psi_{n})/p^{k}$, and the latter is orthogonal to itself under $[\cdot,\cdot]_{n}$,
hence $[L_{n}[\omega_{n}^{\varepsilon}],\mathbf{E}(\Psi_{n})_{\varepsilon}/p^{k}]_{n}=0$.
This proves  
\begin{equation}\label{eq:clcl23}
         [L_{n}, \mathbf{E}(\Psi_{n})_{\varepsilon}/p^{n}]_{n}=0.
\end{equation}

Fix now an isomorphism of $\iw$-modules $H^{1}(\Psi_{\infty},A_{g})\simeq{}\mathscr{D}^{a}\oplus{}\mathscr{D}^{a}$, where $\mathscr{D}=\Hom{\mathrm{cts}}(\iw/p^{k},\divp)$ is the Pontjagin dual of $\iw/p^{k}$
and $a=[\Psi:\Q_{p}]$, which identifies $\mathbf{E}(\Psi_{\infty})_{\varepsilon}/p^{k}$ with the first copy of $\mathscr{D}^{a}$  (cf.\ Proposition $\ref{freeloc}$).
Set $\mathscr{D}_{n}=\mathscr{D}[\omega_{n}]=\Hom{\Z_{p}}(\iw_{n}/p^{k},\divp)$, and for each $m\ges{}n$ let 
$\mathrm{C}_{m/n} : \mathscr{D}_{m}\rightarrow{}\mathscr{D}_{n}$ be the dual of the injective map $\iw_{n}/p^{k}\rightarrow{}\iw_{n}/p^{k}$ sending $\lambda+(\omega_{n},p^{k})$ to 
$(\omega_{m}/\omega_{n})\cdot{}\lambda+(\omega_{m},p^{k})$,
for each $\lambda$ in $\iw$. Then the surjective map $\mathrm{Cor}_{m/n}=\mathrm{C}_{m/n}^{a}\oplus{}\mathrm{C}_{m/n}^{a}$   
corresponds to the corestriction map $H^{1}(\Psi_{m},A_{g})\rightarrow{}H^{1}(\Psi_{n},A_{g})$ under the fixed isomorphism, once one identifies 
$H^{1}(\Psi_{s},A_{g})$ with the $\omega_{s}$-torsion submodule of $H^{1}(\Psi_{\infty},A_{g})$ (cf.\ Corollary $\ref{local control}$).
It follows that for every $m\ges{}n$ one has 
\[
      \mathrm{Cor}_{m/n}(L_{m})=L_{n},
\]
which combined with Equation $(\ref{eq:clcl23})$ proves the Claim $(\ref{eq:cl23})$: for each $y$ in $\big(\mathbf{E}(\Psi_{m})_{\varepsilon}/p^{k}\big)[\omega_{n}]$ one has
\[
           [L_{n},y]_{n}=[\mathrm{Cor}_{m/n}(L_{m}),y]_{n}=[L_{m},y]_{m}=0.
\]
(Here one identifies $H^{1}(\Psi_{n},A_{g})$ with the $\omega_{n}$-torsion submodule of $H^{1}(\Psi_{m},A_{g})$ via the restriction map,
and uses the fact that restriction and corestriction are adjoint to each other under the Tate pairing.)
\end{proof}

\subsubsection{Primes not dividing $p$} 

Recall the notation introduced at the beginning of Section \ref{selsec}: $k$ is a positive integer or the symbol $\infty$, 
$L\in{}\mathscr{S}_{k}$ is a square-free product of $k$-admissible primes, 
and $g$  is either the level raising of $f_k=f\pmod{p^k}$ at $L$ if $k$ is an integer, or $g=f$ if $k=\infty$.
Let $\mathfrak{P}$ be a principal ideal of $\iw_\mathscr{O}$ generated by a regular sequence.
%which is either zero or generated by an element coprime with $p$. 
%Let $1\les{}t\les{}k$ and denote by 
%$h\in{}S_{2}(N^{+},LN^{-};\Z_{p}/p^{t})$  the reduction of $g$ modulo $p^{t}$.

\begin{lemma}\label{control nminus} \hfill
\begin{enumerate}
\item 
Let $\ell$ be a  prime dividing  $N^{-}$. Then the morphism 
\[
    H^{1}(K_{\ell},A_{g,\mathscr{O}}^{[\ell]}(\mathfrak{P}))\lfre{}H^{1}(K_{\ell},\mathbf{A}_{g,\mathscr{O}}^{[\ell]})
\]
induced by the inclusion $A_{g,\mathscr{O}}^{[\ell]}(\mathfrak{P})=\mathbf{A}_{g,\mathscr{O}}^{[\ell]}[\mathfrak{P}]\hlfre{}\mathbf{A}_{g,\mathscr{O}}^{[\ell]}$   is injective.

\item Let $\ell\in{}\mathscr{S}_{k}$ be a $k$-admissible prime.
Then the morphisms 
\[   
H^{1}(K_{\ell},A_{h,\mathscr{O}}^{(\ell)}(\mathfrak{P}))\lfre{}H^{1}(K_{\ell},\mathbf{A}_{g,\mathscr{O}}^{(\ell)})\ \ \ \text{and}\ \ \ 
H^{1}(K_{\ell},A_{h,\mathscr{O}}^{[\ell]}(\mathfrak{P}))\lfre{}H^{1}(K_{\ell},\mathbf{A}_{g,\mathscr{O}}^{[\ell]})
\]
induced by the inclusions $A_{g,\mathscr{O}}^{(\ell)}(\mathfrak{P})\hlfre{}{}\mathbf{A}_{g,\mathscr{O}}^{(\ell)}$ and 
$A_{g,\mathscr{O}}^{[\ell]}(\mathfrak{P})\hlfre{}\mathbf{A}_{g,\mathscr{O}}^{[\ell]}$   respectively are injective.
\end{enumerate}
\end{lemma}
\begin{proof} (1).  
Let  $\mathscr{D}_{\iw_\mathscr{O}}$ denote the Pontrjagin dual of $\iw_\mathscr{O}$.
Then  $A_{g,\mathscr{O}}^{[\ell]}(\mathfrak{P})$ is isomorphic to $\mathscr{D}_{\iw_{\mathscr{O}}}[\mathfrak{P},p^{k}]$ with trivial Galois action 
% is isomorphic to the Kummer 
%dual of $T_{g,\mathscr{O}}^{(\ell)}(\mathfrak{P}^{\iota})^{\iota}
%\cong{}\iw_\mathscr{O}/\mathfrak{P}\otimes_{\Z_{p}}\Z_{p}(1)$
(cf.\ Section \ref{selsec}), hence
$H^{1}(K_{\ell},A_{g,\mathscr{O}}^{[\ell]}(\mathfrak{P}))$
is isomorphic to $\Hom{\mathrm{cont}}(G_{K_\ell},\mathscr{D}_{\iw_\mathscr{O}}[\mathfrak{P},p^{k}])$.
Similarly, $H^{1}(K_{\ell},\mathbf{A}^{[\ell]}_{g})\cong{}\Hom{\mathrm{cont}}(G_{K_\ell},\mathscr{D}_{\iw_\mathscr{O}}[p^{k}])$,
and the map
$H^{1}(K_{\ell},A_{h,\mathscr{O}}^{[\ell]}(\mathfrak{P}))\rightarrow H^{1}(K_{\ell},\mathbf{A}^{[\ell]}_{g,\mathscr{O}})$ corresponds 
to the morphism induced  by the inclusion of $\mathscr{D}_{\iw}[\mathfrak{P},p^{t}]$ in $\mathscr{D}_{\iw}[p^{k}]$.

(2) The injectivity of the map $H^{1}(K_{\ell},A_{g,\mathscr{O}}^{[\ell]}(\mathfrak{P}))\fre{}H^{1}(K_{\ell},\mathbf{A}_{g,\mathscr{O}}^{[\ell]})$
is proved as above. 
The representation $A_{g,\mathscr{O}}^{(\ell)}(\mathfrak{P})$ is isomorphic to the Kummer dual of 
$T_{g,\mathscr{O}}^{[\ell]}(\mathfrak{P}^{\iota})^{\iota}\cong{}\iw_\mathscr{O}/(\mathfrak{P},p^{k})$. Since $p\nmid{\ell}^{2}-1$
because $\ell\in{}\mathscr{S}_{k}$,  
and $\mathscr{O}$ is a flat $\Z_{p}$-algebra, 
$H^{1}(K_{\ell},\iw_\mathscr{O}/(\mathfrak{P},p^{k}))\cong{}\iw_\mathscr{O}/(\mathfrak{P},p^{k})$
and  local Tate duality then  gives 
$H^{1}(K_{\ell},A_{g,\mathscr{O}}^{(\ell)}(\mathfrak{P},p^{k}))\cong{}\mathscr{D}_{\iw_\mathscr{O}}[\mathfrak{P},p^{k}]$. 
Similarly $H^{1}(K_{\ell},\mathbf{A}_{g,\mathscr{O}}^{(\ell)})\cong\mathscr{D}_{\iw_\mathscr{O}}[p^{k}]$, and the map $H^{1}(K_{\ell},A_{h,\mathscr{O}}^{(\ell)}(\mathfrak{P}))\rightarrow{}H^{1}(K_{\ell},\mathbf{A}_{g,\mathscr{O}}^{(\ell)})$
is identified with the inclusion $\mathscr{D}_{\iw_\mathscr{O}}[\mathfrak{P},p^{k}]\fre{}\mathscr{D}_{\iw_\mathscr{O}}[p^{k}]$.
\end{proof}

\begin{lemma}\label{control nplus} If $w$ is a prime of $K$ which divides $N^{+}$, then  $H^{j}(K_{w},A_{h,\mathscr{O}}(\mathfrak{P}))=0$ for every $j$.
\end{lemma}

\begin{proof}
Hypothesis $\ref{controlass}(5)$ and local Tate duality 
imply that $\mathbf{A}_{f}[\mathfrak{m}_{\iw}]=E_{p}$ is an acyclic $G_{w}$-module. Since $\mathscr O/\Z_p$ 
is flat, this in turn implies that $\mathbf{A}_{f,\mathscr{O}}[J]$ is an acyclic $G_{w}$-module for every 
ideal $J$ of $\iw$ generated by a regular sequence. Because $A_{h,\mathscr{O}}(\mathfrak{P})$
is isomorphic to $\mathbf{A}_{f,\mathscr{O}}[\mathfrak{P},p^{t}]$,
this concludes the proof. 
\end{proof}

\subsection{Control theorems} Let $\mathfrak{P}$ be an ideal of $\iw_\mathscr{O}$ generated by a regular sequence.

\begin{proposition}\label{global control} Let $S\in{}\mathscr{S}_{k}$ be a (possibly empty) squarefree  product 
of $k$-admissible primes.  
Then 
\[
\mathrm{Sel}_{\varepsilon}^{S}(K,A_{g,\mathscr{O}}(\mathfrak{P}))
\cong{}\mathrm{Sel}_{\varepsilon}^{S}(K,\mathbf{A}_{g,\mathscr{O}}    
)[\mathfrak{P}].
\]
\end{proposition}  
\begin{proof} %Identify as usual $A_{f,k}$ with $A_{g}$, hence $A_{f,t}$ with $A_{h}$.
Let $\gaun=\mathrm{Gal}(K_{SLNp}/K)$ be the Galois group of the maximal algebraic extension of 
$K$ which is unramified at every finite place $w\nmid{}SLNp$ of $K$. 
Since $H^{0}(\gaun,E_{p})$ vanishes  by Hypothesis $\ref{controlass}.(3)$, 
the same is true for the $\gaun$-module $E_p\otimes_{\Z_p}\mathscr O$, and therefore 
for every ideal $\mathfrak{P}$ of $\iw_\mathscr{O}$ 
generated by a regular sequence 
the natural map gives an isomorphism 
$H^{1}(\gaun,\mathbf{A}_{f,\mathscr{O}}[\mathfrak{P}])\cong{}H^{1}(\gaun,\mathbf{A}_{f,\mathscr{O}})[\mathfrak{P}]$
(cf.\ Step 2 in the proof of Proposition $\ref{freeloc}$).
%As a consequence  
%$H^{1}(\gaun,A_{h,\mathscr{O}}(\mathfrak{P}))$ is isomorphic to the $(\mathfrak{P},p^{t})$-torsion 
%submodule of $H^{1}(\gaun,\mathbf{A}_{g,\mathscr{O}})$. 
This implies that the natural map $\mathrm{Sel}_{\varepsilon}^{S}(K,A_{g,\mathscr{O}}(\mathfrak{P}))\fre{}
\mathrm{Sel}_{\varepsilon}^{S}(K,\mathbf{A}_{g,\mathscr{O}})[\mathfrak{P}]$ is injective, and its cokernel is isomorphic to 
a submodule of the kernel of 
\[
           \bigoplus_{w\mid p} H^{1}_{\mathrm{sing},\varepsilon}(K_{w},A_{g,\mathscr{O}}(\mathfrak{P}))
             \oplus{}\bigoplus_{w\mid\frac{LN^{-}}{
             (L,S)}}H^{1}(K_{w},A_{g,\mathscr{O}}^{[w]}(\mathfrak{P}))
\oplus{}\bigoplus_{w\mid SN^{+}}H^{1}(K_{w},A_{g,\mathscr{O}}(\mathfrak{P})) \qquad\qquad\]
             \[\qquad\qquad\longrightarrow
           \bigoplus_{w\mid p}H^{1}_{\mathrm{sing},\varepsilon}(K_{w},\mathbf{A}_{g,\mathscr{O}}) 
             \oplus{}\bigoplus_{w\mid\frac{LN^{-}}{(L,S)}}H^{1}(K_{w},\mathbf{A}_{g,\mathscr{O}}^{[w]})
             \oplus{}\bigoplus_{w\mid SN^{+}}H^{1}(K_{w},\mathbf{A}_{g,\mathscr{O}}).
\]
The  proposition then follows from Corollary $\ref{local control}(2)$,
Lemma $\ref{control nminus}$ and  Lemma $\ref{control nplus}$. 
\end{proof}

\begin{proposition}\label{speinj} Let $S$ be an integer coprime with $LNp$. The canonical map
\[\mathfrak{Sel}^{\varepsilon}_{S}(K,\mathbf{T}_{g,\mathscr{O}})\otimes_{\iw_\mathscr{O}}\iw_\mathscr{O}/\mathfrak{P}\longrightarrow
\mathfrak{Sel}^{\varepsilon}_{S}(K,T_{g,\mathscr{O}}(\mathfrak{P}))\] is injective. 
\end{proposition}
\begin{proof} 
Let $\gaun$ be the Galois group of the maximal algebraic extension of $K$
unramified outside $SLNp\infty$. Since   $\iw_\mathscr{O}/p^{k}$ has no nontrivial $\mathfrak{P}$-torsion 
the morphism 
\[H^{1}(\gaun,\mathbf{T}_{g,\mathscr{O}})\otimes_{\iw_\mathscr{O}}\iw_\mathscr{O}/\mathfrak{P}\longrightarrow
H^{1}(\gaun,\mathbf{T}_{g,\mathscr{O}}(\mathfrak{P}))\]
is injective. Moreover the cokernel of the inclusion $\mathfrak{Sel}^{\varepsilon}_{S}(K,\mathbf{T}_{g,\mathscr{O}})
\fre{}H^{1}(\gaun,\mathbf{T}_{g,\mathscr{O}})$ is isomorphic to a submodule of 
\begin{equation}\label{dirsum}
\bigoplus_{w\mid p}H^{1}_{\mathrm{sing},\varepsilon}(K_{w},\mathbf{T}_{g,\mathscr{O}})\oplus{}\bigoplus_{w\mid LN^-}H^{1}(K_{w},\mathbf{T}_{g,\mathscr{O}}^{[w]}).\end{equation} To prove the first part of the proposition 
it is then sufficient to show that each summand of the direct sum \eqref{dirsum}
has no non-trivial $\mathfrak{P}$-torsion.  
For $H^{1}_{\mathrm{sing},\varepsilon}(K_{w},\mathbf{T}_{g,\mathscr{O}})$ this is a consequence of Proposition $\ref{freeloc}$. 
Assume that $w=\ell\cdot{}\mathcal{O}_{K}$ for a rational prime $\ell$ dividing $LN^{-}$. In this case 
$\mathbf{T}_{g,\mathscr{O}}^{[w]}$  is isomorphic to $\iw_\mathscr{O}\otimes_{\Z_{p}}\Z_{p}/p^{k}$ (with trivial Galois action), hence 
$H^{1}(K_{w},\mathbf{T}_{g,\mathscr{O}}^{[w]})\cong{}H^{1}(K_{\ell},\Z_{p}/p^{k})\otimes_{\Z_{p}}\iw_\mathscr{O}$
and $H^{1}(K_{w},\mathbf{T}_{g,\mathscr{O}}^{[w]})[\mathfrak{P}]=0$. 
\end{proof}

In light of the applications to the definite and indefinite main conjectures, 
we also record the following result due to Mazur--Rubin and Howard. In the rest of this section, let $\mathfrak P$ be a height one prime ideal of $\Lambda$, and  
let $\mathscr{O}$ be the integral closure of $\Lambda/\mathfrak{P}$ in its fraction field. Consider the 
canonical map \begin{equation}\label{MRH1}
\mathfrak{Sel}^\varepsilon(K,\mathbf{T}_f)/\mathfrak P \longrightarrow 
\mathfrak{Sel}^\varepsilon(K,\mathbf{T}_{f,\mathscr{O}}(\mathfrak P)),\end{equation}
 induced by the composition $\mathrm{sp}_{\mathfrak{P}} : \mathbf{T}_f\rightarrow \mathbf{T}_f\otimes_{\Z_p}(\Lambda/\mathfrak{P})\rightarrow \mathbf{T}_{f,\mathscr{O}}(\mathfrak{P})$, and the map 
 \begin{equation}\label{MRH2}\mathrm{Sel}_\varepsilon(K,A_{f,\mathscr{O}}(\mathfrak P))\longrightarrow 
\mathrm{Sel}_\varepsilon(K,\mathbf{A}_f)[\mathfrak P]
\end{equation}
induced by the Kummer dual of  $\mathrm{sp}_{\mathfrak{P}^{\imath}} : \mathbf{T}_f\rightarrow \mathbf{T}_{f,\mathscr{O}}(\mathfrak{P}^{\imath})$. 
 
\begin{proposition}\label{propMRH}
There is a finite set $\Sigma$ of height one primes in $\Lambda$ such that for $\mathfrak P\not\in\Sigma$, 
the maps \eqref{MRH1} and \eqref{MRH2} 
have finite kernel and cokernels, of order bounded by a constant depending only on $[\mathscr O:\Lambda/\mathfrak P]$. 
\end{proposition}

\begin{proof} The proof proceeds as in \cite[Proposition 5.3.14]{MR-KolSys}, replacing the local results 
in loc. cit.  with Corollary \ref{local control} (for primes dividing $p$), and following 
the proof of \cite[Lemma 2.2.7]{How-HK} for the ordinary-type conditions.  
\end{proof}

\subsection{Global freeness}\label{glofree} Assume in this section that $k<\infty$. Let  $\bar{g}\in{}S_{2}(N^{+},LN^{-};\F_{\!p})$
be the reduction of $g$ modulo $p$, and let $\mathfrak{m}_{\iw}=(\omega_{1},p)\cdot{}\iw$ be the maximal ideal of $\iw$.
Define 
\[
       \mathrm{Sel}^{S}_{\Box}(K,A_{\bar{g}})\subset{}H^{1}(K,\mathbf{A}_{g}[\mathfrak{m}_{\iw}])
\]
to be the Selmer group obtained from $\mathrm{Sel}^{S}(K,A_{\bar{g}})$ by relaxing the local conditions at 
the primes of $K$ dividing $p$ (so that $\mathrm{Sel}^{S}(K,A_{\bar{g}})$
is the set of classes in $\mathrm{Sel}^{S}_{\Box}(K,A_{\bar{g}})$ which are finite at $p$.)

\begin{definition}
A \emph{freeing set relative to $g$} 
is an integer $S\in{}\mathscr{S}_{k}$ 
such that  $\mathrm{Sel}_{\Box}^{S}(K,A_{\bar{g}})=0$.
\end{definition}

A generalization of \cite[Theorem 3.2]{Be-Da-main} guarantees  the existence of infinitely freeing sets relative to $g$
(see also the proof of   \cite[Lemma 2.23]{Be-Da-der1},  and compare with \cite{KPW} where a correction including the ramification condition at primes $q\mid N^+$ is made).
Let $n$ denote either a nonnegative integer or $\infty$.
Set $\omega_{\infty}=0$, $I_{\mathscr{O},n}=\omega_{n}\cdot{}\iw_{\mathscr{O}}$, and $\iw_{{\mathscr{O}},n,k}=\iw_{\mathscr{O}}/(I_{\mathscr{O},n},p^{k})$, and  write  
\[\mathfrak{Sel}_{S}^{\varepsilon}(K_{n},T_{g,{\mathscr{O}}})
=\mathfrak{Sel}_{S}^{\varepsilon}(K,T_{g,{\mathscr{O}}}(I_{\mathscr{O},n}))\ \ \ \text{and}\ \ \
\mathfrak{Sel}_{Sp}(K_{n},T_{g,{\mathscr{O}}})
=\mathfrak{Sel}_{Sp}(K,T_{g,{\mathscr{O}}}(I_{\mathscr{O},n})),\]
where $\varepsilon=\emptyset$ if $E$ is ordinary at $p$, and $\varepsilon$ denotes one of the symbols $\emptyset$, $+$ of $-$ otherwise.

\begin{proposition}\label{global freeness} Let  $S$ be a freeing set relative to $g$ and set $\delta(S)=\#\{\mathrm{prime\ divisors\ of}\  S\}$. 
\begin{enumerate}
\item
The Selmer group   $\mathfrak{Sel}_{S}^{\varepsilon}(K_{n},T_{g,{\mathscr{O}}})$
is a free $\iw_{\mathscr{O},n,k}$-module of rank $\delta(S)$. 
\item The Selmer group   $\mathfrak{Sel}_{Sp}(K_{n},T_{g,{\mathscr{O}}})$
is a free $\iw_{\mathscr{O},n,k}$-module of rank $\delta(S)+2$. 
\end{enumerate}
\end{proposition} 

\begin{proof} If $E/\Q_{p}$ has  ordinary reduction  this is a  variant of 
\cite[Proposition 3.3]{Be-Da-main}
(see also  Section 3 of \cite{Be-Da-der1}). We use the computations of the preceding sections 
to give a proof which works in our more general setting.
Without loss of generality, we can (and will) assume $n\not=\infty$ and $\mathscr{O}=\Z_{p}$ throughout the proof. 
As usual, we then omit $\mathscr{O}$ from the notation. \vspace{1mm}

\emph{Step 1.} Set $\mathrm{Sel}_{\varepsilon}^{S}(K_{n},A_{g})
=\mathrm{Sel}_{\varepsilon}^{S}(K,A_{g}(I_{n}))$. 
%denote as above the Selmer group obtained from $\mathrm{Sel}^{S}(K,A_{g,\mathscr{O}}(I_{\mathscr{O},n}))$
%by relaxing the local conditions at the primes of $K$ dividing $p$. 
We show that 
\begin{equation}\label{eq:claim304050}
                \mathrm{Sel}_{\varepsilon}^{S}(K_{n},A_{g})=0.
\end{equation}
Indeed Proposition $\ref{global control}$ yields 
\[
    \mathrm{Sel}_{\varepsilon}^{S}(K,A_{\bar{g}})=\mathrm{Sel}_{\varepsilon}^{S}(K,\mathbf{A}_{g})
[\mathfrak{m}_{\iw}]=\mathrm{Sel}_{\varepsilon}^{S}(K_{n},A_{g})[\mathfrak{m}_{\iw}].
\] 
As (by assumption) $\mathrm{Sel}_{\varepsilon}^{S}(K,A_{\bar{g}})\subset{}\mathrm{Sel}^{S}_{\Box}(K,A_{\bar{g}})=0$,
Equation \eqref{eq:claim304050}  follows from Nakayama's Lemma. \vspace{1mm}

\emph{Step 2.} 
We show that 
\begin{equation}\label{eq:9102}
\mathfrak{Sel}_{Sp}(K,T_{\bar{g}})\cong{}\mathfrak{Sel}_{Sp}(K_{n},T_{g})[\mathfrak{m}_{\iw}].
\end{equation}
Denote by $K_{SLNp}$
the maximal algebraic extension of $K$
which is unramified outside $SLNp$ and by $\gaun_{s}$ the Galois group of $K_{SLNp}/K_{s}$,
for every $0\les{}s\les{}\infty$. 
If one identifies $H^{1}(\gaun_{0},T_{g}(I_{n}))$ with $H^{1}(\gaun_{n},T_{g})$
via the Shapiro isomorphism, then 
\[
         \mathfrak{Sel}_{Sp}(K_{n},T_{g})=\ker\Big(H^{1}(\gaun_{n},T_{g})\lfre{}\bigoplus_{\mathfrak{l}|LN^{-}}
            H^{1}(K_{n,\mathfrak{l}},T_{g}^{[\ell]})\Big),
\] 
where the direct sum is taken over the primes $\mathfrak{l}$ of $K_{n}$ which divide $LN^{-}$.
Hypothesis $\ref{controlass}.(3)$
guarantees that $H^{0}(K_{s},T_{f}/p^{r})$ vanishes for all $r\ges{}1$ and $s\ges{}0$. 
As a consequence the map 
\[
 H^{1}(\gaun_{0},T_{\bar{g}})\longrightarrow{}H^{1}(\gaun_{n},T_{g})[\mathfrak{m}_{\iw}]
\]
induced by restriction from $K$ to $K_{n}$ and the inclusion $p^{k-1} : T_{\bar{g}}=T_{f}/p\hlfre{}T_{f}/p^{k}=T_{g}$ 
is an isomorphism.  To prove Equation \eqref{eq:9102} it is then sufficient to verify that the map 
\[
\beta_{\mathfrak{l}} : H^{1}(K_{\ell},T_{\bar{g}}^{[\ell]})\lfre{}H^{1}(K_{n,\mathfrak{l}},T_{g}^{[\ell]})
\]
is injective for every prime $\mathfrak{l}$ of $K_{n}$ lying over $\ell|LN^{-}$.
The $G_{K_{\ell}}$-representation $T_{f,r}^{[\ell]}=T_{f}^{[\ell]}/p^{r}$ is isomorphic to $\Z/p^{r}$  for every $r\ges{}1$,
and one has $K_{n,\mathfrak{l}}=K_{\ell}$ (since $\ell$ splits completely in $K_{n}/K$). It follows that 
$H^{1}(K_{\ell},T_{f,r}^{[\ell]})=\Hom{\mathrm{cont}}(G_{K_\ell},\Z/p^{r})$,
and (via the fixed isomorphisms $T_{f,k}=T_{g}$ and $T_{f,1}=T_{\bar{g}}$) the map $\beta_{\mathfrak{l}}$ is (identified with) the injective morphism induced by the inclusion 
$p^{k-1} : \Z/p\hookrightarrow{}\Z/p^{k}$. \vspace{2mm}

\emph{Step 3.} 
We prove (2).
As the local conditions defining $\mathfrak{Sel}_{Sp}(K_{n},T_{g})$
are dual to those defining  $\mathrm{Sel}^{Sp}(K_{n},A_{g})$
(via Tate's duality), Theorem 2.19 of \cite{DDT}, Hypothesis $\ref{controlass}.(3)$
and Lemma $\ref{control nplus}$ yield
\[ 
               \frac{\#\mathrm{Sel}^{Sp}(K_{n},A_{g})}
               {\#\mathfrak{Sel}_{Sp}(K_{n},T_{g})}
               =\#T_{g}(I_{n})\cdot{}
               \prod_{w|Sp}\frac{\#H^{0}(K_{w},T_{g}(I_{n}))}{\#H^{1}(K_{w},T_{g}(I_{n}))}   
               \cdot{}
               \prod_{\ell|LN^{-}/(S,LN^{-})}\frac{\#H^{0}(K_{\ell},T_{g}(I_{n}))}{
               \#H^{1}_{\mathrm{ord}}(K_{\ell},T_{g}(I_{n}))}.
\]
Step $1$ implies that $\mathrm{Sel}^{Sp}(K_{n},T_{g})\subset{}\mathrm{Sel}^{S}_{\Box}(K_{n},T_{g})$ vanishes. 
For each $w\mid pSLN^{-}$, define
\[h_w=\frac{\#H^{0}(K_{w},T_{g}(I_{n}))}{\#H^{1}(K_{w},T_{g}(I_{n}))}\]
and let $h_\infty= \#T_{g}(I_{n}).$ 
For every prime 
$\ell$ dividing $S$, $T_{g}(I_{n})=T_{g}\otimes{}\iw_{n,k}$,
hence Lemma $\ref{decadm}$ and the local Euler characteristic formula give 
$h_w^{-1}=\#H^{2}(K_{\ell},T_{g}(I_{n}))=\#\iw_{n,k}.$ 
For every prime $\ell|L$ Lemma $\ref{decadm}$ similarly gives 
$h_w=1.$ 
If  $\ell|N^{-}$, considering the long exact cohomology sequence 
associated with $\mu_{p^{k}}\hookrightarrow{}T_{g}\twoheadrightarrow{}\Z/p^{k}$
one easily proves that $h_{\ell}=1$ even in this case. 
If $\mathfrak{p}$ is a prime dividing $p$ the local Euler characteristic formula shows that 
$h_{\mathfrak{p}}^{-1}=\#\iw_{n,k}^{2\cdot{}[K_{\mathfrak{p}}:\Q_{p}]}\cdot{}
\#H^{2}(K_{\mathfrak{p}},T_{g}(I_{n}))$.
On the other hand $H^{2}(K_{\mathfrak{p}},T_{g}(I_{n}))$ has the same cardinality 
as  $H^{0}(K_{\mathfrak{p}},A_{g}(I_{n}))$,
which vanishes since its $\mathfrak{m}_{\iw}$-torsion submodule is equal to $H^{0}(K_{\mathfrak{p}},E_{p})=0$ 
(thanks to Hypothesis $\ref{controlass}(4)$ in the ordinary case). The previous equation then  yields 
\begin{equation}\label{eq:91032}
               \#\mathfrak{Sel}_{Sp}(K_{n},T_{g})=h_{\infty}^{-1}\cdot{}\#\iw_{n,k}^{\delta(S)+4}
                =\#\iw_{n,k}^{\delta(S)+2}.
\end{equation}
Since $\iw/I_{n}$ is a Gorenstein local ring, it is isomorphic to
 $\Hom{\Z_{p}}(\iw/I_{n},\Z_{p})$ as a $\iw$-module, so that 
$\iw_{n,k}$ is isomorphic to its Pontrjagin dual $\mathscr{D}_{n,k}=\Hom{\Z_{p}}(\iw_{n,k},\Q_{p}/\Z_{p})$ 
as a $\iw$-module.
Taking $n=0$ and $k=1$ in Equation \eqref{eq:91032}  shows that   
$\mathfrak{Sel}_{Sp}(K,T_{\bar{g}})$ has dimension $\delta(S)+2$ over $\F_{\!p}$,
hence  $$\mathfrak{Sel}_{Sp}(K_{n},T_{g})[\mathfrak{m}_{\iw}]\cong{}\F_{\!p}^{\delta(S)+2}$$ by Step 2. 
By Nakayama's Lemma there exists then a surjective morphism  
\[
     \iw_{n,k}^{\delta(S)+2}\tlfre{}\Hom{\Z_{p}}(\mathfrak{Sel}_{Sp}(K_{n},T_{g}),\divp),
\]
which is an isomorphism by another application of Equation $(\ref{eq:91032})$. As a consequence 
$\mathfrak{Sel}_{Sp}(K_{n},T_{g})$ is isomorphic to $\mathscr{D}_{n,k}^{\delta(S)+2}\cong{}\iw_{n,k}^{\delta(S)+2}$, as was to be shown. \vspace{1mm}

\emph{Step 4.} There is an exact sequence of $\iw_\mathscr{O}$-modules 
\begin{equation}\label{eq}
           0\lfre{}\mathfrak{Sel}_{S}^{\varepsilon}(K_{n},T_{g})
           \lfre{}\mathfrak{Sel}_{Sp}(K_{n},T_{g})\lfre{\partial_{p}\,}
           \bigoplus_{\mathfrak{p}|p}H^{1}_{\mathrm{sing},\varepsilon}(K_{\mathfrak{p}},T_{f}(I_{n,k}))\lfre{}0.
\end{equation}
The only nontrivial fact is the surjectivity of the residue map $\partial_{p}$. 
By construction (cf.\ Section $\ref{selsec}$) the local conditions defining 
$\mathfrak{Sel}_{S}^{\varepsilon}(K_{n},T_{g})$
are dual to those defining $\mathrm{Sel}_{\varepsilon}^{S}(K_{n},A_{g})$,
hence Poitou--Tate duality 
implies that the cokernel of $\partial_{p}$ is isomorphic to a submodule of the Pontrjagin dual of 
$\mathrm{Sel}_{\varepsilon}^{S}(K_{n},A_{g})$
(see e.g. Theorem 7.3 of \cite{Rub-eul}), which is trivial according to  Step 1. \vspace{1mm}

\emph{Step 5.} Step 3, Step 4 and Proposition $\ref{freeloc}.(1)$ prove that $\mathfrak{Sel}_{S}^{\varepsilon}(K_{n},T_{g})$
is a (projective, hence) free $\iw_{\mathscr{O},n,k}$-module of rank $\delta(S)$, thus concluding the proof of the proposition. 
\end{proof}

We record for future application the following 

\begin{cor}\label{corcontrolfreeing} Let $S$ be freeing set relative to $g$. Then, for each integer $1\les{}n\les{}\infty$, the natural projection $\iw\rightarrow{}\iw_{n}$ induces isomorphisms 
of $\iw$-modules 
\[
       \mathfrak{Sel}_{Sp}(K_{\infty},T_{g})/\omega_{n}\cong{}\mathfrak{Sel}_{Sp}(K_{n},T_{g})\ \ \ \text{and}\ \ \ 
       \mathfrak{Sel}_{S}^{\varepsilon}(K_{\infty},T_{g})/\omega_{n}\cong{}\mathfrak{Sel}_{S}^{\varepsilon}(K_{n},T_{g}). 
\]
\end{cor}
\begin{proof} This is a direct consequence of Propositions $\ref{speinj}$ and $\ref{global freeness}$.
\end{proof}

\subsection{Self-duality}  In this section $k<\infty$. Let $\chi : \iw\lfre{}\mathscr{O}_{\chi}$ be a morphism of $\Z_{p}$-algebras into a discrete valuation ring $\mathscr{O}_{\chi}$
contained in $\bar{\Q}_{p}$. Set $\iw_{\chi}=\iw_{\mathscr{O}_{\chi}}$, and write again $\chi : \iw_{\chi}\tlfre{}\mathscr{O}_{\chi}$
for the $\mathscr{O}_{\chi}$-linear extension of $\chi$. 
Set $\mathfrak{P}_{\chi}=\ker(\chi)$, and set 
\[
      T_{g}(\chi)=T_{g,\mathscr{O}_{\chi}}(\mathfrak{P}_{\chi})\ \ \ \text{and}\ \ \ A_{g}(\chi)=A_{g,\mathscr{O}_{\chi}}(\mathfrak{P}_{\chi}).
\]
If $\bar{\chi}=\chi\circ{}\iota$ denotes the composition of $\chi$ with (the $\mathscr{O}_{\chi}$-linear extension of) Iwasawa's main involution,
then $T_{g}(\chi)=T_{g}\otimes_{\Z_p}\mathscr{O}_{\chi}(\bar{\chi})$ and $A_{g}(\chi)=\Hom{\Z_{p}}(T_{g},\mu_{p^{\infty}})\otimes_{\Z_{p}}\mathscr{O}_{\chi}(\bar{\chi})$
as $\mathscr{O}_{\chi}[G_{K}]$-modules (where $\mathscr{O}_{\chi}(\bar{\chi})$ is a copy of $\mathscr{O}_{\chi}$ on which $G_{K}$
acts through the composition of $G_{K}\tlfre{}\mathrm{Gal}(K_{\infty}/K)\hlfre{}\iw^{\ast}$ and $\bar{\chi}$). 
Via the identifications $T_{f,k}\simeq{}T_{g}$ and $A_{f,k}\simeq{}A_{g}$ fixed above, the Weil pairing then yields and isomorphism 
\[
             w_{\chi} : T_{g}(\chi)\simeq{}A_{g}(\chi).
\]
of $\mathscr{O}_{\chi}[G_{K}]$-modules. For each $\varepsilon$ in $\{\emptyset,+,-\}$, one has the following 

\begin{proposition}\label{selfdual} The isomorphism induced in $G_{K}$-cohomology by $w_{\chi}$ restricts to an isomorphism  $$w_{\chi} : \mathfrak{Sel}^{\varepsilon}(K,T_{g}(\chi))\simeq{}\mathrm{Sel}_{\varepsilon}(K,A_{g}(\chi)).$$
More precisely, for each prime $v$ of $K$ dividing $NLp$, the isomorphism induced in $G_{K_{v}}$-cohomology by $w_{\chi}$ identifies the local conditions at $v$ defining 
$\mathfrak{Sel}_{\varepsilon}(K,T_{g}(\chi))$ and $\mathrm{Sel}^{\varepsilon}(K,A_{g}(\chi))$.
\end{proposition}
\begin{proof} For each prime $v$ of $K$ not dividing $pNL$, denote by $\mathcal{L}_{v}(\chi)$ (resp., $\mathcal{L}_{v}^{\ast}(\chi)$) the local condition at $v$ satisfied by the classes in the Selmer group  
$\mathrm{Sel}^{\varepsilon}(K,A_{g}(\chi))$ (resp., $\mathfrak{Sel}_{\varepsilon}(K,T_{g}(\chi))$).
By definition $\mathcal{L}_{v}^{\ast}(\chi)$ is the orthogonal complement of $\mathcal{L}_{v}(\bar{\chi})$ under the perfect local Tate duality 
\[
     (\cdot,\cdot)_{\chi,v} : H^{1}(K_{v},T_{g}(\chi))\times{}H^{1}(K_{v},A_{g}(\bar{\chi}))\lfre{}\mathscr{O}_{\chi}/p^{k}
\]
arising from the $\mathscr{O}_{\chi}$-bilinear extension of the evaluation pairing $T_{g}\times{}A_{g}\lfre{}\mu_{p^{k}}$.
To prove the proposition it is then sufficient to show that $\mathcal{L}_{v}^{\ast}(\bar{\chi})$ is the orthogonal complement of $\mathcal{L}_{v}^{\ast}(\chi)$
under the perfect pairing 
\[
      [\cdot,\cdot]_{\chi,v} : H^{1}(K_{v},T_{g}(\chi))\times{}H^{1}(K_{v},T_{g}(\bar{\chi}))\lfre{}\mathscr{O}_{\chi}/p^{k}
\]
defined as the composition of $(\cdot,\cdot)_{\chi,v}$ and the isomorphism $\mathrm{id}\times{}w_{\chi}^{-1}$. 
This is easily checked when $v\nmid{}p$, in which case $\mathcal{L}_{v}(\chi)$ is either 
an ordinary local condition, or $A_{g}(\chi)$ is cohomologically trivial.
Assume then that $v=\mathfrak{p}$ is a prime of $K$ dividing $p$, and set $\Psi=K_{\mathfrak{p}}$. 

Since $\chi(\gamma-1)$ belongs to the maximal ideal of $\mathscr{O}_{\chi}$,
there exists an integer $n\gg{}0$ such that $\chi(\omega_{n})$ belongs to $p^{k}\mathscr{O}_{\chi}$, id est $\omega_{n}$ belongs to $(\mathfrak{P}_{\chi},p^{k})$.
According to Corollary $\ref{local control}$, one has an isomorphism 
\[
     H^{1}(\Psi,T_{g}(\chi))\simeq{}H^{1}(\Psi,\mathbf{T}_{g})\otimes_{\iw}\iw_{\chi}/\mathfrak{P}_{\chi}\simeq{}H^{1}(\Psi,T_{g}(I_{n}))\otimes_{\iw_{n}}\iw_{n}/\mathfrak{P}_{\chi},
\] 
which restricts to an isomorphism 
\[
       \mathcal{L}_{v}^{\ast}(\chi)\simeq{}H^{1}_{\mathrm{fin},\varepsilon}(\Psi,T_{g}(I_{n}))\otimes_{\iw_{n}}\iw_{n}/\mathfrak{P}_{\chi}.
\]
One has a similar isomorphism for $\bar{\chi}$ in place of $\chi$, and via these isomorphisms $[\cdot,\cdot]_{\chi,v}$ is the reduction modulo $\mathfrak{P}_{\chi}$ of 
the $\mathscr{O}_{\chi}$-linear extension of the pairing 
\[
        \{\cdot,\cdot\}_{\chi,v} : H^{1}(\Psi,T_{g}(I_{n}))\times{}H^{1}(\Psi,T_{g}(I_{n}))\lfre{}\iw_{n}
\]
defined for each $x$ and $y$ in $H^{1}(\Psi,T_{g}(I_{n}))$ by the formula 
\[
     \{x,y\}_{\chi,v}=\sum_{\sigma\in{}G_{n}}[x,\sigma\cdot{}y]_{n,v}\cdot{}\sigma,
\]
where $[\cdot,\cdot]_{n,v}$ is the pairing defined in Equaiton $(\ref{eq:painv})$. Because $L_{n}^{\ast}=H^{1}_{\mathrm{fin},\varepsilon}(\Psi,T_{g}(I_{n}))$
is a $\iw_{n}$-submodule of $H^{1}(\Psi,T_{g}(I_{n}))$, Lemma $\ref{lele self}$ yields $\{L_{n}^{\ast},L_{n}^{\ast}\}_{\chi,v}=[L_{n}^{\ast},L_{n}^{\ast}]_{n,v}=0$,
hence $[\mathcal{L}_{v}^{\ast}(\chi),\mathcal{L}_{v}^{\ast}(\bar{\chi})]_{\chi,v}=0$. In other words 
$\mathcal{L}_{v}^{\ast}(\bar{\chi})$ is contained in the orthogonal complement $\mathcal{L}_{v}^{\ast}(\chi)^{\perp}$ of $\mathcal{L}_{v}^{\ast}(\chi)$ with respect to the pairing $[\cdot,\cdot]_{\chi,v}$. 
On the other hand, Proposition $\ref{freeloc}$ implies that $\mathcal{L}_{v}^{\ast}(\chi)$, $\mathcal{L}_{v}^{\ast}(\chi)^{\perp}$ and $\mathcal{L}_{v}^{\ast}(\bar{\chi})$
all have the same cardinality $|\mathscr{O}_{\chi}/p^{k}|^{[\Psi:\Q_{p}]}$, thus $\mathcal{L}_{v}^{\ast}(\bar{\chi})=\mathcal{L}_{v}^{\ast}(\chi)^{\perp}$. This 
concludes the proof of the proposition. 
\end{proof}

\begin{remark}\label{remchichibar} Let $\tau$ in $G_{\Q}$ be complex conjugation, let 
$\mathrm{Ad}(\tau)$ be conjugation by $\tau$ on $G_{K}$, and let $[\tau] : T_{g}(\chi)\rightarrow{}T_{g}(\bar{\chi})$ be the isomorphism of $\mathscr{O}_{\chi}$-modules 
induced by $\tau : T_{g}\rightarrow{}T_{g}$.
Then the map sending a $1$-cocycle $\varphi : G_{K}\rightarrow{}T_{g}(\chi)$ to the $1$-cocyle $[\tau]\circ{}\varphi\circ{}\mathrm{Ad}(\tau) : G_{K}\rightarrow{}T_{g}(\bar{\chi})$
induces an isomorphism of $\mathscr{O}_{\chi}$-modules 
%$H^{1}(K,T_{g}(\chi))\simeq{}H^{1}(K,T_{g}(\bar{\chi}))$, which restricts to an isomorphism of $\mathscr{O}_{\chi}$-modules 
$\mathfrak{Sel}_{\varepsilon}(K,T_{g}(\chi))\simeq{}\mathfrak{Sel}_{\varepsilon}(K,T_{g}(\bar{\chi}))$.
\end{remark}

\section{Ramified classes and reciprocity laws}\label{ramrec}
We suppose from now until the end of the paper that 
Hypothesis \ref{thehy} holds. 
According to our conventions stated at the end of the Introduction, recall that we often write $M/x$ for $M/xM$ for a module $M$ over a commutative ring with unity $R$ and an element $x\in R$; in particular, we often write $\Z/p^k$ in place of $\Z/p^k\Z$, for an integer $k\geqslant 1$. 

\subsection{Global classes}\label{subsec:indef}
Let $k$ be a positive integer and $L\in\mathscr{S}_{k}$ a 
squarefree  product of $k$-admissible primes.  
Assume that $L\in{}\mathscr{S}_{k}^{\mathrm{ind}}$ is \emph{indefinite}
(i.e. $\epsilon_{K}(LN^{-})=+1$), so that $J_{N^{+},LN^{-}}$ 
is the Picard variety of the {Shimura curve} $X_{N^{+},LN^{-}}$. 
Let $g=f_L\in S_2(N^+,LN^-,\Z/p^{k})$ be the $L$-level raising of $f$ modulo $p^{k}$. 
Let $I_{g}\subset{}\mathbf{T}_{N^{+},LN^{-}}$ denote the kernel of $g$. 
Proposition 4.4 of \cite{P-W}, a slight generalization of  \cite[Theorem 5.15]{Be-Da-main}, shows that 
there is an isomorphism of $\Z_{p}[G_{\Q}]$-modules 
\[
       \pi_{g} :  \mathrm{Ta}_{p}(J_{N^{+},LN^{-}})/I_{g}\cong{}T_{f,k},
\]
which is unique up to multiplication by a $p$-adic unit by Hypothesis $\ref{thehy}$.
For every integer $n\ges{}0$ define 
\[
        \psi_{g,n}: J_{N^{+},LN^{-}}(K_{n})\lfre{}
        H^{1}(K_{n},\mathrm{Ta}_{p}(J_{N^{+},LN^{-}})/I_{g})\cong{}H^{1}(K_{n},T_{f,k}),
\]
where the first (resp., second) map is induced by the Kummer map (resp., by $\pi_{g}$).
It follows from  Proposition 2.7.12 of \cite{Ne-BD} (see also Section 7 of \cite{Be-Da-main} and Theorem 3.10 of \cite{Da-Io})
that for every $x\in{}J_{N^{+},LN^{-}}(K_{n})$ the class $\psi_{g,n}(x)$ is finite at every prime of $K_{n}$ dividing $p$.
  (To apply Proposition 2.7.12 of \cite{Ne-BD}, note that all results in the 
  Appendix A of loc. cit. on flat cohomology of finite flat group schemes 
  hold for the $p$-divisible group of the elliptic curve $E/K_\mathfrak{p}$ for any prime 
  $\mathfrak{p}\mid p$ 
  of $K$ because $K_\mathfrak{p}/\Q_p$, being unramified, has ramification index
  smaller that $p-1$).
Moreover, since $J_{N^{+},LN^{-}}$ has purely toric reduction at every prime divisor of $LN^{-}$,  
Mumford--Tate theory of $p$-adic uniformisation implies that  
these classes are ordinary  at every such prime.
In particular, 
$\psi_{g,n}$ gives a morphism (cf. Section $\ref{Sec.DCSelmer}$)
\[             \psi_{g,n} : J_{N^{+},LN^{-}}(K_{n})\lfre{}\mathfrak{Sel}(K_{n},T_{g}).
\]
Recall the compatible sequence of Heegner points $P_n(L)$, for $n\ges 0$, introduced in Section \ref{compheeg}
and define $$\tilde\kappa_n(g)=\psi_{g,n}(P_n(L)).$$

\subsubsection{Ordinary case}
Suppose that $E$ has ordinary 
reduction at $p$. In this case, define 
\begin{equation}\label{classes-ord}
\kappa_n(g)=
\frac{1}{\alpha_{p}(g)^{n}}
            \Big(\tilde\kappa_{n-1}(g)-\alpha_{p}(g)\cdot{}
            \tilde\kappa_n(g)\big)\Big).
\end{equation}
By the previous discussion $\kappa_n(g)$  belongs to the compact Selmer group  $\mathfrak{Sel}(K_{n},T_{g})=\mathfrak{Sel}(K,T_{g}(I_{n}))$
(with $I_{n}=\omega_{n}\cdot{}\iw$). A simple computation using \eqref{eq:compdef} shows that the corestriction map takes $\kappa_{n+1}(g)$ to $\kappa_{n}(g)$ for all $n\ges1$. Since 
$\mathfrak{Sel}(K_{\infty},T_{g})=\mathfrak{Sel}(K,\mathbf{T}_{g})$ is isomorphic to the inverse limit of the Selmer groups 
$\mathfrak{Sel}(K_{n},T_{g})$ under the corestriction maps, 
we can define 
\begin{equation}\label{classes-ord-bis}\kappa_{\infty}(g) = \inlim_n\kappa_{n}(g) \in
\mathfrak{Sel}(K_{\infty},T_{g}).
\end{equation}

\subsubsection{Supersingular case} Suppose now that $E$ has supersingular 
reduction at $p$. 
Choose a freeing set $S\in \mathscr S_k$ relative to $g$ which is coprime to $L$.
As $a_{p}(g)=0$ (in $\Z/p^{k}$), the norm relations \eqref{P-1}, \eqref{P-0} and \eqref{eq:compdef} imply that, if $\varepsilon=(-1)^{n}$, 
the class $\tilde{\kappa}_{n}(g)$ is killed by $\omega_{n}^{\varepsilon}$ (cf.\ Lemma $\ref{inertpm}$):
\begin{equation}\label{eq:killome}
    \omega_n^\varepsilon\cdot{}\tilde{\kappa}_n(g)=0\ \ \ \text{if }\ \ \varepsilon=(-1)^n.
\end{equation}
If either $p$ splits in $K$, or $p$ is inert in $K$ and $\varepsilon=-1$,
set $\check{\omega}_{n}^{-\varepsilon}=\tilde{\omega}_{n}^{-\varepsilon}$. If $p$ is inert in $K$ and $\varepsilon=+1$ (id est in the exceptional case),
set $\check{\omega}_{n}^{-\varepsilon}=\omega_{n}^{-}$. Since $\mathfrak{Sel}_{S}(K_{n},T_{g})$ is free over $\iw_{n,k}$ (cf.\ Proposition $\ref{global freeness}$, applied with $\varepsilon=\emptyset$), the previous equation implies that,
if $\varepsilon=(-1)^{n}$, there exists  a unique element 
\[ 
     \kappa_n^\varepsilon(g)\in \mathfrak{Sel}_{S}(K_{n},T_{g})/\omega_n^\varepsilon
\]
such that 
\[
       (-1)^{\delta(n)}\check{\omega}_{n}^{-\varepsilon}\cdot{}\kappa_{n}^{\varepsilon}(g)=\tilde{\kappa}_{n}(g),
\]
where $\delta(n)=n/2$ if $n$ is even (id est $\varepsilon=+1$), and $\delta(n)=(n-1)/2$ if $n$ is odd.

According to Corollary $\ref{corcontrolfreeing}$ (applied again with $\varepsilon=\emptyset$), the projections $\iw\tlfre{}\iw_{n}^{\varepsilon}$ induce isomorphisms
\[
            \mathfrak{Sel}_{S}(K_{\infty},T_{g})/\omega_{n}^{\varepsilon}\simeq{}\mathfrak{Sel}_{S}(K_{n},T_{g})/\omega_{n}^{\varepsilon}.
\]
Via these identifications, the natural projections $\iw_{n+2}^{\varepsilon}\tlfre{}\iw_{n}^{\varepsilon}$
then induce surjective maps  
\[
           \pi_{n+2}^{\varepsilon} : \mathfrak{Sel}_{S}(K_{n+2},T_{g})/\omega_{n+2}^{\varepsilon}\tlfre{}\mathfrak{Sel}_{S}(K_{n},T_{g})/\omega_{n}^{\varepsilon}.
\]
As $\mathfrak{Sel}_{S}(K_{\infty},T_{g})$ is finite free over $\iw$, it is equal to the inverse limit of the maps $\pi_{n+2}^{\varepsilon}$, taken over the set $\N^{\varepsilon}$
of nonnegative integers $n$ satisfying $\varepsilon=(-1)^{n}$: 
\begin{equation}\label{eq:invselsprel}
       \mathfrak{Sel}_{S}(K_{\infty},T_{g})\cong{}\inlim_{n\in{}\N^{\varepsilon}}\mathfrak{Sel}_{S}(K_{n},T_{g})/\omega_{n}^{\varepsilon}.
\end{equation}

\begin{lemma} For each $n$ in $\N^{\varepsilon}$, one has 
\[
       \pi_{n+2}^{\varepsilon}(\kappa_{n+2}^{\varepsilon}(g))=\kappa_{n}^{\varepsilon}(g). 
\]
Via the isomorphism def{}ined in Equation $(\ref{eq:invselsprel})$, one then gets a class 
\[
       \kappa_{\infty}^{\varepsilon}(g)=(\kappa_{n}^{\varepsilon}(g))_{n\in\N^{\varepsilon}}\in{}\mathfrak{Sel}_{S}(K_{\infty},T_{g}).
\]
\end{lemma}
\begin{proof} After identifying $\mathfrak{Sel}_{S}(K_{\infty},T_{g})$ with $\iw^{\delta(S)}$ (cf.\ Proposition $\ref{global freeness}$),
this is proved precisely as the corresponding statement for $p$-adic theta elements (cf.\ Lemma $\ref{inertpm}$).
\end{proof}

The class $\kappa_{\infty}^{\varepsilon}(g)$ is finite at every prime divisor $q$ of $S$. Indeed, as $q$ is a $k$-admissible prime relative to $(f,K)$, the $G_{K_{q}}$-module $T_{g}(I_{n})$ splits as the direct sum 
of $\iw_{k,n}=\iw/(p^{k},\omega_{n})\cdot{}\iw$ (with trivial Galois action) and $\iw_{k,n}(1)$, and both $H^{1}(K_{q},\iw_{k,n})=H^{1}_{\mathrm{fin}}(K_{q},T_{g}(I_{n}))$ and 
$H^{1}(K_{q},\iw_{k,n}(1))=H^{1}_{\mathrm{ord}}(K_{q},T_{g}(I_{n}))$
are free $\iw_{k,n}$-modules of rank one (cf.\ Section $\ref{primesL}$). If $\partial_{q} : H^{1}(K_{n},T_{g})\lfre{}H^{1}_{\mathrm{ord}}(K_{q},T_{g}(I_{n}))$
is the composition of the Shapiro isomorphism $H^{1}(K_{n},T_{g})\cong{}H^{1}(K,T_{g}(I_{n}))$, restriction at $q$, and projection onto the ordinary subspace,
it follows that $\partial_{q}(\kappa_{n}^{\varepsilon}(g))$ is the unique class in $H^{1}_{\mathrm{ord}}(K_{q},T_{g}(I_{n}))/\omega_{n}^{\varepsilon}$ mapping to 
$(-1)^{\delta(n)}\cdot{}\partial_{q}(\tilde{\kappa}_{n}(g))$ under multiplication by $\check{\omega}_{n}^{-\varepsilon}$.
Finally, by construction the Heegner class $\tilde{\kappa}_{n}(g)$ is finite at $q$, id est $\partial_{q}(\tilde{\kappa}_{n}(g))=0$, hence $\partial_{q}(\kappa_{n}^{\varepsilon}(g))=0$.
Taking the limit for $n$ in $\N^{\varepsilon}$ tending to infinity, this proves that $\kappa_{\infty}^{\varepsilon}(g)$ is in the kernel of the residue map 
$\partial_{q} : \mathfrak{Sel}_{S}(K_{\infty},T_{g})\rightarrow{}H^{1}_{\mathrm{ord}}(K_{q},\mathbf{T}_{g})$, as claimed. 

We now prove that the class $\kappa_{\infty}^{\varepsilon}$ is $\varepsilon$-finite at $p$. For each positive integer $n$, let $K_{n,p}$ be the product of the completions of $K_{n}$ at the primes dividing $p$. Since $\tilde{\kappa}_{n}(g)$ is (by construction) finite at $p$, and since $\mathbf{E}(K_{n,p})_{\varepsilon}/p^{k}$ is equal to the $\omega_{n}^{\varepsilon}$-torsion submodule of 
the finite local condition $\mathbf{E}(K_{n,p})/p^{n}$ (cf.\ Theorem $\ref{prop:formal-group}$), Equation $(\ref{eq:killome})$ implies that, if $\varepsilon=(-1)^{n}$, then  
the restriction at $p$ of $\tilde{\kappa}_{n}(g)$ belongs to (the image under the Kummer map of) $\mathbf{E}(K_{n,p})_{\varepsilon}/p^{k}$.
According to (the proof of) Lemma $\ref{lele self}$ the latter is contained in the $\varepsilon$-finite subspace $H^{1}_{\mathrm{fin},\varepsilon}(K_{n,p},T_{g})=H^{1}_{\mathrm{fin},\varepsilon}(K_{p},T_{g}(I_{n}))$,
so that the residue $\partial_{p}(\tilde{\kappa}_{n}(g))$ of $\tilde{\kappa}_{n}(g)$ at $p$ is zero in the singular quotient 
$H^{1}_{\mathrm{sing},\varepsilon}(K_{n,p},T_{g})=H^{1}_{\mathrm{sing},\varepsilon}(K_{p},T_{g}(I_{n}))$, provided that $\varepsilon=(-1)^{n}$. On the other hand 
Proposition $\ref{freeloc}$ proves that $H^{1}_{\mathrm{sing},\varepsilon}(K_{n,p},T_{g})$ is a free $\iw_{k,n}$-module, so that multiplication 
by $\check{\omega}_{n}^{-\varepsilon}$ yields an isomorphism between $H^{1}_{\mathrm{sing},\varepsilon}(K_{n,p},T_{g})/\omega_{n}^{\varepsilon}$ and 
$H^{1}_{\mathrm{sing},\varepsilon}(K_{n,p},T_{g})[\omega_{n}^{\varepsilon}]$. As by construction $\check{\omega}_{n}^{-\varepsilon}\cdot{}\partial_{p}(\kappa_{n}^{\varepsilon}(g))$
equals $(-1)^{\delta(n)}\cdot{}\partial_{p}(\tilde{\kappa}_{n}(g))=0$ if $\varepsilon=(-1)^{n}$, we conclude that $\kappa_{n}^{\varepsilon}(g)$
belongs to the kernel of the residue map $\partial_{p} : \mathfrak{Sel}_{S}(K_{n},T_{g})/\omega_{n}^{\varepsilon}\rightarrow{}H^{1}_{\mathrm{sing},\varepsilon}(K_{n,p},T_{g})/\omega_{n}^{\varepsilon}$
if $\varepsilon=(-1)^{n}$. As $H^{1}_{\mathrm{sing},\varepsilon}(K_{\infty,p},T_{g})=H^{1}_{\mathrm{sing},\varepsilon}(K_{p},\mathbf{T}_{g})$
is the inverse limit of the groups $H^{1}_{\mathrm{sing},\varepsilon}(K_{n,p},T_{g})/\omega_{n}^{\varepsilon}$ as $n$ tends to infinity in $\N^{\varepsilon}$
(cf.\ Corollary $\ref{local control}$),
this proves that the class $\kappa_{\infty}^{\varepsilon}(g)$
belongs to the kernel of the residue map $\partial_{p} : \mathfrak{Sel}_{S}(K_{\infty},T_{g})\rightarrow{}H^{1}_{\mathrm{sing},\varepsilon}(K_{\infty,p},T_{g})$.
We summarise the discussion in the following key

\begin{proposition}\label{class in Selmer} $\kappa_{\infty}^{\varepsilon}(g)$ belongs to $\mathfrak{Sel}^{\varepsilon}(K_{\infty},T_{g})$.
\end{proposition}

\begin{remark}
Let $g=f_L$ be the level raising of $f_k=f\pmod{p^k}$ at a definite product $L$ in $\mathscr{S}_k^\mathrm{def}$, and let 
$\ell$ be a $k$-admissible prime relative to $(f,K)$ not dividing $L$, so that $L\ell$ belongs to $\mathscr{S}_{k}^{\mathrm{ind}}$.
Let $g_\ell=f_{L\ell}$ be the level raising at $\ell$ of $g$ (namely the $L\ell$-level raising of $f_k$).
Then, for $\varepsilon=\pm$ (resp., $\varepsilon=\emptyset$) in the supersingular (resp., ordinary) case, the class 
$\kappa_{\infty}^{\varepsilon}(g_\ell)\in{}\mathfrak{Sel}^{\varepsilon}(K_{\infty},T_{g_{\ell}})$ is also an element of the $\varepsilon$-Selmer group $\mathfrak{Sel}^{\varepsilon}_\ell(K_{\infty},T_{g})$
of $g$ relaxed at $\ell$. 
\end{remark}

\subsection{Reciprocity laws} 
\label{sec:reciprocity}
The cohomology classes in Section \ref{class in Selmer} are related
the square-root $p$-adic $L$-functions by the following explicit 
reciprocity laws.

Recall that $\varepsilon=\emptyset$ in the ordinary case and $\varepsilon=\pm$ in 
the supersingular case. 
Equation \eqref{classes-ord-bis} and Proposition $\ref{class in Selmer}$ define 
global Selmer classes $\kappa_{\infty}^{\varepsilon}(g)$ in 
$\mathfrak{Sel}^{\varepsilon}(K_{\infty},T_{g})$. 
Note that each $k$-admissible prime $\ell$ is totally split 
in $K_\infty/K$, being inert in $K$. Therefore, $H^1(K_{\infty,\ell},T_g)=H^1(K_\ell,\mathbf{T}_g)$ is isomorphic to $H^1(K_\ell,T_g)\otimes\Lambda$,
and Lemma \ref{decadm} allows us to define morphisms 
\[
   v_{\ell} : H^{1}(K_\infty,T_{g})\longrightarrow{}H^{1}_{\mathrm{fin}}(K_{\ell},T_{g})\otimes\Lambda\cong{}\Lambda_k\ \ \text{and}\ \ 
   \partial_{\ell} : H^{1}(K_\infty,T_{g})\longrightarrow{}H^{1}_{\mathrm{ord}}(K_{\ell},T_{g})\otimes\Lambda\cong{}\Lambda_k,\] 
defined by composing the restriction map at $\ell$ with the projection onto the finite and the ordinary (or singular) part respectively (cf.\ Section \ref{primesL}). 
Given a global class $x\in{}H^{1}(K,T_{f,k})$, we call $v_{\ell}(x)$ its \emph{finite part} at $\ell$,
and $\partial_{\ell}(x)$ its \emph{residue} at $\ell$. 
If $L=\prod_i\ell_i\in\mathscr{S}_k$ 
is a squarefree product of admissible primes $\ell_i$, then 
we write $\partial_L=\oplus_i\partial_{\ell_i}$ and $v_L=\oplus_iv_{\ell_i}$. 

\begin{theo}[First Reciprocity Law]\label{firstrec} Assume that $L\in{}\mathscr{S}_{k}^{\mathrm{def}}$ is \emph{definite}, 
let $g=f_L$ be the $L$-level raising of $f$ modulo $p^k$, 
and let $\ell\nmid L$ be an admissible prime relative to $g$ and $K$, so that 
$L\ell\in{}\mathscr{S}_{k}^{\mathrm{ind}}$ is \emph{indefinite}. Let $g_\ell$ 
be the $\ell$-level raising of $g$. 
The following equality 
\[\partial_\ell\left(\kappa_\infty^\varepsilon(g_\ell)\right)=\mathcal L^\varepsilon_g\] 
holds in $\Lambda/p^k$ up to units.   
\end{theo}

\begin{proof} 
By \cite[Theorem 4.1]{Be-Da-main}, whose proof works both in the ordinary and in the supersingular case, we have 
$\partial_\ell\left(\kappa_n(g_\ell)\right)=\mathcal L_{g,n}$. 
In the ordinary case, this completes the proof. In the supersingular case, recall that
by definition  we have for all $n$ such that $\varepsilon=(-1)^n$:
\begin{itemize}
\item If $p$ is split in $K$ or $p$ is inert in $K$ and $\varepsilon=-1$ (the non-exceptional case): 
\[
\kappa_n(g_\ell)
=\begin{cases}
(-1)^{n/2}\tilde\omega_n^{-\varepsilon} \kappa_n^\varepsilon(g_\ell), \text{ if $n$ is even};\\
(-1)^{(n-1)/2}\tilde\omega_n^{-\varepsilon} \kappa_n^\varepsilon(g_\ell), \text{ if $n$ is odd}; 
\end{cases}\]
\[
\mathcal L_{g,n}
=\begin{cases}
(-1)^{n/2}\tilde\omega_n^{-\varepsilon}\mathcal L_{g,n}^\varepsilon, \text{ if $n$ is even};\\
(-1)^{(n-1)/2}\tilde\omega_n^{-\varepsilon}\mathcal L_{g,n}^\varepsilon, \text{ if $n$ is odd}; 
\end{cases}\] 
\item If $p$ is inert in $K$ and $\varepsilon=+1$ (the exceptional case): 
\[\kappa_n(g_\ell)=(-1)^{n/2}\omega_n^{-} \kappa_n^+(g_\ell);\]
\[\mathcal L_{g,n}=(-1)^{n/2}\omega_n^{-}\mathcal L_{g,n}^+.\]
\end{itemize} 
In both cases, since $\Lambda_{n,k}^\varepsilon$ is $\tilde\omega_n^{-\varepsilon}$-torsion free (split and non-exceptional case) and is $\omega_n^{-\varepsilon}$-torsion free (exceptional case) it follows from $\partial_\ell\left(\kappa_n(g_\ell)\right)=\mathcal L_{g,n}$ that 
$\partial_\ell\left(\kappa_n^\varepsilon(g_\ell)\right)=\mathcal L^\varepsilon_{g,n}$ for all $n\ges0$, and the conclusion follows. 
\end{proof}

\begin{theo}[Second Reciprocity Law]\label{secondrec}Assume that $L\in{}\mathscr{S}_{k}^{\mathrm{ind}}$ is \emph{indefinite}, let $g=f_L$ be the $L$-level raising of $f$ modulo $p^k$ 
and let $\ell\nmid L$ be an admissible prime relative to $g$ and $K$, so that 
$L\ell\in{}\mathscr{S}_{k}^{\mathrm{def}}$ is \emph{definite}. Let $g_\ell$ be the $\ell$-level raising of $g$.  
Then $\kappa_\infty^\varepsilon(g)$ 
is finite at $\ell$ and the equality \[v_{\ell}\left(\kappa_\infty^\varepsilon(g)\right)=\mathcal L_{g_\ell}^\varepsilon\] 
holds in $\Lambda/p^k$ up to units. 
\end{theo}

\begin{proof} The result follows as in the proof of Theorem \ref{firstrec} from the relation 
\begin{equation}\label{RLII} 
v_{\ell}\left(\kappa_\infty(g)\right)=\mathcal L_{g_\ell}.\end{equation} 
If $N^-\neq 1$, \eqref{RLII} is \cite[Theorem 4.2]{Be-Da-main}, which is proved in 
Section 9 of loc. cit.   using an extension of Ihara's Lemma to indefinite Shimura curves due to Diamond--Taylor \cite{DT}. The same argument applies when $N^-=1$ using 
the standard Ihara's Lemma; alternatively, to prove \eqref{RLII} when $N^-=1$ 
one can adapt the arguments in Section 6 of 
Vatsal's paper \cite{Vat-mu}, where the case $n=1$ is considered. 
\end{proof}

\section{$\varepsilon$-BSD formulae in the definite case}\label{sec:epsilonBSD}
This section is devoted to the proof of BSD formulae for the $\varepsilon$-Selmer groups.  
They are a crucial ingredient in the proof of the main results stated in the Introduction.
We adopt the abuse of notation introduced in the previous section, thus writing $M/x$ instead of $M/xM$ for any element $x$ of a commutative ring with unity $R$, and for any $R$-module $M$.  

Fix a positive integer  $k\ges{}1$ and  a (possibly empty)
\emph{definite} squarefree product $L\in{}\mathscr{S}_{2k}^{\mathrm{def}}$
 of $2k$-admissible primes relative to $(f,K,p)$ (hence $\epsilon_{K}(LN^{-})=-1$). 
Denote by  $\check{g}=f_{L}\in{}S_{2}(N^{+},LN^{-};\Z/p^{2k})$
the $L$-level raising of the reduction of 
 $f$ modulo $p^{2k}$ (cf. Section $\ref{level raising}$) and by 
$g\in{}S_{2}(N^{+},LN^{-};\Z/p^{k})$ the reduction of $\check{g}$ modulo $p^{k}$.

Let $\chi : \iw\rightarrow\mathscr{O}_\chi$ be a
morphism of $\Z_{p}$-algebras, where $\mathscr O_\chi$ is a discrete valuation ring finite over $\Z_p$. 
Denote by $\mathfrak{P}_\chi$ the kernel of $\chi$. 
We assume throughout this section that $\mathscr O_\chi$ is the integral closure of $\Lambda/\mathfrak P_\chi$ in its 
fraction field $\mathscr{K}_\chi=\mathrm{Frac}({\mathscr{O}_\chi})$ and, by an abuse of notation, 
we still denote by $$\chi:\Lambda_{\mathscr{O}_\chi}\longtwoheadrightarrow{}\mathscr{O}_\chi$$ the morphism 
of $\mathscr{O}_\chi$-algebras 
induced by $\chi$ and by $\mathfrak P_\chi\subset{}\iw_{\mathscr{O}_\chi}$ its kernel.  
Let $\mathrm{ord}_{\chi} : {\mathscr{K}_\chi}\twoheadrightarrow{}\Z\cup{}\{\infty\}$
be the normalised discrete valuation,
let $\varpi_{\chi}$ be a uniformiser of ${\mathscr{O}_\chi}$ and  let  $\F_{\chi}={\mathscr{O}}_{\chi}/\varpi_{\chi}$
be its residue field. If  $M$ is a finite free $\mathscr{O}_{\chi}/\varpi_{\chi}^{m}$-module 
(for some integer $m\ges{}1$) and $x$ is a non-zero element of $M$,
denote by $\mathrm{ord}_{\chi}(x)\in{}\N$ the largest nonnegative  integer $t\ges{}0$ such that 
$x\in{}\varpi_{\chi}^{t}\cdot{}M$. After setting $\mathrm{ord}_{\chi}(0)=\infty$,
this defines an \emph{$\mathscr{O}_\chi$-adic valuation} $\mathrm{ord}_{\chi} : M\fre{}\{0,1,\dots{},m-1,\infty\}$. Recall that we already introduced the notation 
\[T_{g}(\chi)=T_{g,{{\mathscr{O}_\chi}}}(\mathfrak P_\chi)\ \ \ \text{and}\ \ \ A_{g}(\chi)=A_{g,{{\mathscr{O}_\chi}}}(\mathfrak P_\chi).\]

\begin{theo}\label{BSD0} Assume that $\mathcal{L}_g^\varepsilon(\bar\chi)\neq 0$. Then 
$\mathrm{length}_{\mathscr{O}_{\chi}}\big(\mathrm{Sel}_{\varepsilon}(K,A_{g}(\chi))\big)
               \les{}2\ord_\chi\left(\mathcal{L}_g^\varepsilon(\bar\chi)\right)$, with equality in the 
               non-exceptional case. 
\end{theo}

The rest of this section is devoted to the proof of Theorem \ref{BSD0}.

\subsection{The Kolyvagin system}\label{kolsys} 
Assume that the value of the $p$-adic $L$-function 
$\mathcal{L}_{g}^{\varepsilon}\in{}\iw/p^{k}$ at $\chi$  is non-zero and denote by 
\begin{equation}\label{eq:nontriv ass}
    t^{\varepsilon}_{\chi}(g)=\mathrm{ord}_{\chi}(\mathcal{L}_{g}^{\varepsilon}(\chi))<\infty
\end{equation}
its $\varpi_{\chi}$-adic valuation. 
Let $\ell\in{}\mathscr{S}_{2k}$
be a $2k$-admissible prime  not dividing  $L$, so that 
$\ell\cdot{}L\in{}\mathscr{S}_{2k}^{\mathrm{ind}}$ is \emph{indefinite}, and let 
$S\in{}\mathscr{S}_{2k}$ be a freeing set relative to $\check{g}$ which is divisible by $\ell\cdot{}L$
(cf. Section $\ref{glofree}$). 
We simplify the notation and write 
\[\mathfrak{Sel}^{\varepsilon}_{S}(K,\mathbf{T}_{\check{g}})\otimes{\mathscr{O}_\chi}=\mathfrak{Sel}^{\varepsilon}_{S}(K,\mathbf{T}_{\check{g},{\mathscr{O}_\chi}})\otimes_{\iw_{{\mathscr{O}_\chi}}}{\mathscr{O}_\chi},\]
where the tensor product on the right is taken with respect to the canonical map
$\chi:\Lambda_{\mathscr{O}_\chi}\rightarrow\mathscr{O}_\chi$ induced by $\chi$. 
Proposition $\ref{global freeness}$ shows that 
$\mathfrak{Sel}^{\varepsilon}_{S}(K,\mathbf{T}_{\check{g}})\otimes{\mathscr{O}_\chi}$
is a free ${\mathscr{O}_\chi}/p^{2k}$-module of rank $\delta(S)$. 

Let $\check{g}_\ell$ be the level raising of $\check{g}$ at $\ell$. 
Section $\ref{ramrec}$ attaches to $\check{g}_\ell$ a global cohomology class $\kappa_\infty^\varepsilon(\check{g}_\ell)$ 
\[\kappa^{\varepsilon}_{\infty}(\check{g}_\ell)\in{}\mathfrak{Sel}^{\varepsilon}_{\ell}(K,\mathbf{T}_{\check{g}})
\subset{}\mathfrak{Sel}_{S}^{\varepsilon}(K,\mathbf{T}_{\check{g}})\subset\mathfrak{Sel}_S^\varepsilon(K,\mathbf{T}_{\check{g},\mathscr{O_\chi}})\]  
(cf.\ Proposition $\ref{class in Selmer}$). To simplify the notation, we write from now on 
\[\kappa_{\infty}^{\varepsilon}(\ell)=\kappa^{\varepsilon}_{\infty}(\check{g}_\ell)\]
Denote by $\kappa^\varepsilon_\chi(\ell)$ 
the image of $\kappa_{\infty}^{\varepsilon}(\ell)$  in 
$\mathfrak{Sel}^{\varepsilon}_{S}(K,\mathbf{T}_{\check{g}})\otimes{\mathscr{O}_\chi}$
under the natural projection, 
and by 
\begin{equation}\label{equality t-ord}
t_{\chi}^{\varepsilon}(g,\ell)=\mathrm{ord}_{\chi}(\kappa_{\chi}^{\varepsilon}(\ell))\end{equation}
its ${\mathscr{O}_\chi}$-adic valuation.
Note that $t_{\chi}^{\varepsilon}(g,\ell)$ is  independent of the choice of $S$ and   
Theorem $\ref{firstrec}$ yields 
\begin{equation}\label{eq:frkol}
       t_{\chi}^{\varepsilon}(g,\ell)\les{}
       \mathrm{ord}_{\chi}\big(\partial_{\ell}(\kappa_{\chi}^{\varepsilon}(\ell))\big)
       =\mathrm{ord}_{\chi}(\mathcal{L}_{\check{g}}^{\varepsilon}(\chi))=t_{\chi}^{\varepsilon}(g)<
       \mathrm{ord}_{\chi}(p^{k}),
\end{equation} 
where \[\partial_{\ell} : \mathfrak{Sel}_{S}^{\varepsilon}(K,\mathbf{T}_{\check{g},{{\mathscr{O}_\chi}}})
\longrightarrow H^{1}_{\mathrm{sing}}(K_{\ell},T_{\check{g},{{\mathscr{O}_\chi}}})\cong{}\iw_{{\mathscr{O}_\chi}}/p^{2k}\] is the scalar extension of the residue map at $\ell$
introduced in Section $\ref{ramrec}$ and the second equality follows from Equation $(\ref{eq:nontriv ass})$.
In particular there exists 
$\tilde{\kappa}_{\chi}^{\varepsilon}(\ell)\in{}\mathfrak{Sel}_{S}^{\varepsilon}(K,\mathbf{T}_{\check{g}})
\otimes\mathscr O_\chi$ such that 
\begin{equation}\label{eq:divcla}
    \mathrm{ord}_{\chi}(\tilde{\kappa}_{\chi}^{\varepsilon}(\ell))=0,\end{equation}
\begin{equation}\label{eq:divcla2}     \kappa_{\chi}^{\varepsilon}(\ell)=
\varpi_{\chi}^{t_{\chi}^{\varepsilon}(g,\ell)}\cdot{}\tilde{\kappa}_{\chi}^{\varepsilon}(\ell).
\end{equation}
While  $\tilde{\kappa}_{\chi}^{\varepsilon}(\ell)$ is not uniquely determined by the previous equations, its image 
\[\hat{\kappa}^{\varepsilon}_{\chi}(\ell)\in{}
\mathfrak{Sel}_{S}^{\varepsilon}(K,\mathbf{T}_{g})\otimes\mathscr O_\chi\overset{\mathrm{def}}=\mathfrak{Sel}_{S}^{\varepsilon}(K,\mathbf{T}_{g,{{\mathscr{O}_\chi}}})\otimes_{\Lambda_{{\mathscr{O}_\chi}}}\mathscr O_\chi\]
under the morphism induced by the projection $T_{\check{g}}\twoheadrightarrow{}T_{g}$ is independent 
of any choice. 
Let 
\[\texttt{s}_{\chi} : \mathfrak{Sel}_{S}^{\varepsilon}(K,\mathbf{T}_{g})\otimes\mathscr O_\chi\longrightarrow\mathfrak{Sel}_{S}^{\varepsilon}(K,T_{g}(\chi))\] be the specialization map. 
Define 
\[
           \xi_{\chi}^{\varepsilon}(\ell)=\xi_{\chi}^{\varepsilon}(g,\ell)
           =\texttt{s}_{\chi}(\hat{\kappa}_{\chi}^{\varepsilon}(\ell))
           \in{}\mathfrak{Sel}_{S}^{\varepsilon}(K,T_{g}(\chi)) .
\]
and 
\[
           \bar{\xi}_{\chi}^{\varepsilon}(\ell)=\bar{\xi}_{\ell}^{\varepsilon}(g,\ell)\in{}H^{1}(K,T_{\bar{g}})\otimes_{\F_{p}}\F_{\chi}
\] 
as the image of $\hat{\kappa}_{\chi}^{\varepsilon}(\ell)$
under the map  induced in cohomology  by the reduction map
$T_{g}(\chi)\twoheadrightarrow{}
T_{\bar{g}}\otimes_{\F_{p}}\F_{\chi}$, where $\bar{g}\in{}S_{2}(N^{+},LN^{-};\F_{p})$
is the reduction of $g$ modulo $p$.

\begin{lemma}[\emph{cf. Lemma 4.5 of} \cite{Be-Da-main}]\label{basickol} \hfill
\begin{enumerate}
\item  $0\not=\xi_{\chi}^{\varepsilon}(\ell)\in{}\mathfrak{Sel}_{\ell}^{\varepsilon}(K,T_{g}(\chi))$
and $v_{\ell}\big(\xi_{\chi}^{\varepsilon}(\ell)\big)=0$.
\item $\mathrm{ord}_{\chi}\big(\partial_{\ell}(\xi_{\chi}^{\varepsilon}(\ell))\big)=
           t_{\chi}^{\varepsilon}(g)-t_{\chi}^{\varepsilon}(g,\ell).$
\item  $0\not=\bar{\xi}_{\chi}^{\varepsilon}(\ell)\in{}\mathfrak{Sel}_{\ell}^{\varepsilon}(K,T_{\bar{g}})\otimes_{\F_{p}}\F_{\chi}$
and    $\partial_{\ell}(\bar{\xi}_{\chi}^{\varepsilon}(\ell))$  is non-zero if and only if 
$t_{\chi}^{\varepsilon}(g,\ell)=t_{\chi}^{\varepsilon}(g)$.
\end{enumerate}
\end{lemma}
\begin{proof} (1) Because the kernel of  $\mathfrak{Sel}_{S}^{\varepsilon}(K,\mathbf{T}_{\check{g}})\fre{}
\mathfrak{Sel}_{S}^{\varepsilon}(K,\mathbf{T}_{g})$ is killed by $p^{k}$ 
and 
$\mathrm{ord}_{\chi}(\tilde{\kappa}_{\chi}^{\varepsilon}(\ell))=0$ by Equation  $(\ref{eq:divcla})$, 
the class $\hat{\kappa}_{\chi}^{\varepsilon}(\ell)$ is not zero, hence so is its image 
$\xi_{\chi}^{\varepsilon}(\ell)$ under the map $\texttt{s}_{\chi}$,
which is injective by Proposition $\ref{speinj}$. 
Let $q$ be a  prime divisor of $S/\ell$.
To prove the first statement one has to show that the residue 
\[\partial_{q}(\xi_{\chi}^{\varepsilon}(\ell))\in{}
H^{1}_{\mathrm{sing}}(K_{q},T_{g,{\mathscr{O}_\chi}})\cong{}\mathscr{O}_{\chi}/p^{k}\]
of $\xi_{\chi}^{\varepsilon}(\ell)$ at $q$ is zero. 
Fix isomorphisms  $H^{1}_{\mathrm{sing}}(K_{q},\mathbf{T}_{\check{g},{\mathscr{O}_\chi}})\cong{}\iw_{\mathscr{O}_\chi}/p^{2k}$
and $H^{1}_{\mathrm{sing}}(K_{q},\mathbf{T}_{g,{\mathscr{O}_\chi}})\cong{}\iw_{\mathscr{O}_\chi}/p^{k}$
such that the map $H^{1}_{\mathrm{sing}}(K_{q},\mathbf{T}_{\check{g},{\mathscr{O}_\chi}})\fre{}
H^{1}_{\mathrm{sing}}(K_{q},\mathbf{T}_{g,{\mathscr{O}_\chi}})$
becomes  identified with the natural projection $\iw_{\mathscr{O}_\chi}/p^{2k}\twoheadrightarrow{}\iw_{\mathscr{O}_\chi}/p^{k}$. 
Since $\partial_{q}(\kappa_{\infty}^{\varepsilon}(\ell))$ is zero  by Proposition $\ref{class in Selmer}$
and $t_{\chi}^{\varepsilon}(g,\ell)<\mathrm{ord}_{\chi}(p^{k})$ by Equation $(\ref{eq:frkol})$, 
it follows  that 
$\partial_{q}(\tilde{\kappa}_{\chi}^{\varepsilon}(\ell))\in{}
\mathscr{O}_{\chi}/p^{2k}$
has $\mathscr{O}_{\chi}$-adic valuation at least $\mathrm{ord}_{\chi}(p^{k})$, hence its projection 
$\partial_{q}(\hat{\kappa}_{\chi}^{\varepsilon}(\ell))\in{}\mathscr{O}_{\chi}/p^{k}$
modulo $p^{k}$ vanishes (here and in the following we wrote $\partial_q$ for the 
scalar extension $\partial_q\otimes\mathrm{id}$ to simplify the notation as before).  This gives 
$$\partial_{q}(\xi_{\chi}^{\varepsilon}(\ell))=\partial_{q}\circ{}\texttt{s}_{\chi}(\hat{\kappa}_{\chi}^{\varepsilon}(\ell))
=\partial_{q}(\xi_{\chi}^{\varepsilon}(\ell))=0,$$
as was to be shown. 
The second statement is proved similarly, using  that $v_{\ell}(\kappa_{\infty}^{\varepsilon}(\ell))=0$
by Proposition $\ref{class in Selmer}$.

(2) Equations $(\ref{eq:frkol})$, $(\ref{eq:divcla})$ and $(\ref{eq:divcla2})$  show that 
$\partial_{\ell}(\tilde{\kappa}_{\chi}^{\varepsilon}(\ell))$ has $\mathscr{O}_{\chi}$-adic valuation 
$t_{\chi}^{\varepsilon}(g)-t_{\chi}^{\varepsilon}(g,\ell)$. Since  $\mathrm{ord}_{\chi}(p^{k})>t_{\chi}^{\varepsilon}(g)$ 
this   is also the $\mathscr{O}_{\chi}$-adic valuation of 
$\partial_{\ell}(\hat{\kappa}_{\chi}^{\varepsilon}(\ell))$,
which is equal to  that of $\partial_{\ell}(\xi_{\chi}^{\varepsilon}(\ell))$
(\emph{cf.} the proof of (1)).  

(3) Note that the class $\bar{\xi}_{\chi}^{\varepsilon}(\ell)$ is equal to the image of 
$\tilde{\kappa}_{\chi}^{\varepsilon}(\ell)$ under the composition 
\[
               \mathfrak{Sel}_{S}^{\varepsilon}(K,\mathbf{T}_{\check{g}})\otimes\mathscr{O}_{\chi}
               \lfre{}\mathfrak{Sel}_{S}^{\varepsilon}(K,\mathbf{T}_{\bar{g}})\otimes\mathscr{O}_{\chi}\lfre{\texttt{s}_{\chi}}
               \mathfrak{Sel}_{S}^{\varepsilon}(K,T_{\bar{g}}(\chi))\otimes_{\F_{p}}\F_{\chi}.
\] 
As above this implies that $\bar{\xi}_{\chi}^{\varepsilon}(\ell)$ is not zero, since $\texttt{s}_{\chi}$
is injective and $\mathrm{ord}_{\chi}(\tilde{\kappa}_{\chi}^{\varepsilon}(\ell))=0$.
Together with (1) this implies the first  statement. 
Since $\partial_{\ell}\big(\bar{\xi}_{\chi}^{\varepsilon}(\ell)\big)\in{}H^{1}_{\mathrm{sing}}
(K_{\ell},T_{\bar{g}}(\chi))\otimes_{\F_{p}}\F_{\chi}\cong{}\F_{\chi}$ is the 
projection of 
$\partial_{\ell}(\xi_{\chi}^{\varepsilon}(\ell))\in{}\mathscr{O}_{\chi}/p^{k}$ modulo $\varpi_{\chi}$,
the second statement follows from (2). 
\end{proof}

\subsection{Proof of Theorem \ref{BSD0}} The proof of Theorem \ref{BSD0} is divided into several steps.
Steps 1, 2 and 3 consist in a generalization to the present context of similar results of \cite{Be-Da-main}.  
The direct generalizations of the techniques in \cite{Be-Da-main} only allow one to prove the inequality \[\mathrm{length}_{\mathscr{O}_{\chi}}\big(\mathrm{Sel}_{\varepsilon}(K,A_{g}(\chi))\big)
               \les{}2\ord_\chi\left(\mathcal{L}_g^\varepsilon(\bar\chi)\right)\]
in Theorem \ref{BSD0}; this inequality holds in both cases, exceptional and non-exceptional, while the opposite inequality can be shown in the non-exceptional case only
with a further inductive argument on the length of $\mathrm{Sel}_{\varepsilon}(K,A_{g}(\chi))$, developed in Steps 4, 5, 6 and 7. The key ingredient for the inequality
\[\mathrm{length}_{\mathscr{O}_{\chi}}\big(\mathrm{Sel}_{\varepsilon}(K,A_{g}(\chi))\big)
               \ges{}2\ord_\chi\left(\mathcal{L}_g^\varepsilon(\bar\chi)\right)\]
 is Step 4 (the basis of the inductive argument, i.e.  the case when $\mathrm{length}_{\mathscr{O}_{\chi}}\big(\mathrm{Sel}_{\varepsilon}(K,A_{g}(\chi))\big)=0$) which combines Gross formula and Lemma \ref{nonexcss} with results of Skinner--Urban (ordinary case) and Fouquet-Wan (supersingular case); the inductive argument then follows in Steps 6 and 7 using a a structure theorem for 
$\mathrm{Sel}_{\varepsilon}(K,A_{g}(\chi))$, which we prove in Step 5. 

\subsubsection{Step 1} If $\mathcal{L}_{g}^{\varepsilon}(\bar{\chi})$
is a $p$-adic unit then $\mathrm{Sel}_{\varepsilon}(K,A_{g}(\chi))$ is trivial. 

\begin{proof} (Cf. \cite[Proposition 4.7]{Be-Da-main}.)
Assume \emph{ad absurdum} that there exists a nontrivial class  $x$
in the Selmer group $\mathrm{Sel}_{\varepsilon}(K,A_{g}(\chi))$.
Choose a $2k$-admissible prime $\ell$ 
such that $v_{\ell}(x)\in{}H^{1}_{\mathrm{fin}}(K_{\ell},A_{g}(\chi))\cong{}\mathscr{O}_{\chi}/p^{k}$ is not zero, which exists 
by Theorem 3.2 of \cite{Be-Da-main}. Since $\mathfrak{Sel}^{\varepsilon}(K,T_{g}(\bar{\chi}))$
is the dual Selmer group of $\mathrm{Sel}_{\varepsilon}(K,A_{g}(\chi))$, Lemma $\ref{basickol}(1)$
and the reciprocity law of global class field theory 
yield  
\[
                 0=\sum_{v}\dia{\mathrm{res}_{v}(x),\mathrm{res}_{v}(\xi_{\bar{\chi}}^{\varepsilon}(\ell))}_{v}
                 =\dia{\mathrm{res}_{\ell}(x),\mathrm{res}_{\ell}(\xi_{\bar{\chi}}^{\varepsilon}(\ell))}_{\ell},
\]
where the sum is taken over all the primes of $K$ and 
$\dia{-,-}_{v}$ denotes the local Tate pairing at $v$ induced by the
duality  $T_{g}(\chi)\times{}A_{g}(\chi)\fre{}\mathscr{O}_{\chi}/p^{k}(1)$
(cf. \cite[Chapter 1]{Mi}). Since $\partial_{\ell}(x)=0$, $v_{\ell}(x)\not=0$ and 
$H^{1}_{\mathrm{fin}}(K_{\ell},T_{g}(\bar{\chi}))$ is the orthogonal complement of 
$H^{1}_{\mathrm{fin}}(K_{\ell},A_{g}(\chi))$
under the perfect pairing $\dia{-,-}_{\ell}$, the previous equation implies that 
the residue at $\ell$ of $\xi_{\bar{\chi}}^{\varepsilon}(\ell)$ has positive $\mathscr{O}_{\chi}$-adic valuation. 
According to Lemma $\ref{basickol}(2)$ this in turn implies that $\mathcal{L}_{g}^{\varepsilon}(\bar{\chi})$
has positive $\mathscr{O}_{\chi}$-adic valuation, contradicting the assumption. 
\end{proof}

\subsubsection{Step 2.} Assume that $\mathrm{Sel}_{\varepsilon}(K,A_{g}(\chi))$ is non-trivial. Then there exist 
two distinct $2k$-admissible primes $\ell_{1}$ and $\ell_{2}$ satisfying the following properties. 

\begin{itemize}
\item[$\mathbf{I}_{1}$.] $t_{\bar{\chi}}^{\varepsilon}(g,\ell_{1})=t_{\bar{\chi}}^{\varepsilon}(g,\ell_{2})<t_{\bar{\chi}}^{\varepsilon}(g)$.
\item[$\mathbf{I}_{2}.$] If $h\in{}S_{2}(N^{+},L\ell_{1}\ell_{2}N^{-};\Z/p^{k})$ denotes the $\ell_{1}\ell_{2}$-level raising 
of $g$, then 
\[
               \mathrm{Sel}_{\varepsilon}(K,A_{h}(\chi))=\mathrm{Sel}_{\varepsilon}^{\ell_{1}\ell_{2}}
               (K,A_{g}(\chi)).
\]
\item[$\mathbf{I}_{3}.$] The $\varpi_{\chi}$-adic valuation of $\mathcal{L}_{h}^{\varepsilon}(\bar{\chi})\in\mathscr{O}_{\chi}/p^{k}$
is equal to $t_{\bar{\chi}}^{\varepsilon}(g,\ell_{i})$ (for $i=1,2$):
\[
              t_{\bar{\chi}}^{\varepsilon}(h)=\mathrm{ord}_{\chi}\big(\mathcal{L}_{h}^{\varepsilon}(\bar{\chi})\big)
              =t_{\bar{\chi}}^{\varepsilon}(g,\ell_{i})<\infty. 
\]
\item[$\mathbf{I}_{4}.$]  $\mathrm{ord}_{\chi}\big(v_{\ell_{1}}\big(\xi_{\bar{\chi}}^{\varepsilon}(\ell_{2})\big)\big)=0$
and $\mathrm{ord}_{\chi}\big(
v_{\ell_{2}}\big(\xi_{\bar{\chi}}^{\varepsilon}(\ell_{1})\big)\big)=0$.
\end{itemize}

\begin{proof} We first prove that there exist infinitely many $2k$-admissible primes 
$\ell$ such that $t_{\bar{\chi}}(g,\ell)<t_{\bar{\chi}}(g).$
Let $\mathfrak{m}_{\iw_{\mathscr{O}_\chi}}$ be as above the maximal ideal of $\iw_{\mathscr{O}_\chi}$, so that 
$A_{g}(\chi)[\mathfrak{m}_{\iw_{\mathscr{O}_\chi}}]\cong{}A_{\bar{g}}\otimes_{\F_{p}}\F_{\chi}$
(as $\chi(g)\equiv{}1\pmod{\varpi_{\chi}}$ for every 
$g\in{}G_{\infty}$). The control theorem of Proposition $\ref{global control}$
yields \[\mathrm{Sel}_{\varepsilon}(K,A_{\bar{g},{\mathscr{O}_\chi}})\cong{}
\mathrm{Sel}_{\varepsilon}(K,A_{g}(\chi))[\mathfrak{m}_{{\iw}_{\mathscr{O}_\chi}}],\] hence $\mathrm{Sel}_{\varepsilon}(K,A_{\bar{g},{\mathscr{O}_\chi}})$
is nontrivial by Nakayama's Lemma. Fix a non-zero class 
\[
                0\not=x\in{}\mathrm{Sel}_{\varepsilon}(K,A_{\bar{g},{\mathscr{O}_\chi}}). 
\] 
According to (a slight generalization of) Theorem 3.2 of \cite{Be-Da-main} there exist infinitely many $2k$-admissible primes 
$\ell$ such that $v_{\ell}(x)\in{}H^{1}_{\mathrm{fin}}(K_{\ell},A_{\bar{g},{\mathscr{O}_\chi}})$
is non zero. We claim that for every such prime $\ell$ one has   
\begin{equation}\label{eq:nmn1}
     t_{\bar{\chi}}^{\varepsilon}(g,\ell)<t_{\bar{\chi}}^{\varepsilon}(g).
\end{equation}
Recall the class $\bar{\xi}_{\bar{\chi}}^{\varepsilon}(\ell)\in{}H^{1}(K,T_{\bar{g}})\otimes_{\F_{p}}\F_{\chi}$
constructed in Section $\ref{kolsys}$. Lemma $\ref{basickol}(3)$ shows that  
$\bar{\xi}_{\bar{\chi}}^{\varepsilon}(\ell)$ belongs to 
$\mathfrak{Sel}_{\ell}^{\varepsilon}(K,T_{\bar{g}})\otimes_{\F_{p}}\F_{\chi}$,
hence (as in the proof of Step 1) the reciprocity law of  global class field theory yields 
\[
            \dia{\partial_{\ell}\big(\bar{\xi}_{\bar{\chi}}^{\varepsilon}(\ell)\big),v_{\ell}(x)}_{\ell}=0,
\]
where $\dia{-,-}_{\ell}$ is the $\F_{\chi}$-linear extension of the perfect local Tate pairing 
\[H^{1}_{\mathrm{sing}}(K_{\ell},T_{\bar{g}})\otimes_{\F_{p}}H^{1}_{\mathrm{fin}}(K_{\ell},A_{\bar{g}})\longrightarrow\F_{p}.\]
Since $v_{\ell}(x)\not=0$ this gives  $\partial_{\ell}\big(\bar{\xi}_{\bar{\chi}}^{\varepsilon}(\ell)\big)=0$,
and the claim $(\ref{eq:nmn1})$ follows from another application of Lemma $\ref{basickol}(3)$.

Fix a $2k$-admissible prime $\ell_{1}$ such that 
$t_{\bar{\chi}}^{\varepsilon}(g,\ell_{1})<t_{\bar{\chi}}^{\varepsilon}(g)$, and such that 
$t_{\bar{\chi}}^{\varepsilon}(g,\ell_{1})\les{}t_{\bar{\chi}}^{\varepsilon}(g,\ell)$ for every 
$2k$-admissible prime $\ell$.  Since $\bar{\xi}^{\varepsilon}_{\bar{\chi}}(\ell_{1})$
is non-zero by Lemma $\ref{basickol}(3)$, Theorem 3.2 of \cite{Be-Da-main} proves that there exists
a $2k$-admissible prime $\ell_{2}\not=\ell_{1}$ such that 
$v_{\ell_{2}}\big(\bar{\xi}_{\bar{\chi}}^{\varepsilon}(\ell_{1})\big)\in{}H^{1}_{\mathrm{fin}}(K_{\ell_{2}},
T_{\bar{g}})\otimes_{\F_{p}}\F_{\chi}\cong{}\F_{\chi}$ is non-zero. 
By construction (cf. Section $\ref{kolsys}$) the latter condition is equivalent to 
\[
            \mathrm{ord}_{\chi}\big(v_{\ell_{2}}(\xi_{\bar{\chi}}^{\varepsilon}(\ell_{1}))\big)=0. 
\]
The second reciprocity law Theorem $\ref{secondrec}$ and the definition of $\xi_{\bar{\chi}}^{\varepsilon}(\ell_{1})$
show that the identities (where we write $v_\ell$ for $v_{\ell}\otimes{}\mathrm{id}$ for $\ell=\ell_1$ and $\ell=\ell_2$ as before) 
\begin{equation}\label{eq:semifinalStep3}
   \varpi_{\chi}^{t_{\bar{\chi}}^{\varepsilon}(g,\ell_{1})}\cdot{}v_{\ell_{2}}\big(\xi_{\bar{\chi}}^{\varepsilon}(\ell_{1})\big)
   =v_{\ell_{2}}\big(\kappa_{\bar{\chi}}^{\varepsilon}(\ell_{1})\big)
   \stackrel{\text{Th.}\ \ref{secondrec}}=\mathcal{L}_{h}^{\varepsilon}(\bar{\chi})
   \stackrel{\text{Th.}\ \ref{secondrec}}{=}
   v_{\ell_{1}}\big(\kappa_{\bar{\chi}}^{\varepsilon}(\ell_{2})\big)   
   =\varpi_{\chi}^{t_{\bar{\chi}}^{\varepsilon}(g,\ell_{2})}\cdot{}v_{\ell_{1}}\big(\xi_{\bar{\chi}}^{\varepsilon}(\ell_{2})\big)
\end{equation}
hold in $\mathscr{O}_{\chi}/p^{k}$ up to multiplication by $p$-adic units (\emph{cf.} the proof of Lemma 
$\ref{basickol}(1)$ for the first and last identities). 
Since $t_{\bar{\chi}}^{\varepsilon}(g,\ell)<\mathrm{ord}_{\chi}(p^{k})$ for $\ell=\ell_{1},\ell_{2}$ by Equation $(\ref{eq:frkol})$, and since 
by construction $t_{\bar{\chi}}^{\varepsilon}(g,\ell_{1})\les{}t_{\bar{\chi}}^{\varepsilon}(g,\ell_{2})$, 
the previous two equations and Lemma $\ref{basickol}(1)$ show that 
\begin{equation}\label{eq:finalStep31}
              t_{\bar{\chi}}^{\varepsilon}(g,\ell_{1})=t_{\bar{\chi}}^{\varepsilon}(g,\ell_{2})<
              t_{\bar{\chi}}^{\varepsilon}(g)
\end{equation}
and that the identities 
\begin{equation}
        \lri{v_{\ell_{1}}\big(\xi_{\bar{\chi}}^{\varepsilon}(\ell_{2})\big),v_{\ell_{2}}\big(\xi_{\bar{\chi}}^{\varepsilon}(\ell_{2})\big)}=(1,0), 
        \end{equation} 
               \begin{equation}\label{eq:finalStep3}
        \lri{v_{\ell_{1}}\big(\xi_{\bar{\chi}}^{\varepsilon}(\ell_{1})\big),v_{\ell_{2}}\big(\xi_{\bar{\chi}}^{\varepsilon}(\ell_{1})\big)}=(0,1)
\end{equation}
hold in $\mathscr{O}_{\chi}/p^{k}\oplus{}\mathscr{O}_{\chi}/p^{k}$ up to multiplication by  $p$-adic units. 
(Here for $\ell=\ell_{1}$ or $\ell=\ell_{2}$ one fixes an isomorphism $H^{1}_{\mathrm{fin}}(K_{\ell},T_{g}(\bar{\chi}))
\cong{}\mathscr{O}_{\chi}/p^{k}$.)
It follows from the definitions (cf. Section $\ref{selsec}$) that 
\[
     \mathrm{Sel}^{\ell_{1}\ell_{2}}_{\varepsilon}(K,A_{g}(\chi))
=\mathrm{Sel}^{\ell_{1}\ell_{2}}_{\varepsilon}(K,A_{h}(\chi)),\] 
\[\mathfrak{Sel}_{\ell_{1}\ell_{2}}^{\varepsilon}(K,T_{g}(\bar{\chi}))=
\mathfrak{Sel}_{\ell_{1}\ell_{2}}^{\varepsilon}(K,T_{h}(\bar{\chi})),
\]
and a class
$z\in{}\mathfrak{Sel}^{\varepsilon}_{\ell_{1}\ell_{2}}(K,T_{g}(\bar{\chi}))$
belongs to $\mathfrak{Sel}^{\varepsilon}(K,T_{h}(\bar{\chi}))$ precisely if 
$v_{\ell_{1}}(z)$ and $v_{\ell_{2}}(z)$ are both trivial.
Poitou--Tate duality (see Theorem 7.3 of \cite{Rub-eul} or Chapter 1 of \cite{Mi}) then  yields a short exact sequence of 
$\mathscr{O}_{\chi}/p^{k}$-modules 
\begin{equation}\label{eq:PTStep3}
           \mathfrak{Sel}_{\ell_{1}\ell_{2}}^{\varepsilon}(K,A_{g}(\bar{\chi}))
           \lfre{v_{\ell_{1}}\oplus{}v_{\ell_{2}}}
           \mathscr{O}_{\chi}/p^{k}\oplus\mathscr{O}_{\chi}/p^{k}
           \lfre{\partial_{\ell_{1}}^{\vee}\oplus{}\partial_{\ell_{2}}^{\vee}}
           \mathrm{Sel}_{\varepsilon}(K,A_{h}(\bar{\chi}))^{\vee}
           \lfre{}\mathrm{Sel}_{\varepsilon}^{\ell_{1}\ell_{2}}(K,A_{g}(\chi))^{\vee}\lfre{}0,
\end{equation}
where $(\cdot{})^{\vee}=\Hom{\Z_{p}}(\cdot{},\divp)$ and for $\ell=\ell_{1},\ell_{2}$ one identifies
$H^{1}_{\mathrm{sing}}(K_{\ell},A_{g}(\chi))$ with the Pontrjagin dual of 
$H^{1}_{\mathrm{fin}}(K_{\ell},T_{g}(\bar{\chi}))\cong{}\mathscr{O}_{\chi}/p^{k}$
under the local Tate duality. Equation $(\ref{eq:finalStep3})$ shows that the first map is surjective,
hence $$\mathrm{Sel}_{\varepsilon}(K,A_{h}(\chi))=\mathrm{Sel}_{\varepsilon}^{\ell_{1}\ell_{2}}(K,A_{g}(\chi)).$$
Together with Equations $(\ref{eq:semifinalStep3})$--$(\ref{eq:finalStep3})$ this concludes the proof. 
\end{proof}

\subsubsection{Step 3.} $\mathrm{length}_{\mathscr{O}_{\chi}}\big(\mathrm{Sel}_{\varepsilon}(K,A_{g}(\chi))\big)
\les{}2 t_{\bar{\chi}}^{\varepsilon}(g)$. 

\begin{proof} As in \cite{Be-Da-main} one proceeds by induction on $t_{\bar{\chi}}(g)$. 
Step 1 shows that the statement holds if $t_{\bar{\chi}}(g)=0$.
Assume then $t_{\bar{\chi}}(g)>0$. If $\mathrm{Sel}_{\varepsilon}(K,A_{g}(\chi))=0$
the statement is trivially verified, hence assume that $\mathrm{Sel}_{\varepsilon}(K,A_{g}(\chi))$
is non-trivial.  According to Step 2 there exists two distinct $2k$-admissible primes 
$\ell_{1}$ and $\ell_{2}$ satisfying the properties $\mathbf{I}_{1}$--$\mathbf{I}_{3}$.
As in loc. cit.  denote by $h\in{}S_{2}(N^{+},L\ell_{1}\ell_{2}L;\Z/p^{k})$
the $\ell_{1}\ell_{2}$-level raising of $g$. 

Let $\zeta_{\bar{\chi}}^{\varepsilon}(\ell_{1})\in{}
\mathfrak{Sel}_{\ell_{1}}^{\varepsilon}(K,T_{g}(\bar{\chi}))$ be a global class such that $\partial_{\ell_{1}}\big(
\zeta_{\bar{\chi}}^{\varepsilon}(\ell_{1})\big)$ generates  the image of the residue map 
$\partial_{\ell_{1}} : 
\mathfrak{Sel}_{\ell_{1}}^{\varepsilon}(K,T_{g}(\bar{\chi}))
\fre{}H^{1}_{\mathrm{sing}}(K_{\ell},T_{g}(\bar{\chi}))\cong{}\mathscr{O}_{\chi}/p^{k}$, viz.
$\partial_{\ell_{1}}$ induces an isomorphism 
\begin{equation}\label{eq:step21}
      \partial_{\ell_{1}} : \mathfrak{Sel}_{\ell_{1}}^{\varepsilon}(K,T_{g}(\bar{\chi}))\big/
      \mathfrak{Sel}^{\varepsilon}(K,T_{g}(\bar{\chi}))\cong{}
      \partial_{\ell_{1}}\big(\zeta_{\bar{\chi}}^{\varepsilon}(\ell_{1})\big)\cdot{}\mathscr{O}_{\chi}/p^{k}.
\end{equation}
Since $\xi_{\bar{\chi}}^{\varepsilon}(\ell_{1})$ belongs to the Selmer group 
$\mathfrak{Sel}_{\ell_{1}}^{\varepsilon}(K,T_{g}(\bar{\chi}))$ by Lemma $\ref{basickol}(1)$,
multiplying $\zeta_{\bar{\chi}}^{\varepsilon}(\ell_{1})$ by a $p$-adic unit if necessary one can assume that there exists 
an integer $m_{1}\ges{}0$ such that 
\[
            \xi_{\bar{\chi}}^{\varepsilon}(\ell_{1})-\varpi_{\chi}^{m_{1}}\cdot{}\zeta_{\bar{\chi}}^{\varepsilon}(\ell_{1})
            \in{}\mathfrak{Sel}^{\varepsilon}(K,T_{g}(\bar{\chi})).
\]  
Equation $(\ref{eq:step21})$, Lemma $\ref{basickol}(2)$ and property $\mathbf{I}_{3}$  then yield
\begin{align}\label{eq:step22}
     \mathrm{length}_{\mathscr{O}_{\chi}}\big(
     \mathfrak{Sel}_{\ell_{1}}^{\varepsilon}(K,T_{g}(\bar{\chi}))\big/
      \mathfrak{Sel}^{\varepsilon}(K,T_{g}(\bar{\chi}))\big)
     =\mathrm{ord}_{\chi}(p^{k})-t_{\bar{\chi}}^{\varepsilon}(g)+t_{\bar{\chi}}^{\varepsilon}(h)+m_{1}.
\end{align}
Similarly let $\zeta_{\bar{\chi}}^{\varepsilon}(\ell_{2})\in{}\mathfrak{Sel}_{\ell_{1}\ell_{2}}^{\varepsilon}(K,T_{g}(\bar{\chi}))$
be a class such that the residue map at $\ell_{2}$ induces an isomorphism 
\[
             \partial_{\ell_{2}} : 
             \mathfrak{Sel}_{\ell_{1}\ell_{2}}^{\varepsilon}(K,T_{g}(\bar{\chi}))\big/
      \mathfrak{Sel}^{\varepsilon}_{\ell_{1}}(K,T_{g}(\bar{\chi}))\cong{}
      \partial_{\ell_{2}}\big(\zeta_{\bar{\chi}}^{\varepsilon}(\ell_{2})\big)\cdot{}\mathscr{O}_{\chi}/p^{k}.
\]
Because $\xi_{\bar{\chi}}^{\varepsilon}(\ell_{2})\in{}\mathfrak{Sel}^{\varepsilon}_{\ell_{1}\ell_{2}}(K,T_{g}(\bar{\chi}))$ by Lemma $\ref{basickol}(1)$, 
one can assume that there exists $m_{2}\ges{}0$ such that 
\[
            \xi_{\bar{\chi}}^{\varepsilon}(\ell_{2})-\varpi_{\chi}^{m_{2}}\cdot{}\zeta_{\bar{\chi}}^{\varepsilon}(\ell_{2})
            \in{}\mathfrak{Sel}^{\varepsilon}_{\ell_{1}}(K,T_{g}(\bar{\chi})),
\]
and apply as above Lemma $\ref{basickol}(2)$ and property $\mathbf{I}_{3}$ to deduce  the equality 
\begin{equation}\label{eq:step23}
         \mathrm{length}_{\mathscr{O}_{\chi}}\big(
     \mathfrak{Sel}_{\ell_{1}\ell_{2}}^{\varepsilon}(K,T_{g}(\bar{\chi}))\big/
      \mathfrak{Sel}^{\varepsilon}_{\ell_{1}}(K,T_{g}(\bar{\chi}))\big)=
      \mathrm{ord}_{\chi}(p^{k})-t_{\bar{\chi}}^{\varepsilon}(g)+t_{\bar{\chi}}^{\varepsilon}(h)+m_{2}.
\end{equation} 
When combined together Equations $(\ref{eq:step22})$ and $(\ref{eq:step23})$ give the equality 
\begin{equation}\label{eq:step24}
      \mathrm{length}_{\mathscr{O}_{\chi}}\big(
     \mathfrak{Sel}_{\ell_{1}\ell_{2}}^{\varepsilon}(K,T_{g}(\bar{\chi}))\big/
      \mathfrak{Sel}^{\varepsilon}(K,T_{g}(\bar{\chi}))\big)=
      2\cdot{}\mathrm{ord}_{\chi}(p^{k})-2\cdot{}t_{\bar{\chi}}^{\varepsilon}(g)
      +2\cdot{}t_{\bar{\chi}}^{\varepsilon}(h)+m_{1}+m_{2}.
\end{equation}
By construction $\mathrm{Sel}_{\varepsilon}^{\cdot}(K,A_{g}(\chi))$ is the dual Selmer group of 
$\mathfrak{Sel}_{\cdot}^{\varepsilon}(K,T_{g}(\bar{\chi}))$, hence Poitou--Tate duality gives a
short exact sequence   
of $\mathscr{O}_{\chi}/p^{k}$-modules (cf. Equation $(\ref{eq:PTStep3})$ in the proof of Step 2)
\[
      0\lfre{}\frac{\mathfrak{Sel}^{\varepsilon}_{\ell_{1}\ell_{2}}(K,T_{g}(\bar{\chi}))}
      {\mathfrak{Sel}^{\varepsilon}(K,T_{g}(\bar{\chi}))}\lfre{\partial_{\ell_{1}}\oplus{}\partial_{\ell_{2}}}
       \mathscr{O}_{\chi}/p^{k}\oplus{}\mathscr{O}_{\chi}/p^{k}
       \lfre{v_{\ell_{1}}^{\vee}\oplus{}v_{\ell_{2}}^{\vee}}\lri{\frac{\mathrm{Sel}_{\varepsilon}(K,A_{g}(\chi))}{\mathrm{Sel}_{\varepsilon}^{\ell_{1}\ell_{2}}(K,A_{g}(\chi))}}^{\vee}\lfre{}0,
\]
where for $\ell=\ell_{1},\ell_{2}$ one  identifies 
$H^{1}_{\mathrm{sing}}(K_{\ell},T_{g}(\bar{\chi}))\cong{}H^{1}_{\mathrm{fin}}
(K_{\ell},A_{g}(\chi))^{\vee}$ with $\mathscr{O}_{\chi}/p^{k}$ under a fixed isomorphism.   
Together with Equation  $(\ref{eq:step24})$ and property $\mathbf{I}_{2}$ this implies 
\begin{equation}\label{eq:step45}
           \mathrm{length}_{\mathscr{O}_{\chi}}\big(\mathrm{Sel}_{\varepsilon}(K,A_{g}(\chi))\big)
           -2\cdot{}t_{\bar{\chi}}^{\varepsilon}(g)
           =\mathrm{length}_{\mathscr{O}_{\chi}}\big(\mathrm{Sel}_{\varepsilon}(K,A_{h}(\chi))\big)
           -2\cdot{}t_{\bar{\chi}}^{\varepsilon}(h)-m_{1}-m_{2}.
\end{equation}
Properties  $\mathbf{I}_{1}$ and $\mathbf{I}_{3}$ give $t_{\bar{\chi}}^{\varepsilon}(h)<
t_{\bar{\chi}}^{\varepsilon}(g)$, hence 
\begin{equation}\label{eq:shshsh}
      \mathrm{length}_{\mathscr{O}_{\chi}}\big(\mathrm{Sel}_{\varepsilon}(K,A_{h}(\chi))\big)
           -2\cdot{}t_{\bar{\chi}}^{\varepsilon}(h)\les{}0
\end{equation}
by the induction hypothesis. The statement follows from Equations $(\ref{eq:step45})$ and $(\ref{eq:shshsh})$.
\end{proof}

\subsubsection{Step 4} Assume that $(f,K,p,\varepsilon)$ is not exceptional and that
$\mathrm{Sel}_{\varepsilon}(K,A_{g}(\chi))=0$.
Then $t_{\bar{\chi}}^{\varepsilon}(g)=0$.  

\begin{proof}  Theorem B of \cite{DT}
implies  that there exists a newform 
$\xi=\sum_{n=1}^{\infty}a_{n}(\xi)\cdot{}q^{n}$ in $S_{2}(\Gamma_{0}(NL))^{\mathrm{new}}$ which is congruent to 
$f$ modulo $p$. More precisely, if $\Q(\xi)$ denotes the field generated 
over $\Q$ by the Fourier coefficients of $\xi$, then  
there exists a prime $\bar{\mathfrak{P}}$ of $\bar{\Q}$
dividing $p$ such that 
$ a_{l}(\xi)\equiv{}a_{l}(E) \pmod{\bar{\mathfrak{P}}}$
for every rational prime $l\nmid{}NLp$.
(loc. cit.  proves the existence of an eigenform 
$\xi\in{}S_{2}(\Gamma_{1}(N)\cap{}\Gamma_{0}(L))$ of conductor divisible by $L$ which is congruent to $f$
modulo $p$.
It is not difficult to prove that an eigenform with these properties has trivial character and conductor $NL$.)
Let $J_{\xi}^{o}/\Q$ be the quotient of $\mathrm{Pic}^{0}(X_{0}(NL)/\Q)$ associated with  $\xi$
by the Eichler--Shimura construction. It is an abelian variety of dimension $[\Q(\xi):\Q]$ equipped with a morphism of $\Q$-algebras  $\Q(\xi)\fre{}\mathrm{End}_{\Q}(J_{\xi}^{o})\otimes_{\Z}\Q$.
Let $J_{\xi}/\Q$ be an abelian variety in the isogeny class of $J_{\xi}^{o}$ 
which has real multiplication by 
the ring of integers $\mathcal{O}$ of $\Q(\xi)$ 
and set $\mathfrak{P}=\bar{\mathfrak{P}}\cap{}\mathcal{O}$. 
Since $E_{p}$ is an irreducible $\F_{p}[G_{\Q}]$-module by Hypothesis $\ref{thehy}(1)$, the Eichler--Shimura relations and the Brauer--Nesbitt theorem   
imply that there are isomorphisms of $\mathcal{O}/\mathfrak{P}[G_{\Q}]$-modules 
\[
                      J_{\xi}[\mathfrak{P}]\cong{}E_{p}\otimes_{\F_{p}}\mathcal{O}/\mathfrak{P}
                      \cong{}A_{\bar{g}}\otimes_{\F_{p}}\mathcal{O}/\mathfrak{P}.
\]
Identify in what follows $J_{\xi}[\mathfrak{P}]$ and $A_{\bar{g}}\otimes_{\F_{p}}\mathcal{O}/\mathfrak{P}$
under a fixed isomorphism, and let 
\[\mathrm{Sel}_{\mathfrak{P}}(J_{\xi}/K)\subset{}H^{1}(K,A_{\bar{g}})\]
be the $\mathfrak{P}$-Selmer group of $J_{\xi}$ over $K$ (cf. \cite{G-P}). 
It follows from the results of \cite[Sections 3--5]{G-P} that 
\begin{equation}\label{eq:compSelmer}
               \mathrm{Sel}_{\mathfrak{P}}(J_{\xi}/K)=
               \mathrm{Sel}(K,A_{\bar{g}})\otimes_{\F_{p}}\mathcal{O}/\mathfrak{P}
\end{equation}
inside $H^{1}(K,A_{\bar{g}})\otimes_{\F_{p}}\mathcal{O}/\mathfrak{P}$, where 
$\mathrm{Sel}(K,A_{\bar{g}})$ is the Selmer group defined by imposing the finite  local condition 
$H^{1}_{\mathrm{fin}}(K_{\mathfrak{p}},A_{\bar{g}})
\cong{}E(K_{\mathfrak{p}})\otimes\F_{p}$
at every prime $\mathfrak{p}$ of $K$ dividing $p$ (viz. $\mathrm{Sel}(K,A_{\bar{g}})=
\mathrm{Sel}_{\emptyset}(K,A_{\bar{g}})$ with the notation of Section $\ref{selsec}$,
independently of whether $E$ has ordinary or supersingular reduction at $p$). Note that since $(f,K,p,\varepsilon)$ is not 
exceptional,
then by definition $E(K_{\mathfrak{p}})_{\varepsilon}=E(K_{\mathfrak{p}})$, hence  
\[
    H^{1}_{\mathrm{fin},\varepsilon}(K_\mathfrak{p},A_{\bar{g}})
=H^{1}_{\mathrm{fin},\varepsilon}(K_{\mathfrak{p}},\mathbf{A}_{f})[\mathfrak{m}_{\iw}]
=E(K_{\mathfrak{p}})_\varepsilon\otimes\F_{p}
    =E(K_{\mathfrak{p}})\otimes\F_{p}=H^{1}_{\mathrm{fin}}(K_{\mathfrak{p}},A_{\bar{g}})
\]
where the first equality follows from Corollary \ref{local control}(2) and the second from Proposition $\ref{hypext}$. It follows that 
$\mathrm{Sel}_{\varepsilon}(K,A_{\bar{g}})=\mathrm{Sel}(K,A_{\bar{g}})$. 
We have an isomorphism 
$H^1(K,A_{\bar{g}})\otimes_{\Z_p}\mathscr{O}_\chi\cong H^1(K,A_{\bar{g},{\mathscr{O}_\chi}})$ 
and an injection $\mathrm{Sel}(K,A_{\bar{g}})\otimes_{\Z_p}\mathscr{O}_\chi\hookrightarrow
H^1(K,A_{\bar{g}})\otimes_{\Z_p}\mathscr{O}_\chi$ by the flatness of $\mathscr{O}_\chi/\Z_p$, 
and therefore  $\mathrm{Sel}(K,A_{\bar{g}})\otimes_{\Z_p}\mathscr{O}_\chi$ injects into $\mathrm{Sel}(K,A_{\bar{g}}(\chi))$. 
Since by assumption $\mathrm{Sel}_\varepsilon(K,A_{g}(\chi))$ is trivial, it follows from  Proposition \ref{global control} and the irreducibility of $\bar{\varrho}_{E,p}$ that the same is true for 
$\mathrm{Sel}(K,A_{\bar{g}})$
and
Equation \eqref{eq:compSelmer} 
gives 
\begin{equation}\label{eq:trivP}
                  \mathrm{Sel}_{\mathfrak{P}}(J_{\xi}/K)=0.
\end{equation}
Let $L(\xi/K,1)_{\mathrm{alg}}$ denote the algebraic part of the special value of 
the complex $L$-function of $\xi$ over $K$, normalized as 
in  \cite[Section 4]{BBV}. 
Results of Skinner--Urban and Fouquet--Wan 
prove the inequality 
\begin{equation}\label{BSDequation}
                   \mathrm{ord}_{\mathfrak{P}}\big(L(\xi/K,1)_{\mathrm{alg}}\big)\les{}\mathrm{length}_{\mathcal{O}_{\mathfrak{P}}}
                   \big(                   \mathrm{Sel}_{\mathfrak{P}^{\infty}}(J_{\xi}/K)\big)+\sum_{q|NL}t_{\xi}(q), 
\end{equation}
where $t_\xi(q)$ is the Tamagawa exponent appearing in \cite[Section 4]{BBV}. 
See \cite{S-U} in the ordinary case; for the non-ordinary case, the reader is referred to \cite[Corollaries 1.9, 1.10]{FW} and also \cite[Theorems 1.5, 1.6]{BSTW}, \cite[Theorem C]{CCSS}. 

Since $J_{\xi}[\mathfrak{P}]$
is an irreducible $G_{K}$-module, $\mathrm{Sel}_{\mathfrak{P}}(J_{\xi}/K)$
 is equal to the $\mathfrak{P}$-torsion submodule of the Selmer group 
 $\mathrm{Sel}_{\mathfrak{P}^{\infty}}(J_{\xi}/K)$,
so that $\mathrm{Sel}_{\mathfrak{P}^{\infty}}(J_{\xi}/K)$ is trivial  by Equation $(\ref{eq:trivP})$.
In addition $t_{\xi}(q)=0$ for every prime $q|N^{+}$ under our assumptions, and  the previous equation yields 
\[
             \mathrm{ord}_{\mathfrak{P}}\big(L(\xi/K,1)_{\mathrm{alg}}\big)\les{}\sum_{q|LN^{-}}t_{\xi}(q).
\]
On the other hand, Gross's formula (see Theorem 4.2 of \cite{BBV} for the formulation 
in the form required in this paper) gives the identity 
\[
                 \mathrm{ord}_{\mathfrak{P}}\big(L(\xi/K,1)_{\mathrm{alg}}\big)
                 =2\cdot{}\mathrm{ord}_{\mathfrak{P}}\big(\psi_{\xi}\big(P_{K}(L)\big)\big)
                 +\sum_{q|LN^{-}}t_{\xi}(q).
\]

It follows combining the two previous formulas that  $\psi_{\xi}\big(P_{K}(L)\big)$ has trivial $\mathfrak{P}$-adic valuation:
\begin{equation}\label{eq:Stepmain4}      
                             \psi_{\xi}\big(P_{K}(L)\big)\in{}\mathcal{O}_{\mathfrak{P}}^{\ast}. 
\end{equation}
Since $\xi$ is congruent to $f$ modulo $p$,
one has   $\psi_{\xi}\big(P_{K}(L)\big)\equiv{}\psi_{\bar{g}}\big(P_{K}(L)\big)
\pmod{\mathfrak{P}}$. In addition, as
$(f,K,p,\varepsilon)$ is not exceptional,  Lemmas $\ref{nonexcord}$
and $\ref{nonexcss}$ show that the equalities 
$\psi_{\bar{g}}\big(P_{K}(L)\big)=\mathcal{L}_{\bar{g}}^{\varepsilon}(\mathbf{1})=\mathcal{L}_{g}^{\varepsilon}(\mathbf{1})
\pmod{p}$ hold in $\F_{p}$ up to multiplication by non-zero elements. 
Equation $(\ref{eq:Stepmain4})$ then yields 
$\mathcal{L}_{g}^{\varepsilon}(\mathbf{1})\in{}\Z_{p}^{\ast}.$ 
This implies that the $p$-adic $L$-function $\mathcal{L}_{g}^{\varepsilon}$ is a unit in 
$\iw/p^{k}$, which in turn gives  $t_{\bar{\chi}}^{\varepsilon}(g)=0$.
\end{proof}

\subsubsection{Step 5}\label{step5} There exist an $\mathscr{O}_{\chi}/p^{k}$-module $\texttt{M}$
and an integer $s\in\{0,1\}$ such that 
\[
                 \mathrm{Sel}_{\varepsilon}(K,A_{g}(\chi))\cong{}(\mathscr{O}_{\chi}/p^{k})^{s}\oplus\mathtt{M}\oplus\mathtt{M}\cong{}\mathfrak{Sel}^{\varepsilon}(K,T_{g}({\chi})).
\]
\begin{proof} Thanks to the isomorphism  $w_{\chi} : \mathfrak{Sel}^{\varepsilon}(K,T_{g}(\chi))\simeq{}\mathrm{Sel}_{\varepsilon}(K,A_{g}(\chi))$ of Proposition \ref{selfdual}, it is enough to prove the statement for $\mathrm{Sel}_{\varepsilon}(K,T_{g}(\chi))$. This follows from the results of \cite[Section 2.6]{How-heeg}
(which in turn grounds on Section 1.4 of \cite{How-HK}). Precisely, the local conditions 
defining the Selmer groups $\{\mathfrak{Sel}^{\varepsilon}(K,T_{h}(\chi))\}_{h}$, for $h$ varying through the reductions of $g$ modulo $p^{j}$ with $1\les{}j\les{}k$,  are cartesian in the sense of 
Definition 2.2.2 of \cite{How-heeg}: this is easily verified for the primes of $K$ not dividing $p$ (where the local condition are unramified, ordinary or the relevant Galois representations are cohomologically trivial), and follows from the local control theorems of Section $\ref{locsecp}$ for the primes of $K$ dividing $p$. 
Moreover, for each prime $v$ of $K$, denote by $\bar{v}$ the complex conjugate of $v$, and by  $\mathcal{L}_{v}$  the local condition at $v$ satisfied by the Selmer classes in 
$\mathfrak{Sel}^{\varepsilon}(K,T_{g}(\chi))$. Then Proposition $\ref{selfdual}$ and Remark $\ref{remchichibar}$ prove that $\mathcal{L}_{v}$  is the exact 
orthogonal complement of $\mathcal{L}_{\bar{v}}$ under the pairing $H^{1}(K_{v},T_{g}(\chi))\times{}H^{1}(K_{\bar{v}},T_{g}(\chi))\rightarrow{}\mathscr{O}_{\chi}/p^{k}$ induced by the (not $G_{K_{v}}$-equivariant) bilinear map 
$T_{g}(\chi)\times{}T_{g}(\chi)\rightarrow{}\mathscr{O}_{\chi}/p^{k}(1)$ arising from the Weil pairing (cf.\ Equation $(4)$ in Section 2.6 of \cite{How-heeg}).
As explained in Section 2.6 of \cite{How-heeg}, we can then apply \cite[Proposition 2.2.7]{How-heeg} and conclude the proof of Step 5.
\end{proof}

\subsubsection{Step 6} Assume that $(f,K,p,\varepsilon)$ is not exceptional and let $\ell_{1}$ and $\ell_{2}$ be  $2k$-admissible primes which satisfy the conditions 
$\mathbf{I}_{1}$--$\mathbf{I}_{4}$ (cf. Step 2). Then (with the notation of loc. cit. )
\begin{equation}\label{eq:equality}
\mathrm{length}_{\mathscr{O}_{\chi}}\big(\mathrm{Sel}_{\varepsilon}(K,A_{g}(\chi))\big)
           -2\cdot{}t_{\bar{\chi}}^{\varepsilon}(g)
           =\mathrm{length}_{\mathscr{O}_{\chi}}\big(\mathrm{Sel}_{\varepsilon}(K,A_{h}(\chi))\big)
           -2\cdot{}t_{\bar{\chi}}^{\varepsilon}(h).
\end{equation}

\begin{proof} We first prove that the dimension of $\mathrm{Sel}_{\varepsilon}(K,A_{\bar{g}})$ over 
$\F_{p}$ is even:
\begin{equation}\label{eq:parity}
         \dim_{\F_{p}}\big(\mathrm{Sel}_{\varepsilon}(K,A_{\bar{g}})\big)\equiv{}0 \pmod{2}.
\end{equation}
Let $x\in{}\mathrm{Sel}_{\varepsilon}(K,A_{\bar{g}})$ be a nonzero class. 
Choose an admissible prime $\ell$ such that \[v_{\ell}(x_{1})\in{}
H^{1}_{\mathrm{fin}}(K_{\ell_{1}},A_{\bar{g}})\cong{}\F_{p}\] is non-zero,
which exists by Theorem 3.2 of \cite{Be-Da-main}. Let $h\in{}S_{2}(N^{+},\ell{}LN^{-};\F_{p})$ be
the $\ell$-level raising of $\bar{g}$. Note that 
$\mathrm{Sel}_{\varepsilon}(K,A_{\bar{g}})$ is identified with $\mathfrak{Sel}^{\varepsilon}(K,T_{\bar{g}})$
under the isomorphism $T_{\bar{g}}\cong{}A_{\bar{g}}$ induced by the Weil pairing,
viz. $\mathrm{Sel}_{\varepsilon}(K,A_{\bar{g}})$ is equal to its own dual Selmer group. 
(This can either  be seen as a special case  of Step $5$ or, more simply,  follows from the discussion in the proof of Step $4$
under the current assumptions.)
As in the proof of Step 1, Poitou--Tate duality then implies that 
$\mathrm{Sel}_{\varepsilon}(K,A_{h})$ is equal to $\mathrm{Sel}_{\varepsilon}^{\ell}(K,A_{\bar{g}})$,
hence 
\[
         \dim_{\F_{p}}\big(\mathrm{Sel}_{\varepsilon}(K,A_{h})\big)=
                  \dim_{\F_{p}}\big(\mathrm{Sel}_{\varepsilon}(K,A_{\bar{g}})\big)-1
\]
since $v_{\ell}(x)\not=0$. If $\mathrm{Sel}_{\varepsilon}(K,A_{h})\not=0$
we can apply the same argument after replacing $\bar{g}$ with $h$. 
In this way one constructs a squarefree product $T\in{}\mathscr{S}_{1}$ of 
$\dim_{\F_{p}}\mathrm{Sel}_{\varepsilon}(K,A_{\bar{g}})$ admissible primes 
such that $\mathrm{Sel}_{\varepsilon}(K,A_{h})=0$, where $h\in{}S_{2}(N^{+},TLN^{-};\F_{p})$
denotes now the  $T$-level raising of $\bar{g}$. 
As in the proof of Step 4, the results of Skinner--Urban and Fouquet-Wan then implies that $L(\xi/K,1)\not=0$,
where $\xi$ is a newform of weight $\Gamma_{0}(TLN)$
which is congruent to $f$ modulo $p$. As a consequence $$-1=\epsilon_{K}(TLN)=
\epsilon_{K}(LN^{-})\cdot{}(-1)^{\dim_{\F_{p}}\big(\mathrm{Sel}_{\varepsilon}(K,A_{\bar{g}})\big)},$$ and since 
by assumption $LN^{-}$
has an \emph{odd} number of prime divisors, this proves $(\ref{eq:parity})$.

We now show \eqref{eq:equality}, which is equivalent to show that the integers $m_1$ and $m_2$ in Equation \eqref{eq:step45} 
are both equal to $0$.  
Preliminarily, note that if $\mathrm{length}_{\mathscr{O}_{\chi}}\big(\mathrm{Sel}_{\varepsilon}(K,A_{g}(\chi))\big)=
2\cdot{}t_{\bar{\chi}}^{\varepsilon}(g)$, then, since 
           $\mathrm{length}_{\mathscr{O}_{\chi}}\big(\mathrm{Sel}_{\varepsilon}(K,A_{h}(\chi))\big)
           \les2\cdot{}t_{\bar{\chi}}^{\varepsilon}(h)$ by Step 3, we have $m_1+m_2=0$ (where $m_1$ and $m_2$ are defined in Step 3; we also have 
           $\mathrm{length}_{\mathscr{O}_{\chi}}\big(\mathrm{Sel}_{\varepsilon}(K,A_{h}(\chi))\big)
           =2\cdot{}t_{\bar{\chi}}^{\varepsilon}(h)$) 
directly from Equation \eqref{eq:step45}. Therefore we assume in the following that 
$\mathrm{length}_{\mathscr{O}_{\chi}}\big(\mathrm{Sel}_{\varepsilon}(K,A_{g}(\chi))\big)<
           2\cdot{}t_{\bar{\chi}}^{\varepsilon}(g)$. 

We first show that $m_1=0$. By \eqref{eq:nontriv ass}, $t_{\bar{\chi}}^{\varepsilon}(g)<\ord_\chi(p^k)$. Since 
the $\F_p$-dimension of $\mathrm{Sel}_{\varepsilon}(K,A_{\bar{g}})$ is even, combining Step 5 and Nakayama Lemma 
shows that 
\[\mathrm{Sel}_{\varepsilon}(K,A_{g}(\chi))\cong{}\mathtt{M}\oplus\mathtt{M}.\] 
If follows that $\mathrm{length}_{\mathscr O_\chi}(\mathtt{M})<\ord_\chi(p^k)$, and therefore 
\[\varpi_\chi^{\ord_\chi(p^k)-1}\cdot\mathrm{Sel}_{\varepsilon}(K,A_{g}(\chi))=0.\] 
We now consider the class  
$\xi_{\bar{\chi}}^{\varepsilon}(\ell_{1})-\varpi_{\chi}^{m_{1}}\cdot{}\zeta_{\bar{\chi}}^{\varepsilon}(\ell_{1})$ in 
            $\mathfrak{Sel}^{\varepsilon}(K,T_{g}(\bar{\chi}))$
appearing in the proof of Step 3. Since $\varpi^{\ord_\chi(p^k)-1}$ kills $\mathrm{Sel}_{\varepsilon}(K,A_{g}(\chi))$, and 
since $\mathrm{Sel}_{\varepsilon}(K,A_{g}(\chi))$ and $\mathfrak{Sel}^{\varepsilon}(K,T_{g}(\bar{\chi}))$ are dual to each other, 
the same is true for $\mathfrak{Sel}^{\varepsilon}(K,T_{g}(\bar{\chi}))$. Therefore 
we obtain the equality
\begin{equation}\label{eq:equality-1}
\varpi_\chi^{\ord_\chi(p^k)-1}\cdot\xi_{\bar{\chi}}^{\varepsilon}(\ell_{1})=\varpi_{\chi}^{\ord_\chi(p^k)-1+m_{1}}\cdot\zeta_{\bar{\chi}}^{\varepsilon}(\ell_{1}).\end{equation}
We now show that the left hand side of this equality is always non-trivial. 
First, enlarge $\{\ell_1\}$ to a freeing set $S$ as in Section \ref{glofree}; 
then by Proposition \ref{global freeness}, $\mathfrak{Sel}_S^{\varepsilon}\big(K,T_g(\chi)\big)$ is 
free over $\mathscr O_\chi/p^k$ of rank $\delta(S)$, the number of prime divisors of $S$. By  Lemma $\ref{basickol}(3)$, 
the class $\bar{\xi}_{\chi}^{\varepsilon}(\ell)$ in $\mathfrak{Sel}_{S}^{\varepsilon}(K,T_{\bar{g}})\otimes_{\F_{p}}\F_{\chi}$
is not trivial, therefore $\varpi_\chi$ does not divide ${\xi}_{\chi}^{\varepsilon}(\ell)$ and 
it follows that $\varpi_\chi^{\ord_\chi(p^k)-1}\cdot\xi_{\bar{\chi}}^{\varepsilon}(\ell_{1})\neq 0$ from 
the freeness result
$\mathfrak{Sel}_S^{\varepsilon}\big(K,T_g(\chi)\big)\cong (\mathscr O_\chi/p^k)^{\delta(S)}$ recalled above. 
On the other hand, if $m_1>0$, then the right hand side of \eqref{eq:equality-1} is zero, which is a contradition. 
Therefore $m_1$ must be equal to $0$.

We now show that $m_2=0$ with a similar argument. 
Consider the class  
$\xi_{\bar{\chi}}^{\varepsilon}(\ell_{2})-\varpi_{\chi}^{m_{2}}\cdot{}\zeta_{\bar{\chi}}^{\varepsilon}(\ell_{2})$
in $\mathfrak{Sel}^{\varepsilon}_{\ell_{1}}(K,T_{g}(\bar{\chi}))$
appearing in the proof of Step 3. 
Since $m_1=0$, we know that $\partial_{\ell_1}(\xi_{\bar\chi}^\varepsilon(\ell_1))$ 
generates the image of the residue map 
$\partial_{\ell_1}:\mathfrak{Sel}_{\ell_{1}}^{\varepsilon}(K,T_{g}(\bar{\chi}))
\fre{}H^{1}_{\mathrm{sing}}(K_{\ell},T_{g}(\bar{\chi}))$, and therefore there exists an integer $m_3\ges 0$ such that the class 
$
\partial_{\ell_1}\left(\xi_{\bar{\chi}}^{\varepsilon}(\ell_{2})-\varpi_{\chi}^{m_{2}}\cdot{}\zeta_{\bar{\chi}}^{\varepsilon}(\ell_{2})\right)$ is 
equal to the class 
$\varpi_\chi^{m_3}\cdot \partial_{\ell_1}(\xi_{\bar\chi}^\varepsilon(\ell_1))$, {i.e.}  
\[\partial_{\ell_1}\left(\xi_{\bar{\chi}}^{\varepsilon}(\ell_{2})-\varpi_{\chi}^{m_{2}}\cdot{}\zeta_{\bar{\chi}}^{\varepsilon}(\ell_{2})-
\varpi_\chi^{m_3}\cdot \xi_{\bar\chi}^\varepsilon(\ell_1)\right)=0.\] 
Therefore by definition the class 
$\xi_{\bar{\chi}}^{\varepsilon}(\ell_{2})-\varpi_{\chi}^{m_{2}}\cdot{}\zeta_{\bar{\chi}}^{\varepsilon}(\ell_{2})-
\varpi_\chi^{m_3}\cdot \xi_{\bar\chi}^\varepsilon(\ell_1)$ 
belongs to $\mathfrak{Sel}^{\varepsilon}(K,T_{g}(\bar{\chi}))$. Since
this group is annihilated by $\varpi_\chi^{\ord_\chi(p^k)-1}$ we obtain the equality 
\[
\varpi_\chi^{\ord_\chi(p^k)-1}\cdot\xi_{\bar{\chi}}^{\varepsilon}(\ell_{2})-\varpi_{\chi}^{\ord_\chi(p^k)-1+m_{2}}\cdot{}\zeta_{\bar{\chi}}^{\varepsilon}(\ell_{2})=
\varpi_\chi^{\ord_\chi(p^k)-1+m_3}\cdot \xi_{\bar\chi}^\varepsilon(\ell_1).
\]
We now suppose \emph{ad absurdum} that $m_2>0$. Then the equation above implies 
\begin{equation}\label{eq:equality-2} 
\varpi_\chi^{\ord_\chi(p^k)-1}\cdot\xi_{\bar{\chi}}^{\varepsilon}(\ell_{2})=
\varpi_\chi^{\ord_\chi(p^k)-1+m_3}\cdot \xi_{\bar\chi}^\varepsilon(\ell_1).\end{equation}
By $\mathbf{I}_4$, 
$\mathrm{ord}_{\chi}\big(v_{\ell_{1}}\big(\xi_{\bar{\chi}}^{\varepsilon}(\ell_{2})\big)\big)=0$, and therefore, 
again using the freeness argument as above, we see that 
$\varpi_\chi^{\ord_\chi(p^k)-1}\cdot\xi_{\bar{\chi}}^{\varepsilon}(\ell_{2})$ is not trivial. Equation \eqref{eq:equality-2} 
then shows that $m_3=0$. Therefore, applying $v_{\ell_1}$ to Equation \eqref{eq:equality-2}, we obtain the equality  
\[\varpi_\chi^{\ord_\chi(p^k)-1}\cdot v_{\ell_1}\big(\xi_{\bar\chi}^\varepsilon(\ell_2)\big)= 
\varpi_\chi^{\ord_\chi(p^k)-1}\cdot v_{\ell_1}\big(\xi_{\bar\chi}^\varepsilon(\ell_1)\big)
\] By $\mathbf{I}_4$, the left hand side of this equality is not trivial, while by Lemma $\ref{basickol}(1)$, the right hand side is 
trivial, which is a contradiction. Therefore, $m_2=0$, concluding the proof of Step 6. 
\end{proof}

\subsubsection{Step 7} Assume that $(f,K,p,\varepsilon)$ is not exceptional. Then 
\[
          \mathrm{length}_{\mathscr{O}_{\chi}}\big(\mathrm{Sel}_{\varepsilon}(K,A_{g}(\chi))\big)
               =2\cdot{}t_{\bar{\chi}}^{\varepsilon}(g).
\]

\begin{proof}
The proof is by induction on $\mathrm{length}_{\mathscr{O}_{\chi}}\big(\mathrm{Sel}_{\varepsilon}(K,A_{g}(\chi))\big)$. 
If $\mathrm{length}_{\mathscr{O}_{\chi}}\big(\mathrm{Sel}_{\varepsilon}(K,A_{g}(\chi))\big)=0$, then the equality follows 
from Step 4. 
When $\mathrm{Sel}_{\varepsilon}(K,A_{g}(\chi))$ is not trivial, choose a pair of $2k$-admissible primes $\ell_1$ and $\ell_2$ satisfying 
conditions $\mathbf{I}_{1}$-$\mathbf{I}_{4}$, and let $h$ be the $\ell_1\ell_2$-level raising of $g$. 
Since $t_{\bar \chi}^\varepsilon(h)<t_{\bar \chi}^\varepsilon(g)$, we see from Step 6 that 
$\mathrm{length}_{\mathscr{O}_{\chi}}\big(\mathrm{Sel}_{\varepsilon}(K,A_{h}(\chi))\big)$ is strictly smaller than 
$\mathrm{length}_{\mathscr{O}_{\chi}}\big(\mathrm{Sel}_{\varepsilon}(K,A_{g}(\chi))\big)$, and therefore by the inductive hypothesis 
$\mathrm{length}_{\mathscr{O}_{\chi}}\big(\mathrm{Sel}_{\varepsilon}(K,A_{h}(\chi))\big)=2t_{\bar \chi}^\varepsilon(h)$. 
A further application of the equality in Step 6 implies then the result.
\end{proof}

 \subsection{Proof of Theorem A in the definite case} \label{def:definite_char}
 Let $\mathfrak X_p^\varepsilon(f)$ be the Pontryagin dual of $\mathrm{Sel}_\varepsilon(K,\mathbf{A}_f)$, which is a compact torsion $\Lambda$-module by Theorem \ref{BSD0}. 
Denote $\mathrm{Char}_p^\varepsilon(f)$ the characteristic power series of  
 $\mathfrak X_p^\varepsilon(f)$. 
 
\begin{theo}[DAMC]\label{Adef}
$(L_p^\varepsilon(f))\subseteq(\mathrm{Char}_p^\varepsilon(f))$, with equality in the non-exceptional case. 
\end{theo}

\begin{proof} The proof easily follows by combining Theorem \ref{BSD0}  with Proposition \ref{propMRH} and repeating the argument in Mazur--Rubin \cite[Section 5.3]{MR-KolSys} and Howard \cite[Section 2.2]{How-HK}. 
\end{proof}

\section{$\varepsilon$-BSD formulas in the indefinite case}
\label{sec:BSD-indefinite}
Assume that $\mathscr{B}$ is \emph{indefinite}
(i.e. $\epsilon_{K}(N^{-})=+1$). 

\begin{proposition}\label{freenes-Sel(T)}
The compact $\Lambda_{\mathscr{O}}$-module $\mathfrak{Sel}^\varepsilon(K,\mathbf{T}_{f,\mathscr{O}})$ is free of finite rank. 
\end{proposition}

\begin{proof}  
By Proposition \ref{speinj} (for $\mathfrak P$ equal to the augmentation ideal of $\Lambda_{\mathscr{O}}$) and Shapiro's Lemma, 
the $\Lambda_{\mathscr{O}}$-quotient 
of $G_\infty$-coinvariants of $\mathfrak{Sel}^\varepsilon(K,\mathbf{T}_{f,\mathscr{O}})$  
injects into $\mathfrak{Sel}^\varepsilon(K,T_{f,{\mathscr{O}}})$, and therefore is ${\mathscr{O}}$-free. 
Thanks to $E_p(K)=0$, we have $\mathbf{T}_{f,\mathscr{O}}^{G_\infty}=0$, so 
the $\Lambda_\mathscr{O}$-module 
$\mathfrak{Sel}^\varepsilon(K,\mathbf{T}_{f,\mathscr{O}})$ is torsion free by \cite[Lemma 1.3.3]{PR-SMF}, 
hence its $\Lambda_\mathscr{O}$-submodule of 
$G_\infty$-invariants is trivial. The result follows then from a standard argument 
(\emph{e.g.} \cite[Proposition 5.3.19(ii)]{NSW}).  
\end{proof}

\subsection{$\Lambda$-adic classes} 
We first construct global classes, in a way similar to Section \ref{subsec:indef}. 
Since $\epsilon_{K}(N^{-})=+1$, $J_{N^{+},N^{-}}$ 
is the Picard variety of the {Shimura curve} $X_{N^{+},N^{-}}$. 
Let $I_{f}\subset{}\mathbf{T}_{N^{+},N^{-}}$ denote the kernel of $f$. 
Modularity implies that 
there is an isomorphism of $\Z_{p}[G_{\Q}]$-modules 
\[
       \pi_{f} :  \mathrm{Ta}_{p}(J_{N^{+},N^{-}})/I_{f}\cong{}T_{f},
\]
which is unique up to multiplication by a $p$-adic unit by Hypothesis $\ref{thehy}(1)$.
For every integer $n\ges{}0$ define 
\[
        \psi_{f,n} : J_{N^{+},N^{-}}(K_{n})\lfre{}
        H^{1}(K_{n},\mathrm{Ta}_{p}(J_{N^{+},N^{-}})/I_{f})\cong{}H^{1}(K_{n},T_{f}),
\]
where the first (resp., second) map is induced by the Kummer map (resp., by $\pi_{f}$).
For every point $x\in{}J_{N^{+},N^{-}}(K_{n})$ the class 
$\psi_{f,n}(x)$ is finite at every prime of $K_{n}$ dividing $p$.
Moreover, since $J_{N^{+},N^{-}}$ has purely toric reduction at every prime divisor of $N^{-}$,  
Mumford--Tate theory of $p$-adic uniformisation implies that  
these classes are ordinary at every such prime. 
Therefore, we obtain a map \[
             \psi_{f,n}: J_{N^{+},LN^{-}}(K_{n})\lfre{}\mathfrak{Sel}(K_{n},T_{f}).
\]
Recall the compatible sequence of Heegner points $P_n=P_n(1)$, for $n\ges 0$, introduced in Section \ref{compheeg} and define 
\[\tilde{\kappa}_n=\psi_{f,n}(P_n).\]   

\subsubsection{Ordinary case} Suppose that $E$ has ordinary 
reduction at $p$. The classes  
\[
\kappa_n=
\frac{1}{\alpha_{p}(g)^{n}}
            \Big(\tilde\kappa_{n-1}-\alpha_{p}(g)\cdot{}
            \tilde\kappa_n\big)\Big) 
\]
belong to $\mathfrak{Sel}(K_{n},T_{f})$ by the previous discussion, and
Equation \eqref{eq:compdef} shows that they are norm-compatible. As in Section \ref{subsec:indef}, define 
\begin{equation}\label{classes-ord-bis}\kappa_{\infty} = \inlim_n\kappa_{n} \in
\inlim_{n}\mathfrak{Sel}(K_{n},T_{f})\end{equation} 
where the inverse limit is computed with respect to the canonical norm maps.   

\subsubsection{Supersingular case}
Using the freeness result of Proposition \ref{freenes-Sel(T)}, by the same argument in Section \ref{subsec:indef} one can define classes 
\[
\tilde\kappa_n^\varepsilon\in \mathfrak{Sel}^{\varepsilon}(K_{n},T_{f})/\omega_n^\varepsilon\]
such that 
$\tilde\omega_n^{-\varepsilon}\cdot \tilde\kappa_n^\varepsilon=\tilde{\kappa}_n$ if $p$ is split in $K$ or $p$ is inert in $K$ and $\varepsilon=-1$ (the non-exceptional case), and $ \omega_n^{-}\cdot\tilde \kappa_n^+=\tilde\kappa_n$ if $p$ is inert in $K$ and 
$\varepsilon=+1$ (the exceptional case). 
Define $\kappa_n^+=(-1)^{n/2}\tilde\kappa_n^+$ if $n$ is even 
and $\kappa_n^-=(-1)^{(n-1)/2}\tilde\kappa_n^-$ if $n$ is odd. A calculation using Equation \eqref{eq:compdef} shows that the classes $\kappa_n^\varepsilon$ are compatible with respect to the canonical projection maps. 
Define as in Section \ref{subsec:indef} \[\kappa_{\infty}^{\varepsilon} = \inlim_n\kappa_{n}^{\varepsilon} \in
\inlim_{n\in\N^\varepsilon}\mathfrak{Sel}^{\varepsilon}(K_{n},T_{f})/\omega_n^\varepsilon,\] 
where $\N^\varepsilon$ is the set of positive integers verifying the condition $(-1)^n=\varepsilon$. 

\subsection{Lengths of Selmer groups}
Fix a morphism $\chi:\Lambda\rightarrow \mathscr O_\chi$ of $\Z_p$-algebras, 
where as above $\mathscr{O}_\chi$ is the integral closure of $\Lambda/\mathfrak{P}_\chi$, and $\mathfrak{P}_\chi=\ker(\chi)$. 
Denote 
\[\kappa_{\chi}^\varepsilon\in 
\mathfrak{Sel}^\varepsilon(K,\mathbf{T}_{f})\otimes\mathscr O_\chi\overset{\mathrm{def}}=
\mathfrak{Sel}^\varepsilon(K,\mathbf{T}_{f,{\mathscr{O}_\chi}})\otimes_{\Lambda_{\mathscr{O}_\chi}}\mathscr O_\chi\]
the image of $\kappa_{\infty}^\varepsilon$ via the canonical map described above, where recall that 
the tensor product $\otimes_{\Lambda_{\mathscr{O}_\chi}}$ is taken with respect to $\chi$.  
We assume 
that $\ord_\chi(\kappa_{\chi}^\varepsilon)$ is finite. 
Using that $\mathfrak{Sel}^\varepsilon(K,\mathbf{T}_{f,{\mathscr{O}_\chi}})$ is $\Lambda_{\mathscr{O}_\chi}$-free by Proposition \ref{freenes-Sel(T)}, define the integer  
\[t_\chi^\varepsilon(f)=\ord_\chi\left(\kappa_{\chi}^\varepsilon\right)<\infty.\] 
For any group $p$-power torsion group $G$, let $G_{/\mathrm{div}}$ denote the quotient of $G$ by its maximal $p$-divisible subgroup. 

\begin{theo} \label{BSD-Indef} Suppose that $t_\chi^\varepsilon(f)<\infty$. Then 
 the $\mathscr{O}_\chi$-corank of $\mathrm{Sel}_\varepsilon(K,A_{f}(\chi))$ 
is $1$ and we have 
\[\mathrm{length}_{\mathscr O_\chi}
\left(\mathrm{Sel}_\varepsilon(K,A_{f}(\chi))_{/\mathrm{div}}\right)\les2\cdot 
\mathrm{length}_{\mathscr O_\chi}
\left((\mathfrak{Sel}^{\varepsilon}(K,\mathbf{T}_{f})\otimes
\mathscr{O}_{\chi})
/\mathscr O_\chi\cdot \kappa_{\bar\chi}^\varepsilon\right),\]
and the equality holds if  $(f,K,p,\varepsilon)$ is not exceptional. \end{theo}

\begin{proof} It follows from the freeness of  
$\mathfrak{Sel}^{\varepsilon}(K,\mathbf{T}_{f,\mathscr O_\chi})$
that there exists 
$\tilde\kappa_{\chi}^{\varepsilon}$ in $\mathfrak{Sel}^{\varepsilon}(K,\mathbf{T}_{f})\otimes\mathscr{O}_{\chi}$ such that 
$
\mathrm{ord}_{\chi}(\tilde\kappa_{\chi}^{\varepsilon})=0$ and 
$
\kappa_{\chi}^{\varepsilon}=
\varpi_{\chi}^{t_{\chi}^{\varepsilon}(f)}\cdot{}\tilde{\kappa}_{\chi}^{\varepsilon}.$ Define
           $\xi_{\chi}^{\varepsilon}
           \in{}\mathfrak{Sel}^{\varepsilon}(K,{T}_{f}(\chi))$
to be the image of $\tilde{\kappa}_{\chi}^{\varepsilon}$ under the (injective) specialization map 
$\texttt{s}_{\chi} : \mathfrak{Sel}^{\varepsilon}_S(K,\mathbf{T}_{f})\otimes\mathscr{O}_{\chi}
\hookrightarrow\mathfrak{Sel}^{\varepsilon}(K,{T}_{f}(\chi))$. We also denote $\kappa_{\chi,k}^\varepsilon$ the image of $\kappa_\chi^\varepsilon$ in $\mathfrak{Sel}^\varepsilon(K,T_{f,k}(\chi))$, 
for all integers $k\geqslant 1$, and $\xi^\varepsilon_{\chi,k}$ the image of $\xi^\varepsilon_\chi$ in 
$H^1(K,T_{f,k}(\chi))$. 
If $k=1$, the element $\xi^\varepsilon_{\chi,1}$ will be denoted $\bar\xi^\varepsilon_{\chi}$.  
As before (\emph{cf.} Step 5 in \S\ref{step5}) we have 
\[
                 \mathrm{Sel}_{\varepsilon}(K,A_{f}(\chi))\cong{}(\mathscr{K}_\chi/\mathscr{O}_{\chi})^{s}\oplus M_\chi\oplus M_\chi.
\] for some integer $s$ and a finite torsion $\mathscr{O}_\chi$-module $M_\chi$.  
Choose an integer \[k>\max\{\mathrm{length}_{\mathscr O_\chi}(M_\chi), t_{\bar{\chi}}^\varepsilon(f)\}.\] 
Using \cite[Theorem 3.2]{Be-Da-main}, choose an admissible prime 
$\ell\in \mathscr S_{k}$ such that 
$v_\ell(\bar\xi_{\bar\chi}^\varepsilon)\neq 0$.
Let $g$ be the $\ell$-level raising of $f$. 
Then since
$H^1_{\mathrm{fin}}(K_\ell,T_{f,\mathscr{O}_\chi})$ is a free $\mathscr{O}_\chi$-module of rank $1$, 
and  $v_\ell(\bar\xi_{\bar\chi}^\varepsilon)\neq 0$, using Proposition 
\ref{local control}, 
we have 
\begin{equation}\label{orders}
\ord_\chi\left(v_{\ell}\left(\kappa_{\bar\chi,k}^\varepsilon\right)\right)=\ord_\chi\left(v_{\ell}\left(\kappa_{\bar\chi}^\varepsilon\right)\right).\end{equation}

\emph{Step 1.} Theorem \ref{BSD0} for $g$ shows that 
\[\mathrm{length}_{\mathscr O_\chi}\left(\mathrm{Sel}_\varepsilon(K,A_g(\chi))\right)\les
 2\cdot\ord_\chi\left(\mathcal L_g^\varepsilon(\bar\chi)\right),\] with equality in the non-exceptional case, and 
Theorem \ref{secondrec} shows that 
\[\ord_\chi\left(v_{\ell}\left(\kappa_{\bar\chi,k}^\varepsilon\right)\right)=\ord_\chi\left(\mathcal L_g^\varepsilon(\bar\chi)\right).\]
Thus, by \eqref{orders} and the injectivity of the map $v_\ell$ (which follows from 
$v_\ell(\bar\xi_{\bar\chi}^\varepsilon)\neq 0$), we have 
\[\mathrm{length}_{\mathscr O_\chi}\left(\mathrm{Sel}_\varepsilon(K,A_g(\chi))\right)\les
 2\cdot\ord_\chi\left(v_{\ell}\left(\kappa_{\bar\chi}^\varepsilon\right)\right)=
2\cdot \mathrm{length}_{\mathscr O_\chi}\left(\mathfrak{Sel}^{\varepsilon}(K,\mathbf{T}_{f})\otimes\mathscr{O}_\chi/\mathscr O_\chi\cdot \kappa_{\bar\chi}^\varepsilon\right)\] with equality in the non-exceptional case. 
 
\emph{Step 2.} 
Recall the relaxed Selmer group $\mathrm{Sel}_\varepsilon^{(\ell)}(K,A_{f,k}(\chi))\supseteq\mathrm{Sel}_\varepsilon(K,A_{f,k}(\chi))$, 
i.e. 
the set of cohomology classes defined requiring the 
same conditions as $\mathrm{Sel}_\varepsilon(K,A_{f,k}(\chi))$ at primes different from $\ell$, and no condition at $\ell$.  
We claim that 
\begin{equation}\label{relaxed} 
\mathrm{Sel}_\varepsilon(K,A_{f,k}(\chi))=\mathrm{Sel}_\varepsilon^{(\ell)}(K,A_{f,k}(\chi)).\end{equation}
To prove this, let $x\in\mathrm{Sel}_\varepsilon^{(\ell)}(K,A_{f,k}(\chi))$. We have to show that 
$x$ is in the kernel of the residue map at $\ell$. 
By global class field theory, using the orthogonality of $\mathrm{res}_v(x)$ and $\mathrm{res}_v(\xi_{\bar\chi,k}^\varepsilon)$ outside 
$\ell$ as in Step 1 of the proof of Theorem \ref{BSD0}, one then obtains 
\[
       0=\sum_{v}\dia{\mathrm{res}_{v}(x),\mathrm{res}_{v}\big(\xi_{\bar\chi,k}^\varepsilon)\big)}_{v}
       =\dia{\mathrm{res}_{\ell}(x),\mathrm{res}_{\ell}\big(\xi_{\bar\chi,k}^\varepsilon)\big)}_{\ell}
       =\dia{\partial_{\ell}(x),v_{\ell}\big(\xi_{\bar\chi,k}^\varepsilon)\big)}_{\ell}.
\]
Since $v_{\ell}\big(\bar\xi_{\bar\chi}^\varepsilon)\big)\not=0$,
and since $\dia{-,-}_{\ell}$ is a perfect pairing, this implies that 
$\partial_{\ell}(x)=0$, as was to be shown. 

\emph{Step 3}. We claim that there is an exact sequence: 
\[0\longrightarrow\mathrm{Sel}_\varepsilon(K,A_g(\chi))\longrightarrow
\mathrm{Sel}_\varepsilon(K,A_{f,k}(\chi))\overset{v_\ell}\longrightarrow H^1_\mathrm{fin}(K_\ell,A_{f,k}(\chi))\longrightarrow 0.\]
To show this, first note that 
\[\mathrm{Sel}_\varepsilon(K,A_g(\chi))\subseteq \mathrm{Sel}_\varepsilon^{(\ell)}(K,A_g(\chi))=
\mathrm{Sel}_\varepsilon^{(\ell)}(K,A_{f,k}(\chi))=\mathrm{Sel}_\varepsilon(K,A_{f,k}(\chi))\]
where the last equality follows from Step 2; this shows the exactness on the left. 
By definition, the kernel of the map 
$v_\ell:\mathrm{Sel}_\varepsilon(K,A_{f,k}(\chi))\rightarrow H^1_\mathrm{fin}(K_\ell,A_{f,k}(\chi))$
is $H^1_\mathrm{ord}(K_\ell,A_{f,k}(\chi))$, 
proving the exactness in the middle. Finally, $v_\ell$ is surjective because, 
under the isomorphism $T_{f,k}(\chi)\simeq A_{f,k}(\chi)$, $\xi_{\bar\chi}^\varepsilon$ is a class in $\mathrm{Sel}_\varepsilon(K,A_{f,k}(\chi))$ which satisfies 
$v_{\ell}\big(\bar\xi_{\bar\chi}^\varepsilon)\big)\not=0$. 

\emph{Step 4.} 
From Step 3 we obtain the equality 
\[\mathrm{length}_{\mathscr{O}_\chi}\left(\mathrm{Sel}_\varepsilon(K,A_g(\chi))\right)=\mathrm{length}_{\mathscr{O}_\chi}\left(\mathrm{Sel}_\varepsilon(K,A_{f,k}(\chi))\right)-\mathrm{length}_{\mathscr{O}_\chi}\left(\mathscr{O}_\chi/p^k\mathscr{O}_\chi\right)\] 
and combinig with Step 1 we get 
\[\mathrm{length}_{\mathscr O_\chi}\left(\mathrm{Sel}_\varepsilon(K,A_{f,k}(\chi))\right)-\mathrm{length}_{\mathscr{O}_\chi}\left(\mathscr{O}_\chi/p^k\mathscr{O}_\chi\right)\les
2\cdot \mathrm{length}_{\mathscr O_\chi}\left( \mathfrak{Sel}^{\varepsilon}(K,\mathbf{T}_{f}\otimes\mathscr{O}_\chi)
/\mathscr O_\chi\cdot \kappa_{\bar\chi}^\varepsilon\right)\] with equality in the non-exceptional case. Since the left hand side has finite order, bounded independently of $k$, 
we see that the $\mathscr{O}_\chi$-corank of $\mathrm{Sel}_\varepsilon(K,A_{f}(\chi))$ 
is $1$.  
By the choice of $k$, 
\[\mathrm{length}_{\mathscr O_\chi}\left(\mathrm{Sel}_\varepsilon(K,A_{f,k}(\chi))\right)-\ord_\chi\left(\mathscr{O}_\chi/p^k\mathscr{O}_\chi\right)=
\mathrm{length}_{\mathscr O_\chi}\left(\mathrm{Sel}_\varepsilon(K,A_{f}(\chi))_{/\mathrm{div}}\right),\] 
 concluding the proof. 
\end{proof} 

\subsection{Proof of Theorem A in the indefinite case}\label{def:indefinite_p-adic_L}    Let $\mathfrak X_p^\varepsilon(f)$ be 
the Pontryagin dual of $\mathrm{Sel}_\varepsilon(K,\mathbf{A}_f)$. 
Then the compact $\Lambda$-module $\mathfrak X_p^\varepsilon(f)$ is pseudo-isomorphic to 
$\Lambda\oplus \mathfrak M\oplus \mathfrak M$ for a torsion $\Lambda$-module $\mathfrak M$, supported only on primes of height $1$; this follows from Theorem \ref{BSD-Indef}, 
the structure results in Step 5 of the proof of Theorem \ref{BSD0}, 
and Proposition  \ref{global control}. 
Let $\mathrm{Char}_p^\varepsilon(f)$  be the characteristic ideal of 
the $\Lambda$-module $\mathfrak M$. 

By Theorem \ref{BSD-Indef}, $\mathfrak{Sel}^\varepsilon(K,\mathbf{T}_f)$ is a free $\Lambda$-module of rank $1$ and 
$\mathfrak{Sel}^\varepsilon(K,\mathbf{T}_f)/\Lambda\cdot\kappa_\infty^\varepsilon$ is a torsion $\Lambda$-module. We may then denote 
 $L_p^\varepsilon(f)$ the characteristic power series of 
$\mathfrak{Sel}^\varepsilon(K,\mathbf{T}_f)/\Lambda\cdot\kappa_\infty^\varepsilon$. 

\begin{theo}[IAMC]\label{Aindef}
$(L_p^\varepsilon(f))\subseteq(\mathrm{Char}_p^\varepsilon(f))$, with equality in the non-exceptional case. 
\end{theo}

\begin{proof}  As in the definite case, the proof combines Theorem \ref{BSD-Indef} with Proposition \ref{propMRH} and follows the argument in Mazur--Rubin \cite[Section 5.2]{MR-KolSys} and Howard \cite[Section 2.2]{How-HK}. \end{proof}.

\section{Proof of Theorems B and C}
\label{sec:proof_B_C}
Fix throughout this section a finite order character $\chi:G_\infty\twoheadrightarrow\mathscr{O}_\chi^\times$ of conductor $p^n$.

\subsection{Comparison of Selmer groups} 
Suppose that $p$ is supersingular. The aim of this section is to compare the discrete Selmer groups 
$\mathrm{Sel}^\varepsilon(K,A_f(\chi))$ and $\mathrm{Sel}(K,A_f(\chi))$ and the compact Selmer groups 
 $\mathfrak{Sel}^\varepsilon(K,T_f(\chi))$ and $\mathfrak{Sel}(K,T_f(\chi))$. 
Let $\mathfrak{p}\mid p$ be a prime and fix $k\in \N\cup\{\infty\}$. Set as before  
$\Psi=K_\mathfrak{p}$, $\Psi_n=K_{n,\mathfrak{p}}$ 
and $\Psi_\infty=K_{\infty,\mathfrak{p}}$. Let $\mathfrak{P}_\chi=(\mathfrak{p}_\chi)$ be the kernel of the character $\chi:\Lambda_{\mathscr{O}_\chi}\twoheadrightarrow\mathscr{O}_\chi$ obtained from $\chi$, 
where $\mathfrak{p}_\chi=\gamma-\chi(\gamma)$. We also view $\chi$ as a character $\chi:\mathscr{O}_\chi[G_n]\twoheadrightarrow\mathscr{O}_\chi$, whose 
kernel we still denote by $\mathfrak{P}_\chi=(\mathfrak{p}_\chi)$. 
To simplify the notation, define  
\[\mathbf{E}_{\mathscr{O}_\chi}(\Psi_n)_{\mathrm{div}}=\mathbf{E}(\Psi_n)\otimes_\Z(\mathscr{K}_\chi/\mathscr{O}_\chi),\]
\[\mathbf{E}_{\mathscr{O}_\chi}(\Psi_n)_{\pm,\mathrm{div}}=\mathbf{E}(\Psi_n)_\pm\otimes_\Z(\mathscr{K}_\chi/\mathscr{O}_\chi).\]

\begin{lemma}\label{lemma comparison} Let $\varepsilon=(-1)^n$.
\begin{enumerate}
\item In the non-exceptional case, $\frac{\left(\mathbf{E}_{\mathscr{O}_\chi}(\Psi_n)_{\mathrm{div}}\right)[\mathfrak{P}_\chi]}{\left(\mathbf{E}_{\mathscr{O}_\chi}(\Psi_n)_{\varepsilon,\mathrm{div}}\right)[\mathfrak{P}_\chi]}$ is finite and 
\[\mathrm{length}_{\mathscr{O}_\chi}\left(\frac{\left(\mathbf{E}_{\mathscr{O}_\chi}(\Psi_n)_{\mathrm{div}}\right)[\mathfrak{P}_\chi]}{\left(\mathbf{E}_{\mathscr{O}_\chi}(\Psi_n)_{\varepsilon,\mathrm{div}}\right)[\mathfrak{P}_\chi]}\right)=[\Psi:\Q_p]\cdot \ord_\chi(\tilde\omega_n^{-\varepsilon}).\]
\item In the exceptional case: 
\begin{enumerate}
\item if $n=0$, so $\chi=\mathbf{1}$ is the trivial character, then  
\[\frac{\left(\mathbf{E}_{\Z_p}(\Psi)_{\mathrm{div}}\right)[(\gamma-1)]}{\left(\mathbf{E}_{\Z_p}(\Psi)_{+,\mathrm{div}}\right)[(\gamma-1)}=
\mathbf{E}(\Psi)\otimes_{\Z_p}\Q_p/\Z_p\cong (\Q_p/\Z_p)^{[\Psi:\Q_p]};\]
\item if $n\ges2$, then $\frac{\left(\mathbf{E}_{\mathscr{O}_\chi}(\Psi_n)_{\mathrm{div}}\right)[\mathfrak{P}_\chi]}{\left(\mathbf{E}_{\mathscr{O}_\chi}(\Psi_n)_{+,\mathrm{div}}\right)[\mathfrak{P}_\chi]}$ is finite and 
\[\mathrm{length}\left(\frac{\left(\mathbf{E}_{\mathscr{O}_\chi}(\Psi_n)_{\mathrm{div}}\right)[\mathfrak{P}_\chi]}{\left(\mathbf{E}_{\mathscr{O}_\chi}(\Psi_n)_{+,\mathrm{div}}\right)[\mathfrak{P}_\chi]}\right)=[\Psi:\Q_p]\cdot\ord_\chi(\omega_n^{-}).\] \end{enumerate}
\end{enumerate}
\end{lemma}

\begin{proof}  
Suppose first $n=0$. Then $\chi$ is the trivial character, $\mathscr{O}_\chi=\Z_p$ and we suppress the index $\mathscr{O}_\chi$ from the notation, thus writing
$\mathbf{E}(\Psi)_{\mathrm{div}}$ for $\mathbf{E}_{\Z_p}(\Psi)_{\mathrm{div}}$ 
and $\mathbf{E}(\Psi)_{\pm,\mathrm{div}}$ for $\mathbf{E}_{\Z_p}(\Psi)_{\pm,\mathrm{div}}$. 
If $p$ is split, then
$\mathbf{E}(\Psi)_{\mathrm{div}}=\mathbf{E}(\Psi)_{+,\mathrm{div}}$, so the quotient in the statement is trivial; on the other hand, $\tilde\omega_0^{-}=1$, and the statement is proved. 
If $p$ is inert, 
$\mathbf{E}(\Psi)_{\mathrm{div}}=\mathbf{E}(\Psi)_{-,\mathrm{div}}$ 
and $\mathbf{E}_{\mathscr{O}_\chi}(\Psi)_{+,\mathrm{div}}=0$, so the quotient in the statement is $\mathbf{E}(\Psi)_{\mathrm{div}}[(\gamma-1)]\cong(\Q_p/\Z_p)^{[\Psi:\Q_p]}$, where the last isomorphism follows from Theorem \ref{prop:formal-group}.  

Suppose $n\ges1$. 
We first observe that we have an exact sequence: 
\begin{equation}\label{eqsec11}
\xymatrix{
0\ar[r]&C\ar[r]& 
\left(\mathbf{E}_{\mathscr{O}_\chi}(\Psi_n)_{\varepsilon,\mathrm{div}}\right)[\mathfrak{P}_\chi]\oplus \left(\mathbf{E}_{\mathscr{O}_\chi}(\Psi_n)_{-\varepsilon,\mathrm{div}}\right)[\mathfrak{P}_\chi]\ar[r]& 
\left(\mathbf{E}_{\mathscr{O}_\chi}(\Psi_n)_\mathrm{div}\right)[\mathfrak{P}_\chi]\ar[r]&0}  
\end{equation}
where $C=0$ if $p$ is inert in $K$ and $C=\left(\mathbf{E}_{\mathscr{O}_\chi}(\Psi)_\mathrm{div}\right)[\mathfrak{P}_\chi]$ if $p$ is split in $K$; in this exact sequence the second arrow is the map $x\mapsto (x,x)$, and the third arrow is the map $(x,y)\mapsto x-y$. 
If $p$ is inert in $K$, it follows from Theorem \ref{prop:formal-group} that 
$\mathbf{E}_{\mathscr{O}_\chi}(\Psi_n)_\mathrm{div}$ is the direct sum of 
$\mathbf{E}_{\mathscr{O}_\chi}(\Psi_n)_{\varepsilon,\mathrm{div}}$ and 
$\mathbf{E}_{\mathscr{O}_\chi}(\Psi_n)_{-\varepsilon,\mathrm{div}}$, which proves \eqref{eqsec11} (also in the exceptional case). In the split case, it follows again 
from Theorem \ref{prop:formal-group} that 
\[\mathbf{E}_{\mathscr{O}_\chi}(\Psi_n)_{\varepsilon,\mathrm{div}}\cap 
\mathbf{E}_{\mathscr{O}_\chi}(\Psi_n)_{-\varepsilon,\mathrm{div}}=
\mathbf{E}_{\mathscr{O}_\chi}(\Psi)_\mathrm{div},\] 
so we need to show that the exact sequence 
\[\xymatrix{
0\ar[r]&\mathbf{E}_{\mathscr{O}_\chi}(\Psi)_\mathrm{div}\ar[r]& 
\mathbf{E}_{\mathscr{O}_\chi}(\Psi_n)_{\varepsilon,\mathrm{div}}\oplus\mathbf{E}_{\mathscr{O}_\chi}(\Psi_n)_{-\varepsilon,\mathrm{div}}\ar[r]& 
\mathbf{E}_{\mathscr{O}_\chi}(\Psi_n)_\mathrm{div}\ar[r]&0}  
\] remains exact after taking $\mathfrak{P}_\chi$-torsion, so we need to show that the map 
\[\left(\mathbf{E}_{\mathscr{O}_\chi}(\Psi_n)_{\varepsilon,\mathrm{div}}\right)[\mathfrak{P}_\chi]\oplus \left(\mathbf{E}_{\mathscr{O}_\chi}(\Psi_n)_{-\varepsilon,\mathrm{div}}\right)[\mathfrak{P}_\chi]\longrightarrow
\left(\mathbf{E}_{\mathscr{O}_\chi}(\Psi_n)_\mathrm{div}\right)[\mathfrak{P}_\chi]\] is surjective. 
The cokernel of this map injects into the quotient $\mathbf{E}_{\mathscr{O}_\chi}(\Psi)_\mathrm{div}/\mathfrak{p}_\chi\mathbf{E}_{\mathscr{O}_\chi}(\Psi)_\mathrm{div}$, and we need to show that this group is trivial.  
Since $\mathbf{E}_{\mathscr{O}_\chi}(\Psi)$ is $\mathscr{O}_\chi$-free, it is enough to show that \begin{equation}\label{t1}
\left(\mathbf{E}_{\mathscr{O}_\chi}(\Psi)/\mathfrak{p}_\chi\mathbf{E}_{\mathscr{O}_\chi}(\Psi)\right)\otimes_{\mathscr{O}_\chi}\mathscr{K}_\chi/\mathscr{O}_\chi=0.\end{equation} 
Now, $\mathfrak{p}_\chi$ acts on the $\mathscr{O}_\chi$-free module 
$\mathbf{E}_{\mathscr{O}_\chi}(\Psi)$ as multiplication by $1-\chi(\gamma)$, and since 
$\chi(\gamma)$ is a primitive $p^n$-root of unity, and $n\ges 1$, the quotient $\mathbf{E}_{\mathscr{O}_\chi}(\Psi)/\mathfrak{p}_\chi\mathbf{E}_{\mathscr{O}_\chi}(\Psi)$ is finite, and \eqref{t1} follows.  

We have therefore a commutative diagram with exact rows 
\[\xymatrix{
0\ar[r]&0\ar[r]\ar[d]& 
\left(\mathbf{E}_{\mathscr{O}_\chi}(\Psi_n)_{\varepsilon,\mathrm{div}}\right)[\mathfrak{P}_\chi]\ar@{=}[r]\ar@{^(->}[d]& 
{\left(\mathbf{E}_{\mathscr{O}_\chi}(\Psi_n)_{\varepsilon,\mathrm{div}}\right)[\mathfrak{P}_\chi]}
\ar[r]\ar@{^(->}[d]&0\\
0\ar[r]&C\ar[r]& 
\left(\mathbf{E}_{\mathscr{O}_\chi}(\Psi_n)_{\varepsilon,\mathrm{div}}\right)[\mathfrak{P}_\chi]\oplus \left(\mathbf{E}_{\mathscr{O}_\chi}(\Psi_n)_{-\varepsilon,\mathrm{div}}\right)[\mathfrak{P}_\chi]\ar[r]& 
\left(\mathbf{E}_{\mathscr{O}_\chi}(\Psi_n)_\mathrm{div}\right)[\mathfrak{P}_\chi]\ar[r]&0
}
\]  where $C$ is defined before, and the middle vertical arrow is the map $x\mapsto (x,0)$. By the snake lemma we obtain an exact sequence 
\[0\longrightarrow C\longrightarrow \left(\mathbf{E}_{\mathscr{O}_\chi}(\Psi_n)_{-\varepsilon,\mathrm{div}}\right)[\mathfrak{P}_\chi]\longrightarrow \frac{\left(\mathbf{E}_{\mathscr{O}_\chi}(\Psi_n)_{\mathrm{div}}\right)[\mathfrak{P}_\chi]}{\left(\mathbf{E}_{\mathscr{O}_\chi}(\Psi_n)_{\varepsilon,\mathrm{div}}\right)[\mathfrak{P}_\chi]}\longrightarrow 0.\]
The Pontryagin dual of the middle term is $\Lambda_{\mathscr{O}_\chi}/(\omega_n^{-\varepsilon},\mathfrak{p}_\chi)$, because the Pontryagin dual of 
 $\mathbf{E}_{\mathscr{O}_\chi}(\Psi_n)_{-\varepsilon,\mathrm{div}}$ is $\Lambda_{\mathscr{O}_\chi}/(\omega_n^{-\varepsilon})$; since 
$\Lambda_{\mathscr{O}_\chi}/\mathfrak{P}_\chi\cong\mathscr{O}_\chi$, 
the length of the middle term is equal to the length of  
$\mathscr{O}_\chi/\chi(\omega_n^{-\varepsilon})$. 
Similarly, if $p$ is split, the length of $C=\left(\mathbf{E}_{\mathscr{O}_\chi}(\Psi)_\mathrm{div}\right)[\mathfrak{P}_\chi]$ is equal to the length of
$\mathscr{O}_\chi/\chi(\gamma-1)$ and therefore the length of the quotient is equal to the length of $\mathscr{O}_\chi/\chi(\tilde\omega_n^{-\varepsilon})$, 
completing the proof in this case. If $p$ is inert, then $C$ is trivial. If $\varepsilon=-1$ (the non-exceptional case), then $\tilde\omega_n^+=\omega_n^+$, 
and the length of the last term is equal to the length of  
$\mathscr{O}_\chi/\chi(\omega_n^{+})=\mathscr{O}_\chi/\chi(\tilde\omega_n^{+})$, 
while if $\varepsilon=+1$ (the exceptional case) then the length is 
$\mathscr{O}_\chi/\chi(\omega_n^{-})$, completing the proof. 
\end{proof}

\begin{proposition}\label{prop comparison} 
In the exceptional case, assume that $n\neq 0$.   
The discrete Selmer groups $\mathrm{Sel}_\varepsilon(K,A_f(\chi))$ and 
$\mathrm{Sel}(K,A_f(\chi))$ have the same $\mathscr{O}_\chi$-corank. Moreover, 
\begin{enumerate}
\item If $p$ is split in $K$ or $p$ is inert in $K$ and $\varepsilon=-1$ (the non-exceptional case), 
\[\mathrm{length}_{\mathscr{O}_\chi}\left(\frac{\mathrm{Sel} (K,A_f(\chi))}{\mathrm{Sel}_\varepsilon(K,A_f(\chi))}\right)=
2\cdot\ord_\chi(\tilde\omega_n^{-\varepsilon}).\]
\item If $p$ is inert in $K$ and $\varepsilon=+1$ (the exceptional case),  
\[\mathrm{length}_{\mathscr{O}_\chi}\left(\frac{\mathrm{Sel} (K,A_f(\chi))}{\mathrm{Sel}_+(K,A_f(\chi))}\right)=
2\cdot\ord_\chi(\omega_n^{-}).\]
\end{enumerate}
\end{proposition}

\begin{proof}
We have the Poitou--Tate exact sequence
\[\begin{split}
0\longrightarrow
\mathrm{Sel}_\varepsilon(K,A_f(\chi))\longrightarrow{\mathrm{Sel}} (K,A_f(\chi))\longrightarrow 
\prod_{\mathfrak{p}\mid p}\frac{\left(\mathbf{E}_{\mathscr{O}_\chi}(\Psi)_{\mathrm{div}}\right)[\mathfrak{P}_\chi]}{\left(\mathbf{E}_{\mathscr{O}_\chi}(\Psi_n)_{\varepsilon,\mathrm{div}}\right)[\mathfrak{P}_\chi]}\longrightarrow\\
\longrightarrow\left(\mathfrak{Sel}^\varepsilon(K,T_f(\bar\chi))\right)^\vee\longrightarrow 
\left(\mathfrak{Sel} (K,T_f(\bar\chi))\right)^\vee\longrightarrow 0.\end{split}
\] 
Now $\mathfrak{Sel}^\varepsilon(K,T_f(\bar\chi))$ is $\mathscr{O}_\chi$-free, and 
therefore $\left(\mathfrak{Sel}^\varepsilon(K,T_f(\bar\chi))\right)^\vee$ is $p$-divisible. If we show that the kernel of the map $\left(\mathfrak{Sel}^\varepsilon(K,T_f(\bar\chi))\right)^\vee\rightarrow \left(\mathfrak{Sel}(K,T_f(\bar\chi))\right)^\vee$ 
is divisible, then, since the local quotient in the middle of the exact sequence above is finite by Lemma \ref{lemma comparison}, 
we have an exact sequence 
\[
0\longrightarrow
\mathrm{Sel}_\varepsilon(K,A_f(\chi))\longrightarrow{\mathrm{Sel}} (K,A_f(\chi))\longrightarrow 
\prod_{\mathfrak{p}\mid p}\frac{\left(\mathbf{E}_{\mathscr{O}_\chi}(\Psi)_{\mathrm{div}}\right)[\mathfrak{P}_\chi]}{\left(\mathbf{E}_{\mathscr{O}_\chi}(\Psi_n)_{\varepsilon,\mathrm{div}}\right)[\mathfrak{P}_\chi]}
\longrightarrow0\]
and the result follows from Lemma \ref{lemma comparison}. So to complete the proof we need to show that the kernel of the (surjective) map 
 \[\xymatrix{
 \left(\mathfrak{Sel}^\varepsilon(K,T_f(\bar\chi))\right)^\vee\ar@{->>}[r]& \left(\mathfrak{Sel}(K,T_f(\bar\chi))\right)^\vee}
 \] 
is divisible. For this, it is enough to show that the cokernel of the (injective) map 
\begin{equation}\label{cok}
\xymatrix{
\mathfrak{Sel}(K,T_f(\bar\chi))\ar@{^(->}[r]&  \mathfrak{Sel}^\varepsilon(K,T_f(\bar\chi))}\end{equation}
is torsion free. Let $x\in \mathfrak{Sel}^\varepsilon(K,T_f(\bar\chi))$ and let $M\ges 1$ be such that $\varpi_\chi^M\cdot x\in \mathfrak{Sel}(K,T_f(\bar\chi))$; to conclude that 
the cokernel of the map \eqref{cok} is torsion-free, it is then enough to show that $x\in\mathfrak{Sel}(K,T_f(\bar\chi))$. Since 
$\varpi_\chi^M\cdot x\in  \mathfrak{Sel}(K,T_f(\bar\chi))$, we  
have (writing $\langle-,-\rangle$ for the local Tate pairing $\langle-,-\rangle_{\mathfrak{P}_\chi,\mathfrak{p}}$ to simplify the notation) 
$\langle \mathrm{res}_\mathfrak{p}(\varpi_\chi^M\cdot x),y\rangle=0$ for all $y\in H^1_{\mathrm{fin}}(K_\mathfrak{p},A_f(\chi))$, and since 
 $\langle \mathrm{res}_\mathfrak{p}(\varpi_\chi^M\cdot x),y\rangle=
 \langle \mathrm{res}_\mathfrak{p}(x),\varpi_\chi^M \cdot y\rangle$, we also have 
 $ \langle \mathrm{res}_\mathfrak{p}(x),\varpi_\chi^M \cdot y\rangle=0$ for all 
 $y\in H^1_{\mathrm{fin}}(K_\mathfrak{p},A_f(\chi))$. 
 Recall that 
 $H^1_{\mathrm{fin}}(K_\mathfrak{p},A_f(\chi))$ is co-free over $\mathscr{O}_\chi$ 
 by Proposition \ref{freeloc}, hence $H^1_{\mathrm{fin}}(K_\mathfrak{p},A_f(\chi))$ is $\varpi_\chi$-divisible.
So the function $y\mapsto \langle \mathrm{res}_\mathfrak{p}(x),y\rangle$ is zero 
on  $H^1_{\mathrm{fin}}(K_\mathfrak{p},A_f(\chi))$ and therefore $x$ belongs to 
$\mathfrak{Sel}(K,T_f(\bar\chi))$, concluding the proof.   
\end{proof}

\subsection{Proof of Theorem B} \label{sec:proof_B}
Recall $L_{p,n}(f)=\mathcal{L}_{f,n}\cdot \left(\mathcal{L}_{f,n}\right)^\iota\in\Z_p[G_n]$.
By Remark \ref{remark:n=0} we may
assume $n\neq 0$ in the exceptional case. 

\emph{Step 1.} We first show that 
\[\mathrm{length}_{\mathscr{O}_{\chi}}\big(\mathrm{Sel} (K,A_{f}(\chi))\big)
            \les\ord_\chi\big(\chi(L_{p,n}(f))\big)\] with equality in the non-exceptional case.                          
Take $g=f_k$ for $L=\emptyset$ in Theorem \ref{BSD0}. 
By Theorem \ref{BSD0}, $\mathrm{Sel}_{\varepsilon}(K,A_{f,k}(\chi))$ is 
finite, of order bounded independently of $k$, so 
$\mathrm{Sel}_{\varepsilon}(K,A_{f}(\chi))$ is finite, and by Proposition \ref{prop comparison}
the Selmer group $\mathrm{Sel} (K,A_f(\chi))$ is finite.  
Let $t_{\bar\chi}^\varepsilon(f)=\ord_{\bar\chi}\big(\bar\chi(\mathcal L_f^\varepsilon)\big)$ and choose 
\[k>\max\left\{\mathrm{length}_{\mathscr{O}_{\chi}}\big(\mathrm{Sel}_{\varepsilon}(K,A_{f}(\chi))\big), t_{\bar\chi}^\varepsilon(f),\ord_\chi(\tilde\omega_n^{-\varepsilon})\right\}.\] 
For such a $k$, we have $\mathrm{Sel}_{\varepsilon}(K,A_{f}(\chi))\cong 
\mathrm{Sel}_{\varepsilon}(K,A_{f,k}(\chi))$, 
and, by Proposition \ref{prop comparison},
\begin{itemize}
\item $\mathrm{length}_{\mathscr{O}_\chi}\left(\mathrm{Sel} (K,A_f(\chi))\right)=
\mathrm{length}_{\mathscr{O}_\chi}\left(\mathrm{Sel}_{\varepsilon}(K,A_{f,k}(\chi))\right)+2\cdot\ord_\chi(\tilde\omega_n^{-\varepsilon})$ in the non-exceptional case; 
\item $\mathrm{length}_{\mathscr{O}_\chi}\left(\mathrm{Sel} (K,A_f(\chi))\right)=
\mathrm{length}_{\mathscr{O}_\chi}\left(\mathrm{Sel}_{\varepsilon}(K,A_{f,k}(\chi))\right)+2\cdot\ord_\chi(\omega_n^{-\varepsilon})$ in the exceptional case. 
\end{itemize}  
By \cite[Proposition 2.6]{Be-Da1},
$(\mathcal{L}_{f}^\varepsilon)^\iota=\pm\gamma_{\infty}\mathcal{L}_{f}^\varepsilon$,
for some $\gamma_{\infty}\in{}G_{\infty}$, and therefore 
$t_{\chi}^\varepsilon(f)=t_{\bar{\chi}}^\varepsilon(f)$.
We thus have 
$\ord_\chi\big(\chi(L_p^\varepsilon(f))\big)
=2\cdot t_{\bar{\chi}}^\varepsilon(f)$.
If $p$ is split in $K$ or $p$ is inert in $K$ and $\varepsilon=-1$ (non-exceptional case) we have 
$\tilde\omega_n^{-\varepsilon}\cdot \mathcal{L}_{f,n}^\varepsilon\equiv \pm\mathcal{L}_{f,n}$ modulo $\omega_{n}$, and therefore  
\[\begin{split}
\mathrm{length}_{\mathscr{O}_\chi}\left(\mathrm{Sel} (K,A_f(\chi))\right)&=
\mathrm{length}_{\mathscr{O}_\chi}\left(\mathrm{Sel}_{\varepsilon}(K,A_{f,k}(\chi))\right)+2\cdot\ord_\chi(\tilde\omega_n^{-\varepsilon})\\
& =2\cdot\ord_\chi\left(\mathcal{L}_{f,n}^\varepsilon(\bar\chi)\right)+2\cdot\ord_\chi(\tilde\omega_n^{-\varepsilon})\\
&=2\cdot\ord_\chi\left(\mathcal{L}_{f,n}(\bar\chi)\right),
\end{split}\]
where the second equality follows from Theorem \ref{BSD0}. 
If $p$ is inert in $K$ and $\varepsilon=+1$ (exceptional case), we have $\omega_n^{-}\cdot \mathcal{L}_{f,n}^+\equiv \pm\mathcal{L}_{f,n}$ modulo $\omega_{n}$, and therefore again by Theorem \ref{BSD0}
\[\begin{split}
\mathrm{length}_{\mathscr{O}_\chi}\left(\mathrm{Sel} (K,A_f(\chi))\right)&=
\mathrm{length}_{\mathscr{O}_\chi}\left(\mathrm{Sel}_{\varepsilon}(K,A_{f,k}(\chi))\right)+2\cdot\ord_\chi(\omega_n^{-}) \\
& \les{}2\cdot\ord_\chi\left(\mathcal{L}_{f,n}^+(\bar\chi)\right)+2\cdot\ord_\chi(\omega_n^{-})\\
&\les{}2\cdot\ord_\chi\left(\mathcal{L}_{f,n}(\bar\chi)\right).
\end{split}\]

\emph{Step 2.} We now use explicit formulas for special values to conclude the proof of Theorem B. 
Recall that Gross's formula gives the equality 
\[\frac{L(E/K,\chi,1)}{\Omega}C_\chi=\chi(L_p(f))\]
where $C_\chi=u^2{\sqrt{|D_K|}p^n}$ with $u=\sharp(\mathcal{O}_{p^n}^\times)/2$, and $\Omega$ is Gross's period, defined in  \cite[Lemma 2.5]{Vat-mu}. 
We refer to \cite[Proposition 7.7]{Gross}, \cite[\S2.3]{Vat-mu} and \cite[Theorem 1.2]{CST} for this result; in particular, in the notation of \cite{CST}, 
$\Omega=\frac{8\pi^2\langle f, f\rangle_{\Gamma_0(N)}}{\langle \phi,\phi\rangle}$, where $\langle f, f\rangle_{\Gamma_0(N)}$ denotes Petersson inner product, $\phi$ is the ($p$-adically normalized) 
Jacquet--Langland lift of $f$ to the definite quaternion algebra of discriminant $N^-$ which we use to define $L_p(f)$, and $\langle\phi,\phi\rangle$ is the height pairing defined in loc. cit. ; see also \cite[Theorem 3.11]{Ch-Hs}. From Theorem A and Gross's formula we see that 
$\mathrm{Sel}(K,A_f(\chi))$ is finite if and only if  
$L(E/K,\chi,1)\neq0$ and 
\[\mathrm{length}_{\mathscr{O}_\chi}\left(\mathrm{Sel}_\mathrm {}(K,A_f(\chi)\right)\leqslant\ord_\chi\left(\frac{L(E/K,\chi,1)\cdot C_\chi}{\Omega}\right)\] 
with equality in the non-exceptional case, which gives Theorem B for $C=\Omega/C_\chi$. 

\subsection{Proof of Theorem C} \label{sec:proof_C}
As noted in Remark \ref{remark:n=0} we may assume that $n\neq 0$ in the exceptional case. We define the \emph{regulator} 
\[\mathrm{Reg}_\chi(E/K)=\frac{h_\mathrm{NT}(P_{\bar\chi})}{ 2\cdot\mathrm{length}_{\mathscr{O}_\chi}
\left(\mathfrak{Sel}(K,T_f(\chi))/\mathscr{O}_\chi \cdot\tilde\kappa_{\bar\chi}\right)}\]
and the \emph{Shafarevich--Tate group} 
\[\sha(K,A_f(\chi))=\mathrm{Sel}(K,A_f(\chi))_{/\mathrm{div}}.\]

\emph{Step 1.}  We first show that 
\[\mathrm{length}_{\mathscr{O}_\chi}\left(\mathrm{Sel}_\mathrm {}(K,A_f(\chi))_{/\mathrm{div}}\right)\les\mathrm{length}_{\mathscr{O}_\chi}
\left(\mathfrak{Sel}(K,T_f(\chi))/\mathscr{O}_\chi \cdot \tilde\kappa_{\bar\chi})\right)\]
with equality in the non-exceptional case.  
Combining Theorem \ref{BSD-Indef} and Proposition \ref{prop comparison}, 
if $p$ is split in $K$ or $p$ is inert in $K$ and $\varepsilon=-1$ (non-exceptional case) we have 
\begin{equation}\label{length0}
\mathrm{length}_{\mathscr{O}_\chi}\left(\mathrm{Sel}_\mathrm {}(K,A_f(\chi))_{/\mathrm{div}}\right)=
2\cdot\mathrm{length}_{\mathscr O_\chi}\left( \left(\mathfrak{Sel}^{\varepsilon}(K,\mathbf{T}_{f,\mathscr{O}_\chi})\otimes_{\Lambda_{\mathscr{O}_\chi}}\Lambda_{\mathscr{O}_\chi}/\mathfrak{P}_{\bar\chi}\right)
/\left(\mathscr O_\chi\cdot \kappa_{\bar\chi}^\varepsilon\right)\right)+2\cdot\ord_\chi(\tilde\omega_n^{-\varepsilon}).\end{equation}
From \eqref{length0}, 
using Proposition \ref{speinj}  and Proposition \ref{prop comparison}(2), we obtain 
\begin{equation}\label{length0-bis}
\mathrm{length}_{\mathscr{O}_\chi}\left(\mathrm{Sel}_\mathrm {}(K,A_f(\chi))_{/\mathrm{div}}\right)=
2\cdot\mathrm{length}_{\mathscr O_\chi}
\left(
\mathfrak{Sel}(K,{T}_{f}(\chi))
/\left(\mathscr O_\chi\cdot \kappa_{\bar\chi}^\varepsilon\right)
\right)
+2\cdot\ord_\chi(\tilde\omega_n^{-\varepsilon}).\end{equation}
Now 
$\tilde\kappa_{\bar\chi}=\chi(\tilde\omega_n^{-\varepsilon})\cdot\kappa_{\bar\chi}^\varepsilon$, and the result follows. If $p$ is inert in $K$ and $\varepsilon=+1$ (the exceptional case), combining Theorem \ref{BSD-Indef} and Proposition \ref{prop comparison}, 
we have 
\begin{equation}\label{length0}
\mathrm{length}_{\mathscr{O}_\chi}\left(\mathrm{Sel}_\mathrm {}(K,A_f(\chi))_{/\mathrm{div}}\right)\les
2\cdot\mathrm{length}_{\mathscr O_\chi}\left( \left(\mathfrak{Sel}^{\varepsilon}(K,\mathbf{T}_{f,\mathscr{O}_\chi})\otimes_{\Lambda_{\mathscr{O}_\chi}}\Lambda_{\mathscr{O}_\chi}/\mathfrak{P}_{\bar\chi}\right)
/\left(\mathscr O_\chi\cdot \kappa_{\bar\chi}^+\right)\right)+2\cdot\ord_\chi(\omega_n^{-}).\end{equation}
From \eqref{length0}, 
using Proposition \ref{speinj}  and Proposition \ref{prop comparison}(2), we obtain 
\begin{equation}\label{length0-bis}
\mathrm{length}_{\mathscr{O}_\chi}\left(\mathrm{Sel}_\mathrm {}(K,A_f(\chi))_{/\mathrm{div}}\right)\leqslant
2\cdot\mathrm{length}_{\mathscr O_\chi}
\left(
\mathfrak{Sel}(K,{T}_{f}(\chi))
/\left(\mathscr O_\chi\cdot \kappa_{\bar\chi}^+\right)
\right)
+2\cdot\ord_\chi(\omega_n^{-}).\end{equation}
Now
$\tilde\kappa_{\bar\chi}=\chi(\omega_n^{-\varepsilon})\cdot\kappa_{\bar\chi}^\varepsilon$, and the result follows. 
%\begin{equation}
%\label{eqn:boh}
%\mathrm{length}_{\mathscr{O}_\chi}\left(\sha(K,A_f(\chi)\right)\les\mathrm{length}_{\mathscr{O}_\chi}
%\left(E(K_n)\otimes_\Lambda\mathscr{O}_\chi/\mathscr{O}_\chi P_\chi\right)
%\end{equation}
%with equality in the non-exceptional case. 

%\section{Proof of Theorem F -- Indefinite BSD formulas} 
\emph{Step 2.} We now use Gross-Zagier formulas to conclude the proof of Theorem C. 
By the Gross-Zagier formula we know that 
\[ {L^\prime(E/K,\chi,1)} = \frac{8\pi ^2\langle f,f\rangle_{\Gamma_0(N)}}{\deg(\phi)C_\chi}\cdot h_\mathrm{NT}(P_{\bar\chi})\]
where $\deg(\phi)$ is the degree of a modular parametrization $\phi:X_0(N)\rightarrow E$, and, with the same notation introduced in \S\ref{sec:proof_B},
$C_\chi=u^2{\sqrt{|D_K|}p^n}$ with $u=\sharp(\mathcal{O}_{p^n}^\times)/2$ and $\langle f,f\rangle_{\Gamma_0(N)}$ denotes the Petersson inner product. 
We refer to  \cite[Theorem 6.3]{Gross-Zagier} and \cite[Theorem 1.1]{CST} for the result in this form; see also 
\cite[Theorem 1.3.1]{YZZ}, \cite[Theorem 1.2.1]{SZhang-As}, \cite[Theorem 8.1]{Zh-2001}.
%, where the constant $C_\chi$ is described explicitly as $C_\chi=\frac{8\pi ^2\langle f,f\rangle_{\Gamma_0(N)}}{\deg(\phi)u^2{\sqrt{|D_K|}p^n}}$ with $u=\sharp(\mathcal{O}_{p^n}^\times)/2$ as before, 
%and $\deg(\phi)$ is the degree of a modular parametrization $X_0(N)\rightarrow E$. 
%See also  
%\cite[Theorem 1.3.1]{YZZ}, \cite[Theorem 1.2.1]{SZhang-As}, \cite[Theorem 8.1]{Zh-2001} for related result.}
%
%$C=\frac{4(\phi^\sharp,\phi^\sharp)^2}{\sqrt{|D_K|}}$ 
%\[C=\begin{cases}
%\frac{8\pi^2(f,f)}{h_Ku_K\sqrt{|D_K|}}, \text{ if } N^-=1\\
%\frac{4(\phi^\sharp,\phi^\sharp)^2}{\sqrt{|D_K|}}, \text{ if } N^-\neq1;
%\end{cases}\] 
%and, in the notation of \cite{SZhang-As}, 
%$\phi^\sharp$ is the quasi newform associated with $f$  
%and $(\phi^\sharp,\phi^\sharp)$ is the $L^2$-norm of $\phi^\sharp$ with respect to the Haar measure 
%normalised as in loc. cit. ; 
%see also \cite{YZZ}, Theorem 1.3.1, and \cite{Zh-2001}, Theorem 8.1 and 
%the discussion that follows. 
Combing Step 1 with the definitions of $\mathrm{Reg}_\chi(E/K)$ and $\sha(K,A_f(\chi))$ 
introduced above and the Gross--Zagier formula we obtain
\[\begin{split}
\mathrm{length}_{\mathscr{O}_\chi}(\sha(K,A_f(\chi)))&=
\mathrm{length}_{\mathscr{O}_\chi}\left(\mathrm{Sel}_\mathrm {}(K,A_f(\chi))_{/\mathrm{div}}\right)\\
&\les 2\cdot\mathrm{length}_{\mathscr{O}_\chi}
\left(\mathfrak{Sel}(K,T_f(\chi))/\mathscr{O}_\chi \cdot \tilde\kappa_{\bar\chi})\right)\\&\les
\ord_\chi\left(\frac{L^\prime(E/K,\chi,1)\deg(\phi)C_\chi}{8\pi^2\langle f,f\rangle_{\Gamma_0(N)}\cdot \mathrm{Reg}_\chi(E/K)}\right)\end{split} \]
with equality in the non-exceptional case, which gives Theorem C for $C=8\pi^2\langle f,f\rangle_{\Gamma_0(N)}/(\deg(\phi)C_\chi)$.

\section{Statements and declarations}

\begin{itemize}

\item[$\cdot$] On behalf of all authors, the corresponding author states that there is no conflict of interest. 

\item[$\cdot$] Data sharing not applicable to this article as no datasets were generated or analysed during the current study.

\end{itemize}

\bibliography{myref1}{}
\bibliographystyle{alpha}

\end{document}